\documentclass[11pt,reqno]{amsproc}

\title[Compressible Navier-Stokes equations with non-homegeneous pressure]{Compressible Navier-Stokes equations \\ with heterogeneous pressure laws}

\author[D.~Bresch]{Didier Bresch}
\address{Univ. Grenoble Alpes, Univ. Savoie Mont Blanc, CNRS, LAMA, 73000 Chamb\'ery, France}
\email{didier.bresch@univ-smb.fr}

\author[P.~Jabin]{Pierre--Emmanuel Jabin}
\address{Department of Mathematics, University of Maryland, College Park, MD 20740, USA}
\email{pjabin@cscamm.umd.edu}

\author[F.~Wang]{Fei Wang}
\address{Department of Mathematics, University of Maryland, College Park, MD 20740, USA}
\email{fwang256@umd.edu}

\usepackage{fancyhdr}

\usepackage[margin=1.25in]{geometry}
\usepackage{amsmath, amsthm, amssymb}
\usepackage{graphicx}

\usepackage[usenames,dvipsnames,svgnames,table]{xcolor}
\usepackage[colorlinks=true, pdfstartview=FitV, linkcolor=blue, citecolor=blue, urlcolor=blue]{hyperref}

\begin{document}

\def\epsilon{\varepsilon}
\def\intint{\int\!\!\!\!\int}
\def\OO{\mathcal O}
\def\SS{\mathbb S}
\def\RR{\mathbb R}
\def\TT{\mathbb T}
\def\ZZ{\mathbb Z}
\def\HH{\mathbb H}
\def\RSZ{\mathcal R}
\def\GG{\mathcal G}
\def\eps{\varepsilon}
\def\tt{\langle t\rangle}
\def\erf{\mathrm{Erf}}
\def\red#1{\textcolor{red}{#1}}
\def\blue#1{\textcolor{blue}{#1}}
\def\mgt#1{\textcolor{magenta}{#1}}
\def\ff{\rho}
\def\gg{\gamma_{art}}
\def\ggg#1{\gamma_{art,#1}}
\def\phiyy{\partial_{yy} \phi}
\def\tilde{\widetilde}
\def\sqrtnu{\sqrt{\nu}}
\def\ww{w}
\def\ft#1{#1_\xi}
\def\cl{\mathcal{L}}
\def\ck{\mathcal{K}}
\def\ce{\mathcal{E}}
\def\fei#1{\textcolor{Magenta}{#1}}
\def\nn{\nonumber\\}
\def\sp{\ \ \ \ \ \ \ \ \ \ \ \ \ \ \ \ \ \ \ \ \ \ \ \ \ \ \ \ \ \ \ \ \ \ \ \ \ \ \ \ \ \ \ }
\def\ss{s_{0}}
\def\TH{\Theta_{2}}
\def\TTT{\Theta_{3}}
\def\bep{C\lambda^{-1}}
\def\yy{\tilde y}
\def\dq{\qquad}
\def\DD{\mathcal{D}}

\newtheorem{theorem}{Theorem}[section]
\newtheorem{corollary}[theorem]{Corollary}
\newtheorem{proposition}[theorem]{Proposition}
\newtheorem{lemma}[theorem]{Lemma}

\theoremstyle{definition}
\newtheorem{definition}{Definition}[section]
\newtheorem{remark}[theorem]{Remark}

\def\theequation{\thesection.\arabic{equation}}
\numberwithin{equation}{section}

%Igor's macros 
\def\dist{\mathop{\rm dist}\nolimits}    %distance
\def\sgn{\mathop{\rm sgn\,}\nolimits}    %sgn
\def\Tr{\mathop{\rm Tr}\nolimits}    %trace
\def\div{\mathop{\rm div}\nolimits}    %divergence
\def\supp{\mathop{\rm supp}\nolimits}    %divergence
 \def\indeq{\qquad{}\!\!\!\!}                     %indentation in formulas
\def\period{.}                           %period in a formula
\def\semicolon{\,;}                      %semicolon in a formula

\def\nts#1{{\cor #1\cob}}
\def\cor{\color{red}}
\def\cog{\color{green}}
\def\cob{\color{black}}
\def\coe{\color{blue}}
%%%%%%%

%**end of header

\begin{abstract}
  This paper concerns the  existence of global  weak solutions {\it \`a la Leray}  for compressible Navier-Stokes equations  with a  pressure law which depends on the density and on  time and space variables $t$ and $x$. The assumptions on the pressure contain only locally Lipschitz assumption with respect to the density variable and some hypothesis with respect to the  extra time and space variables. It may be seen as a first step to consider heat-conducting Navier-Stokes equations with physical laws such as the truncated virial assumption. The paper focuses on the construction of approximate solutions through a new regularized and fixed point procedure and on the weak stability process  taking advantage of the new method introduced by the two first authors with a careful study of an appropriate regularized quantity linked to  the pressure.
%\\
%\today.
\end{abstract}

%\subjclass[2000]{35Q35, 35Q30, 76D05}
%\keywords{}
\maketitle
%%%%%%%%%%%%%%%%%%%%%%%%%%%%%%%%%%%%%%%%%%%%%%
%%%%%%%%%%%%%%%%%%%%%%%%%%%%%%%%%%%%%%%%%%%%
\section{Introduction and main result}
\smallskip

As mentioned in \cite{BrJaUnderpressure}, the existence of global weak solutions, in the sense of J. Leray,  to the non-stationary  barotropic compressible Navier-Stokes system with constant shear and bulk viscosities $\mu$ and $\lambda$ remained a longstanding open problem in space dimension strictly greater than one until the first results by P.--L.~Lions (see \cite{Li}) with $P(\rho)= a\rho^\gamma$
($\gamma>3d/(d+2)$). Many important contributions followed to improve the result including E. Feireisl--A. Novotny--H. Petzeltova ($\gamma>d/2$, see \cite{FeNoPe}), P.I. Plotnikov--V.A. Weigant ($\gamma=d/2$, see \cite{PlWe}),  E.~Feireisl (pressure law $s\mapsto P(s)$ non-monotone on a compact set, see \cite{Fe0}) and more recently D. Bresch-P.–E. Jabin (thermodynamically unstable pressure law $s\mapsto P(s)$ or anisotropic viscosities, see \cite{BreJab18}).  

\smallskip

   One of the main issue is that the weak bound of the divergence of the velocity field does not {\it a priori} rule out singular behaviors by the density which may oscillate, concentrate or even vanish (vacuum state) even if this is not the case initially.  
 
 \smallskip
  
   Heat-conducting viscous compressible Navier-Stokes equations (Navier-Stokes-Fourier) with constant viscosities namely with a pressure law  $(\rho,\vartheta) \mapsto P(\rho,\vartheta)$ and an extra equation on the temperature $\vartheta$  has been firstly discussed in \cite{Li} and solved  by E. Feireisl and A. Novotny for specific pressure laws, see \cite{Febook}  and \cite{FeNo} which in some sense are monotone with respect to the density after a fixed value. In the present paper, we prepare the resolution of the heat-conducting compressible Navier-Stokes equations with a truncated virial pressure law 
   \begin{equation} 
P(\rho,\vartheta)=\rho^{\gamma}+\vartheta\,\sum_{n=0}^{[\gamma/2]} B_n(\vartheta)\,\rho^n. \label{virial}
  \end{equation}
    Such pressure law is not monotone with respect to the density after a fixed value and therefore is not thermodynamically stable. This paper concerns the existence of global weak solutions \`a la Leray for compressible Navier-Stokes equations with a pressure law which depends on the density and on time and space variables $t$ and $x$.  It may be seen as a first step to consider heat-conducting Navier-Stokes equations with physical laws such as the truncated virial assumption. More precisely, we consider the compressible Navier-Stokes (CNS) equations
\begin{align} 
&\partial_{t}\rho+\div(\rho u) =0 \label{eq:00}\\
&\partial_t (\rho u) + \div(\rho u\otimes u)  - \mu \Delta u - (\mu+\lambda ) \nabla {\rm div} u 
+ \nabla P = 0 \label{eq:01}
\end{align}with initial condition 
  \begin{equation}
  \label{eq:ini}
  \rho\vert_{t=0} =\rho_0 \ \ \ \ \ (\rho u)\vert_{t=0} =m_0,
  \end{equation}
in a periodic box $\Omega = \TT^d$ for $d\ge2$ and $\mu$ and $\lambda$ two constants satisfying the physical constraint $\mu>0$ and $\lambda +2\mu/d> 0$.  The pressure $P=P(t, x, \rho)$ is  a given  function depending on the time $t$, space $x$, and the density $\rho$. For simplicity in the redaction we consider in the sequel that the shear viscosity $\mu=1$ and the bulk viscosity $\lambda = -1$: This
does not changed the mathematical proof and result.

    For simplicity, we consider the periodic boundary conditions in $x$, namely $\Omega = \TT^d$, even if arguments can be adapted to the whole space case as well. As explained previously, the article should be seen as a first step to solve the truncated virial case where we assume that the temperature $\vartheta(t,x)$ is actually given instead of solving the temperature equation
  \begin{equation}
\partial_t (\rho\,E)+\div_x(\rho\,E\,u)+\div_x(P(\rho,\vartheta)\,u)=\div_x (\nabla_xu\cdot u)+\div_x(\kappa(\vartheta)\,\nabla \vartheta),\label{fourier}
  \end{equation}
  where $E=|u|^2/2+e(\rho,\vartheta)$ is the total energy density with $e(\rho,\vartheta)$ is the specific internal energy and initial condition
\begin{equation}\label{initemp}
\rho E\vert_{t=0} = \rho_0 E_0.
\end{equation}  
 with the virial pressure state law \eqref{virial}. The main result presented here will be used in our upcoming article (see \cite{BrJaWaCNS}) to construct solutions to the full system \eqref{eq:00}--\eqref{eq:ini} and \eqref{fourier}--\eqref{initemp} as it provides the starting point for the fixed point procedure that we adopt.  
If $\vartheta$ is given then naturally $P(t,x,\rho)=P(\rho,\vartheta(t,x))$. But  there are however several other contexts (for instance in biology) where it is necessary to involve non spatially homogeneous pressure law and for this reason, it is useful to consider more general formulas for $P$ than given by \eqref{virial}.  Note that as shown in \cite{BrJaMe}, the procedure developped here is also applicable for
the compressible Brinkman system (semi-stationary compressible Stokes system) which is standard system that may be seen in porous media and biology.

 \medskip
 
    The construction of appropriate approximate solutions  will be a difficulty in our paper. It is based on an original approximate system for which existence of solutions is obtained through a regularization and  a fixed point approach. The weak stability  property on the sequence of approximate solution is obtained using the new method introduced by the two first authors in \cite{BreJab18} and taking care of the regularized term linked to the pressure state law which involves serious difficulties.

\medskip

\noindent We assume {\bf hypothesis on the pressure law $(t,x,s) \mapsto P(t,x,s)$}:
Some of them are used to ensure the propagation of energy and the others are used to garantee the
propagation of compactness on the density.

\eject

\noindent More precisely, let us present:

\smallskip

\noindent --  Assumptions to ensure the {\it propagation of energy}.

\smallskip

\noindent Let $\gamma> {3d}/{(d+2)}$: 
  \begin{align} 
    &\!(P1)\quad  \mbox{There exist } q>2,\;0\leq\bar\gamma\leq \gamma/2, \mbox{ and a smooth function }P_0 \mbox{ such that} \nonumber\\
%    &\qquad\qquad\  \nonumber \\
        &\hskip1.2cm  |P(t, x, s) - P_0(t,x,s)| \le  C\,R(t,x) + C\, s^{\bar \gamma}\ \mbox{for } R\in L^q((0,T)\times\mathbb{T}^d).  \label{P:1} \\
      %
    %\begin{aligned}
      &\!(P2)\quad  \mbox{There exist } p<\gamma+\frac{2\,\gamma}{d}-1,\;q>2, \ \Theta_{1}(t,x)\in L^q((0,T)\times\mathbb{T}^d),\ \mbox{such that}\nonumber\\
      &\qquad\qquad C^{-1}\,s^\gamma-\Theta_{1}(t,x)\le P_0(t, x, s)\le Cs^p+\Theta_{1}(t,x).\label{P:2}\\
      %\end{aligned}\\
%    
    &(P3)\quad \label{P:3} {\mbox{There exist } p<\gamma+\frac{2\,\gamma}{d}-1,\;} \ \mbox{and } \TH\in L^{q}([0,\ T]\times \TT^d)\ \mbox{with}\ q>1\ \mbox{such that}\nonumber\\ 
    &\qquad\qquad |\partial_{t}P_0(t, x, s)|\le Cs^{p} + \TH(t,x).\\
    &(P4)\quad \label{P:4} |\nabla_xP_0(t, x, s)|\le C\,s^{\gamma/2} + \TTT(t,x),\ \mbox{for }\TTT\in L^2([0,\ T],\ L^{2d/(d+2)}(\TT^d)).
  \end{align}

\medskip

 \noindent --  Assumptions required for the  {\it propagation of compactness} 
 on the density.
  \begin{align}
    &\label{P:5}
    \begin{aligned}
       &\!(P5)\quad \hbox{ The pressure } P\mbox{ is locally Lipschitz in the sense of that}\\ 
     &\qquad\qquad\qquad |P(t, x,z)-P(t, y,w)|\le Q(t,x,y)+(C(z^{\gamma-1} +w^{\gamma-1})\\
   &\qquad\qquad\qquad\qquad\qquad + (\tilde P(t,x)+\tilde P(t,y))|z-w|,\\
       &\qquad \mbox{for some } \tilde P \in L^{\ss}([0, T], \mathbb{T}^d) \mbox{ and } Q \in L^{s_1}([0, T], \mathbb{T}^{2d}) \mbox{ for some }  \ss, s_1>1.\\
     \end{aligned}\\
    &\label{P:las}\begin{aligned}
 &\!(P6)\quad \mbox{The functions } Q,\; \tilde P \mbox{ satisfy that } \mbox{for some}\  r_h \rightarrow0,\ \text{as}\  h\rightarrow 0\\
&   \frac{1}{\Vert \ck_{h}\Vert_{L^1}}\int_{0}^T\int_{\TT^{2d}}  \ck_{h}(x-y) \left( |\tilde P(t,x)-\tilde P(t,y)|^{\ss} 
+ |Q(t,x,y)|^{s_1} \right)\,dx\,dy\,ds = r_h.\\
\end{aligned}
    \end{align}   

\medskip

\noindent {\it The total energy of the CNS system.}
The total energy of the system, which is the sum of the kinetic and the potential energies, reads
  \begin{equation*}
    \ce(t,x,\rho, \rho u) = \int_{\TT^{d}} \left(\frac{|\rho u|^2}{2\rho } + \rho e(t,x,\rho)\right)\,dx %+ \frac{\mu(\gg-1)}{\gg}\int_{\TT^{d}} \rho^{\gg}\,dx
  \end{equation*}
 where %$\gg$ satisfies $\gg\ge 5r_0/2$ with $r_0$ being the dual H\"older exponent of $\ss$,   and 
  \begin{equation}
  \label{e:pot}
  e(t,x,\rho) = \int_{\rho_{ref}}^\rho \frac{P(t, x, s)}{s^2}\,ds
  \end{equation}
  with $\rho_{ref}$ a constant reference density. We also define similarly the reduced total energy $\mathcal{E}_0 (t,x,\rho,\rho\,u)$ which is based on $P_0$ instead of $P$, see assumption~\eqref{P:1}.
 Note that we assume as usually
 \begin{equation} u_0 = \frac{m_0}{\rho_0} \hbox{ when } \rho_0 \not = 0 \hbox{ and } u_0=0 \hbox{ elsewhere,} 
 \end{equation}
 \begin{equation} \frac{|m_0|^2}{\rho_0} = 0 \hbox{ \it a.e.} \hbox{ on } \{x\in \Omega: \rho_0(x)=0\}.
 \end{equation}
  The following is our main result dealing with heterogeneous pressure laws. 
\begin{theorem}
\label{main}
Assume the initial data $m_0$ and $\rho_0\ge 0$ with $\int_{{\mathbb T}^d} \rho_0 = M_0>0$ satisfy
$$\ce(\rho_0, m_0) = \int_{\TT^{d}} \left(\frac{|m_0|^2}{2\rho_0 } + \rho_0 e(0,x,\rho_0)\right)\,dx <\infty.$$
Suppose that the pressure $P$ satisfies~\eqref{P:1}--\eqref{P:las}. Then there exists a global weak solution to Compressible Navier--Stokes System \eqref{eq:00}--\eqref{eq:ini} such that
$$ u \in L^2(0,T;H^1(\TT^d)), \qquad |m|^2/2\rho \in L^\infty(0,T;L^1(\TT^d))$$
$$ \rho \in {\mathcal C}([0,T], L^\gamma(\TT^d) \hbox{ weak }) \cap L^p((0,T)\times \TT^d)
    \hbox{ where } 0 < p < \gamma (d+2)/2 -1$$
with the heterogeneous pressure state law $P$ satisfying the energy inequality
\[\begin{split}
      &\int_{\TT^d} \mathcal{E}_0(\rho,u)\,dx+\int_0^t\int_{\TT^d} |\nabla u(s,x)|^2\,dx\,ds\leq {\mathcal E}(\rho_0,u_0)\\
      &\qquad+\,\int_0^t\int_{\TT^d} \div_x u(s,x)\,(P(s,x,\rho(s,x))-P_0(s,x,\rho(s,x)))\,ds\,dx\\
      &\qquad +\int_0^t\int_{\TT^d} (\rho\,\partial_t e_0(\rho)+\rho\,u\cdot\nabla_x e_0(\rho))\,dx\,ds
      \end{split}
\]
where 
$$\mathcal{E}_0(\rho,u) = |\rho u|^2/2\rho + \rho \int_{\rho_{ref}}^\rho \frac{P_0(t,x,s)}{s^2} \, ds.$$
\end{theorem}

\begin{remark} Note that $u\in L^2(0,T; H^1(\TT^d))$ comes from the control of the gradient of the velocity field $\nabla u$  in $L^2((0,T)\times \Omega)$ and the control of $|m|^2/\rho$ in $L^1((0,T)\times \TT^d)$ using the fact that $\int_\Omega \rho = \int_\Omega \rho_0 = M>0.$ The interested reader is referred to \cite{Li}.
\end{remark}

%
%%%%%%%%%%%%%%%%%%%%%%%%%%%%%%%%%%%%%%%%%%%%%
%%%%%%%%%%%%%%%%%%%%%%%%%%%%%%%%%%%%%%%%%%%%
\section{The approximation systems with a sketch of proof and  {\it a priori } estimates}
%%%%%%%%%%%%%%%%%%%%%%%%%%%%%%%%%%%%%%%%%%%%
We present here the approximate system upon which we rely to construct the solution to \eqref{eq:00}--\eqref{eq:ini} with the pressure law $P$ given by \eqref{P:1}--\eqref{P:las}.  As is classical in compressible Fluid Mechanics, the approximation procedure is performed through several stages, involving different approximate systems.

%%%%%%%%%%%%%%%%%%%%%%%%%%%%%%%%%%%%%%%%%%%%%%%%%%%%%%%%%%%%%%%%%%%%%%%%%%%%
\subsection{The approximate system with artificial and delocalized pressures}
%%%%%%%%%%%%%%%%%%%%%%%%%%%%%%%%%%%%%%%%%%%%%%%%%%%%%%%%%%%%%%%%%%%%%%%%%%%%
One of the main difficulty is to find a proper approximation of the above system so that we may construct a solution of it and prove the compactness of the solutions. We propose to define the approximating system
\begin{align} 
&\partial_{t}\rho_{\epsilon,\eta}+\div(\rho_{\epsilon,\eta} u_{\epsilon,\eta}) =0 \label{app:00}\\
&\partial_t (\rho_{\epsilon,\eta} u_{\epsilon,\eta}) + \div(\rho_{\epsilon,\eta} u_{\epsilon,\eta}\otimes u_{\epsilon}) 
 - \Delta u_{\epsilon,\eta}+ \nabla( P_{art,\eta}(\rho_{\epsilon,\eta}) + \cl_{\epsilon}*P) = 0 \label{app:01}
\end{align}
with initial condition 
\begin{equation} \label{app:0ini}
\rho_{\epsilon,\eta}\vert_{t=0} =\rho_{0,\epsilon,\eta} \hbox{ and }
(\rho_{\epsilon,\eta} u_{\epsilon,\eta})\vert_{t=0}=m_{0,\epsilon,\eta}
\end{equation}
where an artificial pressure term reads
\[
P_{art,\eta}(\rho_{\epsilon,\eta})= \eta_1\, \rho_{\epsilon,\eta}^{\gamma_{art,1}}+\ldots+\eta_m\, \rho_{\epsilon,\eta}^{\gamma_{art,m}}
\]
for some fixed parameters $\gamma_{art}=\gamma_{art,1}>\gamma_{art,2}>\dots>\gamma_{art,m}$. The coefficients $\eta_1,\ldots,\eta_m$ will later be let to converge to $0$ in that order and the $\gamma_{art,i}$ will be chosen so that
\[
\gamma_{art,1}>2\,\gamma,\quad \gamma_{art,i+1}+2\,\frac{\gamma_{art,i+1}}{d}-1>\gamma_{art,i},\quad \gamma+2\,\frac{\gamma}{d}-1>\gamma_{art,m}.
\]
In addition an appropriate regularization of the 
pressure state law $\cl_{\epsilon}*(P(t,\cdot,\rho_{\epsilon,\eta}(t,\cdot))$ has been introduced. 
 More precisely  the key step is to construct a suitable mollifying operator $\cl_{\epsilon}$ defined as follows
 \begin{equation*}
   \cl_{\epsilon} (x)= \frac{1}{\log 2}\int_\epsilon^{2\epsilon} L_{\epsilon'}(x)\, \frac{d\epsilon'}{\epsilon'}.
  \end{equation*}
 where $L_{\epsilon}$ is a standard mollifier  given by
  \begin{equation*}
     L_{\epsilon}(x) = \frac{1}{\epsilon^d}L\left(\frac{x}{\epsilon}\right),
  \end{equation*}
with $L$ is a non-negative smooth function such that $L\in C^\infty_0(\TT^d)$ and 
  $\int_{\TT^d} L(x)\,dx =1.$ Then $L_{\epsilon}\rightarrow \delta_0$ as $\epsilon\rightarrow0$, with $\delta_0$ being the Dirac Delta function at $0$. 
It is straightforward to check that 
  \begin{equation*}
     \int_{\TT^d} \cl_{\epsilon}(x)\,dx =1
  \end{equation*}
  and
  \begin{equation*}
     \cl_{\epsilon}\rightarrow \delta_0,\ \ \ \ \text{as}\ \ \ \epsilon\rightarrow0.
  \end{equation*}
We observe that we easily have the following global existence result through a fixed point argument that will be presented in the Appendix
for readers convenience
\begin{theorem}
  Assume that $P$ satisfies \eqref{P:1} with $\gg>\gamma$ and that the initial data $\rho_{0,\eps,\eta},\;u_{0,\eps,\eta}$ satisfy the uniform bound
  \[
\sup_{\eps,\eta} \int_{\TT^d} (\eta_1\, (\rho_{0,\eps,\eta}(x))^{\ggg1}+\ldots+\eta_m\, (\rho_{0,\eps,\eta}(x))^{\ggg{m}}+ \rho_{0,\eps,\eta}(x)\,|u_{0,\epsilon\eta}(x)|^2)\,dx<\infty.
  \]
   There exist $\rho_{\eps,\eta} \in L^\infty([0,\ T],\ L^{\gg}(\TT^d))\cap L^p([0,\ T]\times \TT^d)$ for any $p<\gg+2\,\gg/d-1$, $u_{\eps,\eta}\in L^2([0,\ T],\ H^1(\TT^d))$ solution to \eqref{app:00}-\eqref{app:01}. Moreover, $\rho_{\eps,\eta},\;u_{\eps,\eta}$ satisfy the uniform in $\eps$ bounds
\begin{subequations}
\label{pri:app:1}
  \begin{align}
    &\sup_{\eps} \sup_{t\in [0,\ T]} \int_{\TT^d} (\eta_1\,\rho_{\eps, \eta}^{\ggg{1}}(t,x)+\ldots+\eta_m\,\rho_{\eps,\eta}^{\ggg{m}}(t,x) +\rho_{\eps,\eta} (t,x)\,|u_{\eps,\eta} (t,x)|^2)\,dx<\infty, \\
    &\sup_{\eps} \int_0^T\int_{\TT^d} |\nabla u_{\eps,\eta}|^2\,dx\,dt<\infty,\\
    &\sup_{\eps} \int_0^T\int_{\TT^d} \eta_1\,\rho_{\eps,\eta}^p(t,x)\,dx\,dt<\infty\ \mbox{for any}\ p<\gg+2\,\gg/d-1.
\end{align}
\end{subequations}
Finally, we have the explicit energy inequality
\begin{equation}
\begin{split}
  & \int_{\TT^d} \left(\eta_1\,\frac{\rho_{\eps,\eta}^{\ggg{1}}(t,x)}{\ggg{1}-1}+\ldots +\eta_m\,\frac{\rho_{\eps,\eta}^{\ggg{m}}(t,x)}{\ggg{m}-1} +\rho_{\eps,\eta} (t,x)\,|u_{\eps,\eta} (t,x)|^2\right)\,dx\\
  &\qquad+\int_0^t\int_{\TT^d} |\nabla u_{\eps,\eta}(s,x)|^2\,dx\,ds \leq \int_0^t\int_{\TT^d} \div u_{\eps,\eta}\,\mathcal{L}_\eps\star_x P\, dx\,ds\\
  &\qquad +\int_{\TT^d} \left(\eta_1\,\frac{(\rho_{\eps,\eta}^0)^{\ggg{1}}(t,x)}{\ggg{1}-1}+\ldots +\eta_m\,\frac{(\rho_{\eps,\eta}^0)^{\ggg{m}}(t,x)}{\ggg{m}-1} +\rho_{\eps,\eta}^0 (t,x)\,|u_{\eps,\eta}^0 (t,x)|^2\right)\,dx. 
  \end{split}\label{energyineq0}
\end{equation}
\label{thapprox1}
\end{theorem}
%\fei{only P1?}
The main difficulty and contribution of the present article is the limit  passage $\varepsilon \to 0$, with $\eta$ fixed, given by the following result
%\begin{theorem}
%Assume that $P$ satisfies \eqref{P:1} with $\gg>\gamma$, and that the initial data $\rho^0_\eps,\;u^0_\eps$ satisfy that $\rho^0_\eps\to \rho^0$ in $L^{\gg}(\Pi^d)$, $\rho_\eps^0\,u_\eps^0\to \rho^0\,u^0$ and $\rho_\eps^0\,|u_\eps^0|^2\to |\rho^0\,u^0|^2$ in $L^1(\Pi^d)$. TO BE COMPLETED... \label{thepstomu}
%  \end{theorem}

\begin{theorem}
Assume that $P$ satisfies \eqref{P:5} and \eqref{P:las}. Let $\gg>\max(2s_0^*, s_1^*, 2+d)$, where $s_0^*$ and $s_1^*$ are the H\"older conjugate exponents of $\ss$ and $s_1$ respectively. Suppose that the initial data $\rho^0_\eps,\;u^0_\eps$ of the 
system~\eqref{app:00}--~\eqref{app:01} satisfy that $\rho_{0,\eps,\eta} \to \rho_{0,\eta}$ in $L^{\gg}(\TT^d)$, $\rho_{0,\epsilon,\eta}\,u_{0,\epsilon,\eta}\to \rho_{0,\eta}\,u_{0,\eta}$ and $\rho_{0,\epsilon,\eta} \,|u_{0,\epsilon,\eta}|^2\to |\rho_{0,\eta}\,u_{0,\eta}|^2$ in $L^1(\TT^d)$. Let $(\rho_{\epsilon,\eta}, u_{\epsilon;\eta})$ be the corresponding sequence of solutions satisfying the energy estimate~\eqref{pri:app:1}. Then $\rho_{\epsilon,\eta}$ is compact in $L^{p}(\TT^d)$ for $1\le p<\gg$ as $\epsilon\rightarrow0$.
\label{thepstomu}
  \end{theorem}
The particular form of the mollifier operator ${\mathcal L}_\epsilon$ is strongly used for the compactness property on $\{\rho_{\epsilon,\eta}\}_\epsilon$ to have enough control of terms involving the pressure terms in the method introduced by the two first authors in \cite{BreJab18}.
  Using the previous Theorem, the limit passage provides a sequence of global weak solutions $(\rho_\eta,u_\eta)$ to the following system
\begin{align} 
&\partial_{t}\rho_\eta+\div(\rho_\eta\, u_\eta) =0 \label{app1:00}\\
&\partial_t (\rho_\eta u_\eta) + \div(\rho\, u_\eta\otimes u_\eta) 
- \Delta u_\eta + \nabla(P_{art,\eta} (\rho_\eta) + P(t,x,\rho_\eta)) = 0, \label{app1:01}
\end{align}
for some large $\gg\geq \gamma$ with initial boundary conditions
\begin{equation}\label{app1:ini}
\rho_\eta\vert_{t=0} = \rho_{0,\eta}, \qquad \rho_\eta u_\eta\vert_{t=0} = m_{0,\eta}.
\end{equation} 
  Fortunately once we obtain global weak solutions to \eqref{app1:00}-\eqref{app1:ini} then passing to the limit as $\eta_1\rightarrow0$, then $\eta_2\to0$ and up to $\eta_m\to 0$, to obtain global weak solutions to \eqref{eq:00}--\eqref{eq:ini} is in fact a straightforward consequence of \cite{BreJab18}. More precisely we have 
\begin{theorem}
  Assume that $P$ satisfies \eqref{P:1}--\eqref{P:las}.
  Consider any sequence $\rho_\eta\in L^\infty([0,\ T],\ L^{\gg}(\TT^d))$ with $\ggg{m}<\gamma+2\,\gamma/d-1$, $\ggg{i}<\gamma_{art,i+1}+2\,\gamma_{art,i+1}/d-1$ and $\gg> 2\,\gamma$, any sequence $\;u_\eta\in L^2([0,\ T],\ H^1(\TT^d))$ of solutions to \eqref{app1:00}-\eqref{app1:01} over $[0,\ T]$.
{Suppose moreover that $\rho_\eta^0\to \rho^0$ in $L^\gamma(\TT^d)$, $\rho_\eta^0\,u_\eta^0\to \rho^0\,u^0$ and} $\rho_\eta^0\,|u_\eta^0|^2\to \rho^0\,|u^0|^2$ both in $L^1(\TT^d)$. Assume finally that $\sup_\eta\sup_{t\in [0,\ T]}\int_{\TT^d} \rho_\eta\,|u_\eta|^2\,dx<\infty$. 
  
  Then $\rho_\eta$ is compact in $L^1_{t,x}$, $u_\eta$ is compact in $L^2_{t,x}$ and converge to a global solution to \eqref{eq:00}, \eqref{eq:01} with
\[\begin{split}
      &\int_{\TT^d} \mathcal{E}_0(\rho,u)\,dx+\int_0^t\int_{\TT^d} |\nabla u(s,x)|^2\,dx\,ds\leq {\mathcal E}(\rho_0,u_0)\\
      &\qquad+ \,\int_0^t\int_{\TT^d} \div_x u(s,x)\,(P(s,x,\rho(s,x))-P_0(s,x,\rho(s,x)))\,ds\,dx\\
      &\qquad +\int_0^t\int_{\TT^d} (\rho\,\partial_t e_0(\rho)+\rho\,u\cdot\nabla_x e_0(\rho))\,dx\,ds.
      \end{split}
\]\label{thmuto0}
  \end{theorem}
\noindent The proof of Th. \ref{thmuto0} will be discussed in the appendix of the article for reader's convenience. This will end the proof 
of the main theorem \ref{main}.

\medskip
 
 \noindent {\it Important remark.}
 It is important to note that the requirement for having several exponents $\ggg{i}$ in the artificial pressure $P_{art,\eta}$ appears from the constraints in the proofs of Theorems \ref{thapprox1}-\ref{thmuto0}. To recover the appropriate energy terms in Theorem \ref{thapprox1}, we need to treat the actual pressure $P$ as a source term. This is only possible if $\div u\, \mathcal{L}_\eps\star P$ is integrable uniformly in $\eps$ and, as $P\lesssim \rho^\gamma$, it forces that $\gg>2\,\gamma$.

 On the other hand, assuming that $\ggg1,\ldots,\ggg{i-1}=0$, to pass to the limit in the term $\eta_i\,\rho^{\ggg{i}}$ as $\eta_i\to 0$ but $\eta_{i+1}>0$, we again need to have $\rho^{\ggg{i}}$ integrable. From the gain of integrability detailed in the next subsection, this only appears possible if $\ggg{i}<\gamma_{art,i+1}+2\,\gamma_{art,i+1}/d-1$.
  If we had only one correction in $P_{art,\eta}$, {\em i.e.} $m=1$, then we would actually need both $\gg>2\,\gamma$ and $\gg<\gamma+2\,\gamma/d-1$, which is of course not possible if $d\geq 2$.
  The introduction of several exponents $\ggg{i}$ seems to be a fairly straightforward manner of resolving this issue.
%%%%%%%%%%%%%%%%%%%%%%%%%%%%%%%%%%%%%% 
\subsection{Basic energy estimates}
%%%%%%%%%%%%%%%%%%%%%%%%%%%%%%%%%%%%%%
As those are used several times, we collect here the basic energy estimates for the generic system
\begin{equation}
  \begin{split}
    &\partial_t \rho+\div(\rho\,u)=0,\\
    &\partial_t(\rho\,u)+\div(\rho\,u\otimes u)-\Delta u+\nabla (P_0(t,x,\rho)+S(t,x))=0.
    \end{split}\label{generic}
\end{equation}
There exist a well-known gain in integrability on $\rho$ from the momentum equation. For convenience later, we write it in a slightly more general form.
\begin{lemma}
  Assume that $\rho\in L^\infty([0,\ T],\ L^{\gamma_0}(\TT^d))$ solves \eqref{conservative} with a velocity field $u\in L^2([0,\ T],\ H^1(\TT^d))$ and source term $S\in L^1([0,\ T],\ L^{\gamma_0^*}(\TT^d)$ satisfies that {change $L^\infty$ to $W^{-1, \infty}$?}
  \[\begin{split}
  &\nabla_x P_0(t,x,\rho)\in L^1([0,\ T],\ W^{-1,p}(\TT^d))\\
  &\qquad+H^{-1}([0,\ T],\ L^{2pd/(2d+2p-pd)}(\TT^d))\cap W^{-1, \infty} ([0,\ T],\ L^{pd/(p+d)}(\TT^d)).
\end{split}
  \]
Then for any $0<\theta<\gamma_0/p^*$,
\[\begin{split}
&\int_0^T\int_{\TT^d} \rho^\theta(s,x)\,P_0(s,x,\rho(s,x))\,ds\,dx\\
&\quad\leq C_d\,\|\rho\|_{L^\infty_t L^{\gamma_0}_x}^\theta\,(1+\|u\|_{L^2_t H^1_x})\,\|\nabla_x P_0(\rho)\|_{L^1_t W^{-1,p}_x+H^{-1}_t L^{2pd/(2d+2p-pd)}_x\cap W^{-1, \infty} L^{pd/(p+d)}}\\
&\qquad+C_d\,\|\rho\|_{L^\infty_t L^{\gamma_0}_x}^\theta\,\|S\|_{L^1_t L^{\gamma_0^*}_x}.
\end{split}
\]\label{bogovski}
\end{lemma}
\begin{proof}
  We can rewrite the assumption simply as
  \[
\nabla_x (P_0(t,x,\rho)+S)=\div_x f+\partial_t g,
\]
where $f\in L^1([0,\ T],\ L^p(\TT^d))$ and $g\in L^2([0,\ T],\ L^{2pd/(2d+2p-pd)}(\TT^d))$, with in addition $g\in L^\infty ([0,\ T],\ L^{pd/(p+d)}(\TT^d))$. For a fixed exponent $\theta>0$ to be chosen later, we define $c_\theta=\int_{\TT^d} \rho^\theta(t,x)\,dx$ and $B(t,x)=-\nabla_x \,\Delta_x^{-1}\,(\rho^\theta-c_\theta)$. In the case of a bounded domain with a boundary instead of the torus, one has to be more careful and use the appropriate Bogovski operator (see \cite{Febook} for example).

The idea is then simply for multiply by $B$ and first notice that
\[
\begin{split}
  &\int_0^T \int_{\TT^d} B(s,x)\cdot \nabla_x (S+P_0(s,x,\rho))\,dx\,ds=\int_0^T \int_{\TT^d} (\rho^\theta(s,x)-c_\theta)\,(S+ P_0(s,x,\rho))\,dx\,ds\\
  &\qquad\geq -C+\int_0^T \int_{\TT^d} \rho^\theta(s,x)\,(S+P_0(s,x,\rho(s,x)))\,dx\,ds.
\end{split}
\]
The integral of $\rho^\theta\,S$ can be bounded immediately to yield the second in the right-hand side of the lemma. On the other hand
\[
\begin{split}
  &\int_0^T \int_{\TT^d} B(s,x)\cdot \nabla_x (S+P_0(s,x,\rho))\,dx\,ds\\
  &\quad =-\int_0^T \int_{\TT^d} \nabla_x B(s,x)\,:\,f(s,x)\,dx\,ds\\
  &\qquad -\int_0^T \int_{\TT^d} \partial_t B(s,x)\cdot g(s,x)\,dx\,ds+\int_{\TT^d} (B(0,x)\cdot g(0,x)-B(0,T)\cdot g(T,x))\,dx.\\  
\end{split}
\]
By standard Calderon-Zygmund theory, $\|\nabla_x B\|_{L^\infty_t L^{\gamma_0/\theta}_x}\leq C_d\,\|\rho\|_{L^\infty_t L^{\gamma_0}_x}^\theta$. Hence the first term in the r.h.s. is directly bounded by
\[
-\int_0^T \int_{\TT^d} \nabla_x B(s,x)\,:\,f(s,x)\,dx\,ds\leq \|f\|_{L^1_t L^p_x}\,\|\nabla B\|_{L^\infty_t L^{p^*}_x}\leq C_d\, \|f\|_{L^1_t L^p_x}\,\|\rho\|_{L^\infty_t L^{\gamma_0}_x}^\theta,
\]
since $p^*\leq \gamma_0/\theta$ as $\theta<\gamma_0/p^*$.
By Sobolev embedding $\|B\|_{L^\infty_t L^q_x}\leq C_d\,\|\rho\|_{L^\infty_t L^{\gamma_0}_x}^\theta$ with $1/q=\theta/\gamma_0-1/d$. Hence we have again that
\[
\int_{\TT^d} (B(0,x)\cdot g(0,x)-B(0,T)\cdot g(T,x))\,dx\leq C_d\,\|\rho\|_{L^\infty_t L^{\gamma_0}_x}^\theta\,\|g\|_{L^\infty_t L^{pd/(p+d)}_x},
\]
since $1-(p+d)/pd=1-1/d-1/p\geq \theta/\gamma_0-1/d$ by the same condition on $\theta$.
The second term in the r.h.s is handled by using the continuity equation \eqref{conservative} satisfied by $\rho$. Since $\gamma_0\geq 2$, $\rho$ is a renormalized solution to \eqref{conservative} by Th. \ref{ren:sol} and hence we have that
\[
\partial_t \rho^\theta+\div(\rho^\theta\,u)=(1-\theta)\,\rho^\theta\,\div u.
\]
We may replace
\[
\int_0^T\!\! \int_{\TT^d} \partial_t B(s,x)\cdot g(s,x)\,dx\,ds=\int_0^T\!\! \int_{\TT^d} \nabla_x\,\Delta^{-1}_x\,((1-\theta)\,\rho^\theta\,\div u-\div(\rho^\theta\,u)-\tilde c_\theta)\cdot g(s,x)\,dx\,ds,
\]
for some time dependent constant $\tilde c_\theta$. Using that $g\in L^2_t L^{2pd/(2d+2p-pd)}_x$, we bound in a similar manner all the terms and conclude that
\[
-\int_0^T \int_{\TT^d} \partial_t B(s,x)\cdot g(s,x)\,dx\,ds\leq C_d\,\|\rho\|_{L^\infty_t L^{\gamma_0}_x}^\theta\,\|u\|_{L^2_t\,H^1_x}\,\|g\|_{L^2_t\,L^{2p/(p+2)}_x}.
\]
\end{proof}
%%%%%%%%%%%%%%%%%%%%%%%%%%%%%%%%%%5
\section{Technical Preliminaries}
%%%%%%%%%%%%%%%%%%%%%%%%%%%%%%%%%%
We list here technical results and considerations, which were mostly developed in \cite{BreJab18} and upon which our proof relies.
%%%%%%%%%%%%%%%%%%%%%%%%%%%%%%%%%%%%%%%%
\subsection{Our compactness criterion}
%%%%%%%%%%%%%%%%%%%%%%%%%%%%%%%%%%%%%%
As is classical in compressible Fluid Mechanics, the main difficulty in obtaining existence is to prove the compactness of a sequence of approximations of the density $\rho_\eps$. As mentioned above, we follow here the general strategy of \cite{BreJab18}, and we hence rely on the following criterion.
\begin{lemma}
\label{cpt}
Let $\rho_{\epsilon}$ be a family of functions which are bounded in some $L^p([0, T]\times \TT^d)$ with $1\le p<\infty.$
Assume that $\ck_h$ is a family of positive bounded functions such that 
  \begin{itemize}
  \item $\sup_h \int_{|x|\ge \eta} \ck_h(x) \,dx < \infty$ for any $\eta>0$.
  \item $\Vert \ck_h\Vert_{L^{1}} \rightarrow \infty$ as $h\rightarrow0.$
  \end{itemize}
Assume that for some $q\ge1$
  \begin{align*}
     \sup_\epsilon\Vert \partial_{t}\rho_{\epsilon}\Vert_{L^q([0, T], W^{-1,1}(\TT^d))} <\infty
  \end{align*}
and 
  \begin{align*}
  \lim_{h\rightarrow0}\limsup_{\epsilon} \int_{0}^T\int_{\TT^{2d}}\frac{\ck_h(x-y)}{\Vert \ck_h\Vert_{L^1}} |\rho_{\epsilon}(x)-\rho_{\epsilon}(y)|^p\,dx\,dy\,ds=0.
  \end{align*}
   Then the family of functions $\rho_{\epsilon}$ is compact in $L^p([0, T]\times \TT^d)$. Conversely if $\rho_{\epsilon}$ is compact in $L^p([0, T]\times \TT^d)$, then the above limit is $0$.
\end{lemma}

The construction of a suitable kernel function $\ck_h$ for the system that we are considering again follows \cite{BreJab18}. We first define a bounded, positive, and symmetric function $\tilde K_h$
such that
  \begin{equation*}   
  \tilde K_h(x) = \frac{1}{(h+|x|)^{d+a}},\ \ \ \ \text{for}\ \ \ \ |x|\le \frac{1}{2}
  \end{equation*}
with some $a>0$ and $\tilde K_h$ independent of $h$ for $|x|\ge2/3$. We will also require that $\tilde K_h\in C^\infty(\TT^d \backslash B(0, 3/4))$ and that $\supp \tilde K_h\subset B(0, 1)$. Setting
  \begin{equation*}
     K_h = \frac{\tilde K_h}{\Vert \tilde K_h\Vert_{L^1(\TT^d)}}.
  \end{equation*}
we have immediately that
  \begin{equation*}
     \Vert K_h\Vert_{L^1(\TT^d)} =1
  \end{equation*}
and 
  \begin{equation}
  \label{x*der}
  |x||\nabla K_h(x)| \lesssim |K_h(x)|.
  \end{equation}
For our compactness argument, we use the operator
  \begin{equation}
  \label{kernel}
  \ck_{h_0} = \int_{h_0}^1 K_h(x) \frac{dh}{h}.
  \end{equation}
Note that 
  \begin{equation*}
     \Vert \ck_{h_0}\Vert_{L^1(\TT^d)} = c_0 |\log h_0|
  \end{equation*}
for some positive constant $c_0$.
With the above notation, one of our main steps is to show that 
  \begin{align*}   
  \limsup_{\epsilon} \int_{0}^T\int_{\TT^{2d}}\ck_{h_0}(x-y) |\rho_{\epsilon}(x)-\rho_{\epsilon}(y)|^p\,dx\,dy\,ds 
  \rightarrow 0
  \end{align*}
as $h_0\rightarrow0$, from where the compactness of the family $\rho_{\epsilon}$ follows.
%%%%%%%%%%%%%%%%%%%%%%%%%%%%%%%%%%
\subsection{Technical lemmas}
%%%%%%%%%%%%%%%%%%%%%%%%%%%%%%%%%%
As our main strategy is to control differences $\delta\rho_\eps$, which requires some specific lemmas. One may find proofs for these lemmas in~\cite{BreJab18}. Our basic way of estimating differences is through
\begin{lemma}
\label{dif:u}
Let $u\in W^{1, 1}$, we have
  \begin{align*}
     |u(x)-u(y)| 
  \lesssim (D_{|x-y|}u(x)+D_{|x-y|}u(y))|x-y|,
  \end{align*}
  where 
  \begin{equation*}
     D_{h}u(x)
  =
  \frac{1}{h}\int_{|z|\le h}\frac{|\nabla u(x+z)|}{|z|^{d-1}}\,dz.
  \end{equation*}
\end{lemma}
The next lemma provides a bound for the term $D_{h}u(x)$ in term of the maximal function.
\begin{lemma}
\label{D:M}
For any $u\in W^{1, p}$ with $p\ge1$, the following inequality 
  \begin{equation*}
     D_h\,u(x) \lesssim M|\nabla u|(x)
  \end{equation*}
holds.
\end{lemma}
\begin{remark}
 
By the above two lemmas we deduce immediately the classical inequality
  \begin{equation}
  \label{maximalineq}
   |u(x)-u(y)| 
  \lesssim (M\nabla u(x)+M\nabla u(y))|x-y|.
  \end{equation}
\end{remark}

In several critical places of the proof,  we need to estimate the difference $D_{|z|}u(x)-D_{|z|}u(x-z)$ while relying only on the $L^2$ regularity of $\nabla u$. Using classical harmonic analysis results, we can get the following.
\begin{lemma}
\label{squ:est}
Assume that $u\in H^1(\TT^d)$. Then for any $1<p<\infty$, one has
  \begin{align*}
     \int_{h_0}^1\int_{\TT^{d}}K_{h}(z) \Vert D_{|z|}u(x)-D_{|z|}u(x-z)\Vert_{L^p_x}\,dz\frac{dh}{h} \lesssim \Vert u\Vert_{B^1_{p, 1}}
  \end{align*}
as a result of which, we further have that
  \begin{align*}
     \int_{h_0}^1\int_{\TT^{d}}K_{h}(z) \Vert D_{|z|}u(x)-D_{|z|}u(x-z)\Vert_{L^2_x}\,dz\frac{dh}{h} \lesssim \Vert u\Vert_{H^1} |\log h_0|^{1/2}.
  \end{align*}
Moreover, the following estimate
  \begin{align*}
     \int_{h_0}^1\int_{\TT^{2d}}K_{h}(z)K_{h}(\xi) \Vert D_{|z|}u(x)-D_{|z|}u(x-\xi)\Vert_{L^2_x}\,dz\,d\xi\frac{dh}{h} \lesssim \Vert u\Vert_{H^1} |\log h_0|^{1/2}
  \end{align*}
holds.
\end{lemma}
In most instances, the above estimate is sufficient. But in several cases, we need the more general version,
\begin{lemma}
\label{squ:est:1}
Consider a family of kernels $N_r\in W^{s, 1}(\TT^d)$, where $s>0$, which satisfy
  \begin{itemize}
  \setlength\itemsep{.3em}
  \item $\sup_{|\xi|\le1}\sup_{r}r^{-s} \int_{\TT^{d}}|z|^s|N_r(z) -N_r(z-r\xi)|\,dz <\infty,$
  \item $\sup_r(\Vert N_r \Vert_{L^1}+r^s\Vert N_r \Vert_{W^{s,1}})<\infty$.
  \end{itemize}
Then the estimate
  \begin{align*}   
  \int_{h_0}^1\int_{\TT^{d}}K_{h}(z) \Vert N_h*u(x)-N_h*u(x-z)\Vert_{L^p_x}\,dz\frac{dh}{h} \lesssim \Vert u\Vert_{L^p} |\log h_0|^{1/2}
  \end{align*}
holds for any $u\in L^p$ with $1<p\le2$.
\end{lemma}
%%%%%%%%%%%%%%%%%%%%%%%%%%%%%%%%%%%%%%%%%%%%%%%%%%
\subsection{The choice of the weight function}
\label{sub:wei}
%%%%%%%%%%%%%%%%%%%%%%%%%%%%%%%%%%%%%%%%%%%%%%%%%%%
We now turn to the construction of an appropriate weight function tailored for the proof of Th. \ref{thepstomu}. First we define the function $w_{\epsilon}$ which satisfies the equations
  \begin{align}
  \label{eq:wei}
  &\partial_{t}w_\epsilon + u_{\epsilon}\cdot\nabla w_{\epsilon} = - D_{\epsilon}w_{\epsilon}\\
  &w_{\epsilon}(0) = 1 \label{eq:ini:wei}
  \end{align}
where $D_{\epsilon}$ is given by
  \begin{align}
  \label{D:eps}
  D_{\epsilon} = \lambda(M|\nabla u_{\epsilon}| + |\rho_{\epsilon}|^\gamma + K_h*(|\div u_{\epsilon}| 
  %+ \mu\rho_{\epsilon}^{\gg} 
  + |\cl_{\epsilon}*P|  +|\tilde P_{\epsilon}^{1+l}|) ).
  \end{align}
Denote 
  \begin{equation*}
     w_{\epsilon, h} = K_h*w_{\epsilon}.
  \end{equation*}
Then the weight function $W_{\epsilon, h}$ we use is given by
  \begin{equation*}
     W_{\epsilon, h}^{x, y} = w_{\epsilon, h}(x) + w_{\epsilon, h}(y)
  \end{equation*}
which could capture the feature that $W_{\epsilon, h}$ is big if either one of $w_{\epsilon, h}(x)$ and  $w_{\epsilon, h}(y)$ is big.
Since the function $W^{x, y}_{\epsilon} = w_{\epsilon}(x) + w_{\epsilon}(y)$ satisfies the following equation
  \begin{equation*}
     \partial_{t}W_{\epsilon} + u_{\epsilon}^{x}\cdot\nabla_x W_{\epsilon} + u_{\epsilon}^{y}\cdot\nabla_y W_{\epsilon}
  =
  -(D_{\epsilon}^{x}w_{\epsilon}(x)+ D_{\epsilon}^{y}w_{\epsilon}(y)),
  \end{equation*}
it follows that
  \begin{align}
  \label{wei:eq}
  \partial_{t}W_{\epsilon, h} + u_{\epsilon}^{x}\cdot\nabla_x W_{\epsilon, h} + u_{\epsilon}^{y}\cdot\nabla_y W_{\epsilon, h}
  =
  -D_{\epsilon, h}^{x, y}+ \text{Com}_{\epsilon, h}^{x, y}
  \end{align}
where
  \begin{equation}
  \label{D:eps:h}
     D_{\epsilon, h}^{x, y}
  =
  K_h*(D_{\epsilon}w_{\epsilon})(x) +  K_h*(D_{\epsilon}w_{\epsilon})(y)
  \end{equation}
and
  \begin{equation}   
  \label{com:term}
  \text{Com}_{\epsilon, h}^{x, y}
  =
  [u_{\epsilon}\cdot, K_h*] \nabla w_{\epsilon}(x) +  [u_{\epsilon}\cdot, K_h*] \nabla w_{\epsilon}(y).
  \end{equation}
We conclude the subsection by listing several properties of this weight function without giving a proof (see again \cite{BreJab18} for the proof). 
  \begin{proposition}
  \label{wei:pro}
  Assume that $(\rho_{\epsilon}, u_{\epsilon})$ solves system~\eqref{app:00}--\eqref{app:01} with the bounds~\eqref{pri:app:1} satisfied. Then there exists a  weight function  $w_{\epsilon}$ which satisfies Equation~\eqref{eq:wei}--\eqref{eq:ini:wei} with $D_{\epsilon}$ given by~\eqref{D:eps} such that the following hold:
  \begin{itemize}
  \item For any $t, x$, $0\le w_{\epsilon} \le 1$.
  \item If $p\ge\gamma+1$, then we have
  \begin{equation}
  \label{wei:1}
  \sup_{t\in[0, T]} \int_{\TT^d} \rho_{\epsilon}(t, x)|\log w_{\epsilon}(t, x)|\,dx \le C(1+\lambda).
  \end{equation}
  \item For $p\ge 1+\gamma$, 
  \begin{equation}
  \label{wei:2}
  \sup_{t\in[0, T]} \int_{\TT^d} \rho_{\epsilon}(t, x)\mathbf{1}_{K_h*w_{\epsilon}\le\eta}\,dx \le C\frac{1+\lambda}{|\log\eta|}.
  \end{equation}
  \item For $p>\gamma$, we have the following commutator estimate
  \begin{align}
  \label{wei:3}
  \int_{h_0}^1\int_{0}^t \Vert &K_h*(K_h*(|\div u_{\epsilon}|+ |\cl*P|+|\tilde P_{\epsilon}^{1+l})))w_{\epsilon}
  \nn&
  -
  K_h*(|\div u_{\epsilon}|+ |\cl*P|+|\tilde P_{\epsilon}^{1+l}|)w_{\epsilon, h}
  \Vert_{L^q}\, dt\frac{dh}{h}
  \le C|\log h_0|^{1/2}
  \end{align}
  with $q=\min(2, p/\gamma)$.
  \end{itemize}
  \end{proposition}
  %%%%%%%%%%%%%%%%%%%%%%%%%%%%%%%%%%%%%
  %%%%%%%%%%%%%%%%%%%%%%%%%%%%%%%%%%
  \section{Proof of Theorem \ref{thepstomu}}
  In this section, we give a proof of the Theorem~\ref{thepstomu} using the compactness argument provided in Lemma~\ref{cpt}. Because all coefficients $\eta_i$ are fixed for this section, we drop the index $\eta$ in our notations to keep them simple.

  In order to carry out our approach, we introduce a smooth function $\chi(\xi)\in C^1(\mathbb{R})$ given by
  \begin{equation}
  \label{chi}
     \chi(\xi)=|\xi|^{1+l}
  \end{equation}
where $0<l<1/2$ is to be specified below. We aim to show
  \begin{equation}
  \label{cpt:fun}
    \limsup_{\eps\rightarrow0}\int_{\TT^{2d}}\ck_{h_0}(x-y)\chi(\delta\rho_{\epsilon})\,dxdy \rightarrow 0\ \ \ \ \text{as}\ \ \ \ h_0\rightarrow 0.
  \end{equation}
 To close the estimate, it is convenient to consider the the following quantity instead:
  \begin{equation*}
     T_{h_0, \eps}(t) = \int_{h_0}^1\int_{\TT^{2d}}K_h(x-y) W_{\epsilon, h}^{x, y} \chi(\delta\rho_\eps)\,dxdy\frac{dh}{h}
  \end{equation*}
  where
  \begin{equation*}
  W_{\epsilon, h}^{x, y} = w_{\epsilon, h}^x + w_{\epsilon, h}^y.
  \end{equation*}
The proof of statement~\eqref{cpt:fun} is divided into the following several lemmas. Before stating the lemma, we recall some notation used in subsection~\ref{sub:wei}. The penalization term is defined as
  \begin{equation*}
     D_{\epsilon, h}^{x, y}
  =
  K_h*(D_{\epsilon}w_{\epsilon})(x) +  K_h*(D_{\epsilon}w_{\epsilon})(y)
  \end{equation*}
and commutator term is given as
  \begin{equation*}   
  \text{Com}_{\epsilon, h}^{x, y}
  =
  [u_{\epsilon}\cdot, K_h*] \nabla w_{\epsilon}(x) +  [u_{\epsilon}\cdot, K_h*] \nabla w_{\epsilon}(y).
  \end{equation*}

Compared with~\cite{BreJab18}, we have a different approximation system~\eqref{app:00} and \eqref{app:01}. The main innovation in this paper is the treatment of the pressure term, which is in subsection~\ref{Pre}. For the estimate of the terms $I_1$, $I_2$, and $I_3$ defined below in Lemma~\ref{T:1}, we use similar ideas as in~\cite{BreJab18}. 
% which is treated
% illustrate the ideas of this paper are d
% We treat this term in .
\subsection{The estimate for $T_{h_0, \eps}(t)$}
\begin{lemma}
\label{T:1}  Let $\rho_\eps$ and $u_\eps$ be a sequence of solutions to the system \eqref{app:00}-\eqref{app:01} satisfying the bound \eqref{pri:app:1} with $\gg\ge3d/(d+2)$. Assume that the pressure $P$ satisfies~\eqref{P:1}, \eqref{P:2}, and  \eqref{P:5}. 
Then we have the estimate
  \begin{align}
  \label{T:h0}
    T_{h_0, \eps}(t) \lesssim T_{h_0, \eps}(0)+ I_1 + I_2 +I_3 +I_4 +I_5,
  \end{align}
  where the terms $I_1$--$I_5$ are given by
  \begin{align}
  & I_1 = 
 \int_{0}^t\int_{h_0}^1\int_{\TT^{2d}}\delta u_{\epsilon}\nabla_x K_h(x-y)W_{\epsilon, h}^{x, y}\chi(\delta\rho)\,dxdy\frac{dh}{h}ds \label{I:1}  \\
  & I_2 =   
  -\int_{0}^t\int_{h_0}^1\int_{\TT^{2d}}K_h(x-y)D_{\epsilon, h}^{x, y}\chi(\delta\rho)\,dxdy\frac{dh}{h}ds    %2
\label{I:2}   \\
  & I_3 = 
  \int_{0}^t\int_{h_0}^1\int_{\TT^{2d}} K_h(x-y)\text{Com}_{\epsilon, h}^{x, y}\chi(\delta\rho)\,dxdy\frac{dh}{h}ds   %3
\label{I:3}   \\
  & I_4 = 
  -\frac{1}{2}\int_{0}^t\int_{h_0}^1\int_{\TT^{2d}} K_h(x-y)W_{\epsilon, h}^{x, y}\chi'(\delta\rho)\overline\rho\delta{(\div u_{\epsilon}})(x)   %4
   \,dxdy\frac{dh}{h}ds
 \label{I:4}   \\
     & I_5 = 
  \int_{0}^t\int_{h_0}^1\int_{\TT^{2d}}K_h(x-y)W_{\epsilon, h}^{x, y}\left(\chi(\delta\rho)-\frac{1}{2}\chi'(\delta\rho)\delta\rho\right)\overline{\div_xu_{\epsilon}}(x)    %5
   \,dxdy\frac{dh}{h}ds.
 \label{I:5}
  \end{align}
  \end{lemma}
  \begin{proof}
From~\eqref{app:00}, one gets an equation for $\delta\rho_{\epsilon}$
  \begin{equation*}
     \partial_{t}\delta\rho_{\epsilon} + \div_x(\rho_{\epsilon} u_{\epsilon})(x) - \div_y(\rho_{\epsilon} u_{\epsilon})(y) = 0,
  \end{equation*}
which may be rewritten as
  \begin{align} 
  \label{del:rho}  
  \partial_{t}\delta\rho_{\epsilon} + \div_x(\delta\rho_{\epsilon} u_{\epsilon})(x) + \div_y(\delta\rho_{\epsilon} u_{\epsilon})(y)+\rho_{\epsilon}(y)\div_xu_{\epsilon}(x)  - \rho_{\epsilon}(x)\div_y u_{\epsilon}(y) = 0.
  \end{align}
Note that the terms $\rho_{\epsilon}(y)\div_xu_{\epsilon}(x)$ and $\rho_{\epsilon}(x)\div_y u_{\epsilon}(y)$ are well-defined since $ \rho_{\epsilon}\in L^2$ and $\div_xu_{\epsilon}\in L^2$.
By~\eqref{pri:app:1}, we have $\rho_{\epsilon}^{1+l}\in L^2$ for $\gg> 2(1+l)$ and $\nabla_xu_{\epsilon}\in L^2$. Hence, by Theorem~\ref{ren:sol}, $\delta\rho_{\epsilon}$ is a renormalized solution for the system\eqref{del:rho}. Noticing that 
  \begin{align*} 
  -\rho_{\epsilon}(y)&\div_xu_{\epsilon}(x)  + \rho_{\epsilon}(x)\div_y u_{\epsilon}(y)
  \nn&
  =
  \frac{1}{2}(\delta\rho_{\epsilon}(\div_xu_{\epsilon}(x) + \div_y u_{\epsilon}(y)-\bar\rho_{\epsilon}(\div_xu_{\epsilon}(x) -\div_y u_{\epsilon}(y)))
  \end{align*}
we arrive at
  \begin{align}  
  \label{eqa:del:rho} 
  &\partial_{t}\chi(\delta\rho_\eps) +  \div_x(\chi(\delta\rho_{\epsilon}) u_{\epsilon})(x) + \div_y(\chi(\delta\rho_{\epsilon} ) u_{\epsilon})(y)\\ 
 & \nonumber = \left(\chi(\delta\rho_{\epsilon})-\frac{1}{2}\chi'(\delta\rho_{\epsilon})\delta\rho_{\epsilon}\right)(\div_xu_{\epsilon}(x)  + \div_y u_{\epsilon}(y)
  %\nn&\qquad
  - \frac{1}{2}\chi'(\delta\rho_{\epsilon})\bar\rho(\div_xu_{\epsilon}(x)  - \div_y u_{\epsilon}(y)).
  \end{align}
From the definition of $\chi$ in~\eqref{chi}, it follows easily
  \begin{equation*}
  \chi(\delta\rho_{\epsilon}) \hbox{ and }  \chi'(\delta\rho_{\epsilon})\delta\rho_{\epsilon} \le C\rho_{\epsilon}^{1+l},
  \end{equation*}
which implies that $\chi(\delta\rho_{\epsilon}), \chi'(\delta\rho_{\epsilon})\bar\rho \in L^2$.
Since $\nabla_xu_{\epsilon}\in L^2$, all the terms on the right side of~\eqref{eqa:del:rho} make sense. 
By~\eqref{wei:eq}, we obtain
  \begin{align}   
  \label{z:chi}
  \partial_{t}&(K_h(x-y)W_{\epsilon, h}^{x, y}\chi(\delta\rho_{\epsilon})) 
  =
  K_h(x-y)\partial_{t}W_{\epsilon, h}^{x, y}\chi+ K_h(x-y)W_{\epsilon, h}^{x, y}\partial_{t}\chi
  \nn\quad&
  =-K_h(x-y)u_{\epsilon}(x)\nabla_x W_{\epsilon, h}^{x, y}\chi- K_h(x-y)u_{\epsilon}(y)\nabla_y W_{\epsilon, h}^{x, y}\chi
  -K_h(x-y)D_{\epsilon, h}\chi
    \nn&\qquad
   +K_h(x-y)\text{Com}_{\epsilon, h}\chi 
  + 
  K_h(x-y)W_{\epsilon, h}^{x, y}(\chi-\chi'\delta\rho_{\epsilon})(\div_xu_{\epsilon}(x)+\div_y u_{\epsilon}(y)) 
   \nn&\qquad
  + 
  \frac{1}{2}K_h(x-y)W_{\epsilon, h}^{x, y}\chi'\delta\rho_{\epsilon}\overline{\div u_{\epsilon}}(x)
  -\frac{1}{2}K_h(x-y)W_{\epsilon, h}^{x, y}\chi'\overline\rho_{\epsilon}\delta{\div u_{\epsilon}}(x)
    \nn&\qquad
  -
  K_h(x-y)W_{\epsilon, h}^{x, y}\div_x(\chi u_{\epsilon}(x))
  -
  K_h(x-y)W_{\epsilon, h}^{x, y}\div_y(\chi u_{\epsilon}(y)).
  \end{align}
%where $\text{Com}_{\epsilon, h}$ and $D_{\epsilon, h}$ are defined as in~\eqref{com:term} and \eqref{D:eps:h} respectively. 
The above equation may be justified as the following. First,  in order to show $K_h(x-y)u_{\epsilon}(x)\nabla_x W_{\epsilon, h}^{x, y}\chi \in L^1_{x, y}$, we just need to prove $K_h(x-y)u_{\epsilon}(x)\chi\in L^1_{x, y}$ since $\nabla_x W_{\epsilon, h}^{x, y}\in L^\infty$. In fact we note
  \begin{equation*}
    \label{}
    \begin{split}
 & \int_{\TT^{2d}} K_h(x-y)|u_{\epsilon}(x)|\chi\,dxdy
  =
  \int_{\TT^{d}} K_h(y)  \int_{\TT^{d}}|u_{\epsilon}(x)|\chi(\rho_{\epsilon}(x)-\rho_{\epsilon}(x-y))\,dxdy\\
  &\qquad
  \lesssim \int_{\TT^{d}} K_h(y) \,dy \lesssim 1.
\end{split}
    \end{equation*}
Therefore, the term $K_h(x-y)u_{\epsilon}(x)\nabla_x W_{\epsilon, h}^{x, y}\chi$ is well-defined. Similar arguments could show that
$K_h(x-y)u_{\epsilon}(y)\nabla_y W_{\epsilon, h}^{x, y}\chi \in L^1_{x,y}$.
Second, noting that
  \begin{equation*}
  K_h(x-y)W_{\epsilon, h}^{x, y} \le 2K_h(x-y),
  \end{equation*}
 the terms   $K_h(x-y)W_{\epsilon, h}^{x, y}(\chi-\chi'\delta\rho_{\epsilon})(\div_xu_{\epsilon}(x)+\div_y u_{\epsilon}(y))$, $K_h(x-y)W_{\epsilon, h}^{x, y}\chi'\delta\rho_{\epsilon}\overline{\div u_{\epsilon}}(x)$, and $K_h(x-y)W_{\epsilon, h}^{x, y}\chi'\overline\rho_{\epsilon}\delta{\div u_{\epsilon}}(x)$ belong to $L^1_{x,y}$ by similar arguments as for the first term.
Third, we note that $D_{\epsilon, h}$ is smooth and belongs to $L^\infty$. Hence, $K_h(x-y)D_{\epsilon, h}\chi$ makes sense since $\chi(\delta\rho) \in L^1_x$. One may check easily that $\rho^{1+l} u_{\epsilon} \in L^1$ for $\gg\ge3d/(d+2)$ and thus $K_h(x-y)\text{Com}_{\epsilon, h}\chi \in L^1_{x, y}$.
Lastly, $\div_x(\chi u_{\epsilon}(x)) \in W^{-1, r}$ for some $r>1$ and  $K_h(x-y)W_{\epsilon, h}^{x, y}\in W^{1, r'}$ where $r'$ is the H\"older conjugate exponent of $r$. Therefore, the terms $K_h(x-y)W_{\epsilon, h}^{x, y}\div_x(\chi u_{\epsilon}(x))$ and $K_h(x-y)W_{\epsilon, h}^{x, y}\div_y(\chi u_{\epsilon}(y))$ make sense.
Using the product rule, we further rewrite \eqref{z:chi} as
  \begin{align*}
  \label{}
  \partial_{t}&(K_h(x-y)W_{\epsilon, h}^{x, y}\chi(\delta\rho_{\epsilon})) 
  = -\div_x\left(u(x)K_h(x-y)W_{\epsilon, h}^{x, y}\chi\right) -\div_y\left(u(y)K_h(x-y)W_{\epsilon, h}^{x, y}\chi\right) 
    \nn&\quad
  +\delta u_{\epsilon}(x)\nabla_x K_h(x-y)W_{\epsilon, h}^{x, y}\chi
  -K_h(x-y)D_{\epsilon, h}\chi + K_h(x-y)\text{Com}_{\epsilon, h}\chi 
    \nn&\quad
  + K_h(x-y)W_{\epsilon, h}^{x, y}(\chi-\chi'\delta\rho_{\epsilon})\overline{\div_xu_{\epsilon}}(x)
  + 
  \frac{1}{2}K_h(x-y)W_{\epsilon, h}^{x, y}\chi'\delta\rho_{\epsilon}\overline{\div u_{\epsilon}}(x)
    \nn&\quad
  -\frac{1}{2}K_h(x-y)W_{\epsilon, h}^{x, y}\chi'\overline\rho_{\epsilon}\delta{\div u_{\epsilon}}(x),
  \end{align*}
which could be justified similarly as the equation~\eqref{z:chi}.
Integrating the time derivative of $T_{h_0, \eps}(t)$ from $0$ to $t$ gives~\eqref{T:h0}, concluding the proof.
\end{proof}
\subsection{A bound for $I_1$}
%Next we treat the terms $I_1$--$I_5$ one by one and the estimates are given in Lemmas~\ref{T:1-3}, \ref{T:I:4}, and \ref{T:I:5}.
In this subsection, we estimate the terms $I_1$ in the following lemma. 
\begin{lemma}
\label{T:I1}
Let $I_1$ be given by \eqref{I:1}. Under the assumptions in Lemma~\ref{T:1}, the estimate
  \begin{align*}   
  I_1 \le C|\log h_0|^{1/2} + \bep D_1
  \end{align*}
holds with the penalization $D_1$ defined by
  \begin{equation}
  \label{D:1}
     D_1=\lambda\int_{0}^t\int_{h_0}^1\int_{\TT^{2d}}K_h(x-y)(K_h*((M|\nabla u_{\epsilon}| + |\rho_{\epsilon}|^\gamma)w_{\epsilon})(x)%+K_h*(M|\nabla u_{\epsilon}| + |\rho_{\epsilon}|^\gamma)(y))
  \chi(\delta\rho_{\epsilon})\,dxdy\frac{dh}{h}ds
  \end{equation}
for $t\le T$, where $T$ can be any positive number and the constant $C$ may depend on time $T$.
\end{lemma}
\begin{proof}
We first recall 
  \begin{equation*}
  \label{}
  I_1 = 
 \int_{0}^t\int_{h_0}^1\int_{\TT^{2d}}\delta u_{\epsilon}\nabla_x K_h(x-y)W_{\epsilon, h}^{x, y}\chi(\delta\rho)\,dxdy\frac{dh}{h}ds.
  \end{equation*}
 By Lemma~\ref{dif:u}, it follows
  \begin{align*}   
  |\delta u_{\epsilon}(x)| = |u_{\epsilon}(x)-u_{\epsilon}(y)| 
  \lesssim |x-y|(D_{|x-y|}u_{\epsilon}(x)+D_{|x-y|}u_{\epsilon}(y)),
  \end{align*}
with $D_{h}u_{\epsilon}(x)$ given by 
  \begin{equation*}
     D_{h}u_{\epsilon}(x)
  =
  \frac{1}{h}\int_{|z|\le h}\frac{|\nabla u_{\epsilon}(x+z)|}{|z|^{d-1}}\,dz.
  \end{equation*}
Hence, in view of~\eqref{x*der}, we obtain
  \begin{align*}
  I_1 & \lesssim 
\int_{0}^t\int_{h_0}^1\int_{\TT^{2d}}K_h(x-y)(D_{|x-y|}u_{\epsilon}(x)+D_{|x-y|}u_{\epsilon}(y))W_{\epsilon, h}^{x, y}\chi(\delta\rho_{\epsilon})\,dxdy\frac{dh}{h}ds
 \nn&=
 2 \int_{0}^t\int_{h_0}^1\int_{\TT^{2d}}K_h(x-y)(D_{|x-y|}u_{\epsilon}(x)+D_{|x-y|}u_{\epsilon}(y))w_{\epsilon, h}^{x}\chi(\delta\rho_{\epsilon})\,dxdy\frac{dh}{h}ds
  \end{align*} 
  where we used symmetry in $x$ and $y$ of the integral bound in the last step. Since we only have  \[\Vert u_{\epsilon}\Vert_{L^2H^1} \lesssim 1 \  \text{and}\ \left\lVert   \rho \right \rVert_{L^{\gg}} \lesssim 1 ,\]
  we can not expect the last integral to be much smaller than
  \begin{equation*}
  \label{}
  \left\lVert \int_{h_0}^1 K_h \frac{dh}{h} \right \rVert_{L^{1}} = |\log h_0|.
  \end{equation*}
Instead, we use the penalty defined in~\eqref{D:eps} to absorb the main contribution of $I_1$ and prove 
%that the difference between $I_1$ and the main contribution
the remainder is of the size of $|\log h_0|^{1/2}$.
 In order to proceed, we rewrite
 \begin{align}
  \label{i:1}  \int_{0}^t\int_{h_0}^1&\int_{\TT^{2d}}K_h(x-y)(D_{|x-y|}u_{\epsilon}(x)+D_{|x-y|}u_{\epsilon}(y))w_{\epsilon, h}^{x}\chi(\delta\rho_{\epsilon})\,dxdy\frac{dh}{h}ds
  \nn&=
\int_{0}^t\int_{h_0}^1\int_{\TT^{2d}}K_h(x-y)(D_{|x-y|}u_{\epsilon}(y)-D_{|x-y|}u_{\epsilon}(x))w_{\epsilon, h}^{x}\chi(\delta\rho_{\epsilon})\,dxdy\frac{dh}{h}ds
  \nn&\quad+
  2\int_{0}^t\int_{h_0}^1\int_{\TT^{2d}}K_h(x-y)D_{|x-y|}u_{\epsilon}(x) w_{\epsilon, h}^{x}\chi(\delta\rho_{\epsilon})\,dxdy\frac{dh}{h}ds
  \nn& = I_{1, 1} + I_{1, 2}. 
%\int_{0}^t\int_{h_0}^1\int_{\TT^{2d}}K_h(x-y)|D_{|x-y|}u_{\epsilon}(y)-D_{|x-y|}u_{\epsilon}(x)|w_{\epsilon, h}^{x}\rho_{\epsilon}^{1+l}(x)\,dxdy\frac{dh}{h}ds
%  \nn&\quad+
%\int_{0}^t\int_{h_0}^1\int_{\TT^{2d}}K_h(x-y)|D_{|x-y|}u_{\epsilon}(y)-D_{|x-y|}u_{\epsilon}(x)|w_{\epsilon, h}^{x}\rho_{\epsilon}^{1+l}(y)\,dxdy\frac{dh}{h}ds
%   \nn&\quad+
%  2\int_{0}^t\int_{h_0}^1\int_{\TT^{2d}}K_h(x-y)D_{|x-y|}u_{\epsilon}(x)w_{\epsilon, h}^{x}\chi(\delta\rho_{\epsilon})\,dxdy\frac{dh}{h}ds
  \end{align}
%For the first integral on the right of the last inequality, 
To estimate the term $I_{1, 1}$, we change the variable to arrive at
  \begin{align*}
I_{1, 1}&=\int_{0}^t\int_{h_0}^1\int_{\TT^{2d}}K_h(x-y) (D_{|x-y|}u_{\epsilon}(y)-D_{|x-y|}u_{\epsilon}(x)) w_{\epsilon, h}^{x}\rho_{\epsilon}^{1+l}(x)\,dxdy\frac{dh}{h}ds
  \nn&=
\int_{0}^t\int_{h_0}^1\int_{\TT^{2d}}K_h(z)(D_{|z|}u_{\epsilon}(x-z)-D_{|z|}u_{\epsilon}(x))w_{\epsilon, h}^{x}\rho_{\epsilon}^{1+l}(x)\,dx\,dz\frac{dh}{h}ds.
  \end{align*}
From Proposition~\ref{wei:pro}, we know $0\le w_{\epsilon}\le1$, which implies
  \begin{equation*}   
  0\le w_{\epsilon, h} \le 1
  \end{equation*}
for any $h>0$ since $\Vert K_h\Vert_{L^1}=1.$
By H\"older's inequality, Lemma~\ref{squ:est}, we obtain  
 \begin{align*}
&\int_{0}^t\int_{h_0}^1\int_{\TT^{2d}}K_h(z)(D_{|z|}u_{\epsilon}(x-z)-D_{|z|}u_{\epsilon}(x))w_{\epsilon, h}^{x}\rho_{\epsilon}^{1+l}(x)\,dx\,dz\frac{dh}{h}ds
  \nn&\qquad\lesssim
  \int_{0}^t\int_{h_0}^1\int_{\TT^{d}}K_h(z)\Vert |D_{|z|}u_{\epsilon}(x-z)-D_{|z|}u_{\epsilon}(x)|\Vert_{L^2_x}\,dz\frac{dh}{h}ds
   \nn&\qquad\lesssim
   |\log h_0|^{1/2}\Vert u_{\epsilon}\Vert_{L^2H^1}.
  \end{align*}
%The second integral $I_{1, 2}$ is treated similarly as
%  \begin{align*}    &\int_{0}^t\int_{h_0}^1\int_{\TT^{2d}}K_h(x-y)|D_{|x-y|}u_{\epsilon}(y)-D_{|x-y|}u_{\epsilon}(x)|w_{\epsilon, h}^{x}\rho_{\epsilon}^{1+l}(y)\,dxdy\frac{dh}{h}ds
%  \nn&\qquad= \int_{0}^t\int_{h_0}^1\int_{\TT^{2d}}K_h(z)|D_{|z|}u_{\epsilon}(x-z)-D_{|z|}u_{\epsilon}(x)|w_{\epsilon, h}^{x}\rho_{\epsilon}^{1+l}(x-z)\,dx\,dz\frac{dh}{h}ds
%   \nn&\qquad\lesssim
%  \int_{0}^t\int_{h_0}^1\int_{\TT^{d}}K_h(z)\Vert |D_{|z|}u_{\epsilon}(x-z)-D_{|z|}u_{\epsilon}(x)|\Vert_{L^2_x}\,dz\frac{dh}{h}ds
%   \nn&\qquad\lesssim
%   |\log h_0|^{1/2}\Vert u_{\epsilon}\Vert_{L^2H^1}.
%  \end{align*}
While for the second integral $I_{1, 2}$, it is not in a form to which we could directly apply Lemma~\ref{squ:est}.  Instead, we rewrite it as
  \begin{align*}   
 I_{1,2}&=  2\int_{0}^t\int_{h_0}^1\int_{\TT^{3d}}K_h(x-y)K_h(x-z)(D_{|x-y|}u_{\epsilon}(x)-D_{|x-y|}u_{\epsilon}(z))w_{\epsilon}^{z}\chi(\delta\rho_{\epsilon})\,dxdy\frac{dh}{h}ds
  \nn&\quad+
\int_{0}^t\int_{h_0}^1\int_{\TT^{3d}}K_h(x-y)K_h(x-z)D_{|x-y|}u_{\epsilon}(z)w_{\epsilon}^{z}\chi(\delta\rho_{\epsilon})\,dx\,dy\,dz\frac{dh}{h}ds
  \nn&\le  \int_{0}^t\int_{h_0}^1\int_{\TT^{3d}}K_h(x-y)K_h(x-z)(D_{|x-y|}u_{\epsilon}(x)-D_{|x-y|}u_{\epsilon}(z))w_{\epsilon}^{z}\chi(\delta\rho_{\epsilon})\,dx\,dy\,dz\frac{dh}{h}ds
  \nn&\quad+
  C\int_{0}^t\int_{h_0}^1\int_{\TT^{3d}}K_h(x-y)K_h(x-z)M(\nabla u_{\epsilon})(z)w_{\epsilon}^{z}\chi(\delta\rho_{\epsilon})\,dx\,dy\,dz\frac{dh}{h}ds
  \end{align*}
where we used Lemma~\ref{D:M} in the last step. By Lemma~\ref{squ:est} and the uniform boundedness of $\rho_{\epsilon}$ in $L^{\gg}$, we further get
  \begin{align}  \int_{0}^t\int_{h_0}^1\int_{\TT^{3d}}&K_h(x-y)K_h(x-z)(D_{|x-y|}u_{\epsilon}(x)-D_{|x-y|}u_{\epsilon}(z))w_{\epsilon}^{z}\chi(\delta\rho_{\epsilon})\,dx\,dy\,dz\frac{dh}{h}ds
    \nn&\lesssim
  \int_{0}^t\int_{h_0}^1\int_{\TT^{2d}}K_h(y)K_h(z)\Vert |D_{|y|}u_{\epsilon}(x-z)-D_{|y|}u_{\epsilon}(x)|\Vert_{L^2_x}\,dy\,dz\frac{dh}{h}ds
  \nn&
  \lesssim
  |\log h_0|^{1/2}\Vert u_{\epsilon}\Vert_{L^2H^1}.
  \end{align}
Collecting the estimates of $I_{1,1}$ with $ I_{1,2}$ and applying them to \eqref{i:1} gives
\begin{align}
\label{i:1:fin}
     I_1
  &\lesssim
  |\log h_0|^{1/2}+
  \int_{0}^t\int_{h_0}^1\int_{\TT^{3d}}K_h(x-y)K_h(x-z)M(\nabla u_{\epsilon})(z)w_{\epsilon}^{z}\chi(\delta\rho_{\epsilon})\,dx\,dy\,dz\frac{dh}{h}ds
  \end{align}
where the last integral could be bounded by $\bep D_1$ and the proof is completed.
  \end{proof}

\subsection{An estimate for $I_2$} We denote
  \begin{align*}
  \DD(x) = |\div u_{\epsilon}|(x) 
  + |\cl_{\epsilon}*P|(x)  +|\tilde P_{\epsilon}|^{1+l}(x)
  \end{align*}
 and the estimate for $I_2$ is provided in the lemma below. 
\begin{lemma}
\label{T:I2}
Let $I_2$ be as in~\eqref{I:2}. Under the assumptions in Lemma~\ref{T:1}, then we have that
  \begin{align*}   
  I_2 \le C|\log h_0|^{\theta} -2D_1-2D_2
  \end{align*}
holds for some $1>\theta>0$ with the penalization $D_1$ defined in~\eqref{D:1} and $D_2$ given by
  \begin{align}
  \label{D:2}
     D_2 = &\lambda
  \int_{0}^t\int_{h_0}^1\int_{\TT^{2d}}K_h(x-y)K_h*\DD(x)w_{\epsilon, h}(x)
  \chi(\delta\rho_{\epsilon})\,dxdy\frac{dh}{h}ds
  \end{align}
for $t\le T$, where $T$ can be any positive number and the constant $C$ may depend on time $T$.
\end{lemma}
  \begin{proof}
  The term $I_2$ is negative and helps us in controlling other terms. We pull out the penalization terms $D_1$ with $D_2$ and the error is bounded by $C|\log h_0|^{1/2}$. To be more specific, we have
  \begin{align*}   
  I_2&= -\int_{0}^t\int_{h_0}^1\int_{\TT^{2d}}K_h(x-y)D_{\epsilon, h}^{x, y}\chi(\delta\rho_{\epsilon})\,dxdy\frac{dh}{h}ds
  \nn&=
  -\lambda\int_{0}^t\int_{h_0}^1\int_{\TT^{2d}}K_h(x-y)\overline{(K_h*((M|\nabla u_{\epsilon}| + |\rho_{\epsilon}|^\gamma)w_{\epsilon})}(x)%+K_h*(M|\nabla u_{\epsilon}| + |\rho_{\epsilon}|^\gamma)(y))
  \chi(\delta\rho_{\epsilon})\,dxdy\frac{dh}{h}ds
  \nn&\quad-\lambda
  \int_{0}^t\int_{h_0}^1\int_{\TT^{2d}}K_h(x-y)\overline{(K_h*(K_h*\DD w_{\epsilon})}(x)
  \chi(\delta\rho_{\epsilon})\,dxdy\frac{dh}{h}ds
  \end{align*}
By the symmetry in $x$ and $y$ of the above expression, we further get
   \begin{align}
   \label{i:2}
  I_2 &=
  -2\lambda\int_{0}^t\int_{h_0}^1\int_{\TT^{2d}}K_h(x-y)(K_h*((M|\nabla u_{\epsilon}| + |\rho_{\epsilon}|^\gamma)w_{\epsilon})(x)%+K_h*(M|\nabla u_{\epsilon}| + |\rho_{\epsilon}|^\gamma)(y))
  \chi(\delta\rho_{\epsilon})\,dxdy\frac{dh}{h}ds
  \nn&\quad-2\lambda
  \int_{0}^t\int_{h_0}^1\int_{\TT^{2d}}K_h(x-y)(K_h*(K_h*\DD w_{\epsilon})(x)
   \chi(\delta\rho_{\epsilon})\,dxdy\frac{dh}{h}ds 
  \nn&= -2D_{1} + I_{2, 1}.
  \end{align}
We extract the second penalization $D_2$ from $I_{2, 1}$ as
  \begin{align*}
  &I_{2, 1}=
%  -2\lambda\int_{0}^t\int_{h_0}^1\int_{\TT^{2d}}K_h(x-y)(K_h*((M|\nabla u_{\epsilon}| + |\rho_{\epsilon}|^\gamma)w_{\epsilon})(x)%+K_h*(M|\nabla u_{\epsilon}| + |\rho_{\epsilon}|^\gamma)(y))
%  \chi(\delta\rho_{\epsilon})\,dxdy\frac{dh}{h}ds
%  \nn&\quad
  -2\lambda
  \int_{0}^t\int_{h_0}^1\int_{\TT^{2d}}K_h(x-y)K_h*\DD (x)w_{\epsilon, h}(x)
  \chi(\delta\rho_{\epsilon})\,dxdy\frac{dh}{h}ds
  \nn\ &+2\lambda
  \int_{0}^t\int_{h_0}^1\int_{\TT^{2d}}K_h(x-y)\Big(K_h*\DD (x)w_{\epsilon, h}(x)
   -K_h*(K_h*\DD w_{\epsilon})(x)\Big)
  \chi(\delta\rho_{\epsilon})\,dxdy\frac{dh}{h}ds.
%  \nn&=
%  -2D_2 +2\lambda
%  \int_{0}^t\int_{h_0}^1\int_{\TT^{2d}}K_h(x-y)(K_h*(|\div u_{\epsilon}| + |\cl*P|)(x)K_h*w_{\epsilon}(x) 
%  \nn &\quad -K_h*(K_h*(|\div u_{\epsilon}| + |\cl*P|)w_{\epsilon})(x))
%  \chi(\delta\rho_{\epsilon})\,dxdy\frac{dh}{h}ds.
  \end{align*}
Noting $w_{\epsilon, h}(x) = K_h*w_{\epsilon}(x) $, in view of~\eqref{wei:3}, we may 
bound the last commutator integral in the above equality by
  \begin{equation*}  
  C|\log h_0|^{\theta}
  \end{equation*}
for some $1>\theta>0$.
Therefore, we arrive at
  \begin{equation*}   
  I_{2,1}\le -2D_2 +C|\log h_0|^{\theta}.
  \end{equation*}
Hence, from~\eqref{i:2} we get
  \begin{equation}
  \label{i:2:fin}
  I_2 \le -2D_1-2D_2 +C|\log h_0|^{\theta}
  \end{equation}
concluding the proof.
  \end{proof}
\subsection{Treatment of $I_3$} We bound the term $I_3$ in this subsection.
\begin{lemma} 
\label{T:I3}
Let $I_3$ be given by~\eqref{I:3}. Under the assumptions in Lemma~\ref{T:1}, the estimate
  \begin{align*}   
   I_3 \le C|\log h_0|^{1/2} -\bep D_1
  \end{align*}
holds with the penalization $D_1$ defined by~\eqref{D:1}  
for $t\le T$, where $T$ can be any positive number and the implicit constant may depend on time $T$.
\end{lemma}
  \begin{proof}
   In view of~\eqref{com:term}, we may write $I_3$ as
  \begin{align*}   
   I_3&=\int_{0}^t\int_{h_0}^1\int_{\TT^{2d}}K_h(x-y)\text{Com}_{\epsilon, h}^{x, y}\chi(\delta\rho_{\epsilon})\,dxdy\frac{dh}{h}ds\nn&
  =
  \int_{0}^t\int_{h_0}^1\int_{\TT^{2d}}K_h(x-y)\left([u_{\epsilon}\cdot, K_h*] \nabla w_{\epsilon}(x) +  [u_{\epsilon}\cdot, K_h*] \nabla w_{\epsilon}(y)\right)\chi(\delta\rho_{\epsilon})\,dxdy\frac{dh}{h}ds
  \nn&
  =
  2\int_{0}^t\int_{h_0}^1\int_{\TT^{2d}}K_h(x-y)[u_{\epsilon}\cdot, K_h*] \nabla w_{\epsilon}(x) \chi(\delta\rho_{\epsilon})\,dxdy\frac{dh}{h}ds
  \end{align*}
where we used the symmetry in $x$ and $y$ in the last step.
Expanding the commutator and using the identity
  \begin{equation*}
  u_{\epsilon}\cdot \nabla w_{\epsilon}(x) = \div(u_{\epsilon} w_{\epsilon}(x)) -\div(u_{\epsilon})w_{\epsilon}(x),
  \end{equation*}
we arrive at
  \begin{align*}   
  I_3=&
  2\int_{0}^t\int_{h_0}^1\int_{\TT^{3d}}K_h(x-y)
  (u_{\epsilon}^{x}\cdot \nabla K_h(x-z) w_{\epsilon}^{z} - u_{\epsilon}^{z}\cdot \nabla K_h(x-z) w_{\epsilon}^{z})\chi(\delta\rho_{\epsilon})\,dx\,dy\,dz\frac{dh}{h}ds
  \nn&
  +
  2\int_{0}^t\int_{h_0}^1\int_{\TT^{2d}}K_h(x-y)
  K_h*(\div u_{\epsilon} w_{\epsilon})({x})\chi(\delta\rho_{\epsilon})\,dxdy\frac{dh}{h}ds
  \nn=&
  2\int_{0}^t\int_{h_0}^1\int_{\TT^{3d}}K_h(x-y)
  (u_{\epsilon}^{x}-u_{\epsilon}^{z})\cdot \nabla K_h(x-z) w_{\epsilon}^{z}\chi(\delta\rho_{\epsilon})\,dx\,dy\,dz\frac{dh}{h}ds
  \nn&
  +
  2\int_{0}^t\int_{h_0}^1\int_{\TT^{2d}}K_h(x-y)
  K_h*(\div u_{\epsilon} w_{\epsilon})({x})\chi(\delta\rho_{\epsilon})\,dxdy\frac{dh}{h}ds
  \end{align*}
where the second integral in the last equality of the above expression is bounded by $\bep D_1$ since 
  \begin{equation*}
  |\div u_{\epsilon}| \le |\nabla u_{\epsilon}| \le M|\nabla u_{\epsilon}|.
  \end{equation*}
  By Lemma~\ref{dif:u} and the inequality~\eqref{x*der}, the first integral is estimated as
  \begin{align}   
  \label{i:3}
  &\int_{0}^t\int_{h_0}^1\int_{\TT^{3d}}K_h(x-y)
  (u_{\epsilon}^{x}-u_{\epsilon}^{z})\cdot \nabla K_h(x-z) w_{\epsilon}^{z}\chi(\delta\rho_{\epsilon})\,dx\,dy\,dz\frac{dh}{h}ds
  \nn&\quad
  \lesssim
  \int_{0}^t\int_{h_0}^1\int_{\TT^{3d}}K_h(x-y)
  (D_{|x-z|}u_{\epsilon}(x)+D_{|x-z|}u_{\epsilon}(z))
  \nn&
  \sp\times|(x-z)\cdot \nabla K_h(x-z)| w_{\epsilon}^{z}\chi(\delta\rho_{\epsilon})\,dx\,dy\,dz\frac{dh}{h}ds
  \nn&\quad
  \lesssim
  \int_{0}^t\int_{h_0}^1\int_{\TT^{3d}}K_h(x-y)K_h(x-z)
  |D_{|x-z|}u_{\epsilon}(x)-D_{|x-z|}u_{\epsilon}(z)| w_{\epsilon}^{z}\chi(\delta\rho_{\epsilon})\,dx\,dy\,dz\frac{dh}{h}ds
  \nn&\qquad \qquad
   +2\int_{0}^t\int_{h_0}^1\int_{\TT^{3d}}K_h(x-y)K_h(x-z)
  D_{|x-z|}u_{\epsilon}(z) w_{\epsilon}^{z}\chi(\delta\rho_{\epsilon})\,dx\,dy\,dz\frac{dh}{h}ds
  \end{align}
  where the second integral in the last inequality is bounded by $\bep D_1$ by Lemma~\ref{D:M}. By the definition of $\chi$ in~\eqref{chi}, we change the variable to get
  \begin{align*}   
  \int_{0}^t&\int_{h_0}^1\int_{\TT^{3d}}K_h(x-y)K_h(x-z)
  |D_{|x-z|}u_{\epsilon}(x)-D_{|x-z|}u_{\epsilon}(z)| w_{\epsilon}^{z}\chi(\delta\rho_{\epsilon})\,dx\,dy\,dz\frac{dh}{h}ds
  \nn&=
  \int_{0}^t\int_{h_0}^1\int_{\TT^{3d}}K_h(y)K_h(z)
  |D_{|z|}u_{\epsilon}(x)-D_{|z|}u_{\epsilon}(x-z)| w_{\epsilon}^{x-z}
  \chi(\rho_{\epsilon}^x-\rho_{\epsilon}^{x-z})\,dx\,dy\,dz\frac{dh}{h}ds
  \nn&\lesssim
  \int_{0}^t\int_{h_0}^1\int_{\TT^{3d}}K_h(y)K_h(z)
  |D_{|z|}u_{\epsilon}(x)-D_{|z|}u_{\epsilon}(x-z)| w_{\epsilon}^{x-z}
  \nn&\quad
  \sp\times(\rho_{\epsilon}^{1+l}(x)+\rho_{\epsilon}^{1+l}(x-z))\,dx\,dy\,dz\frac{dh}{h}ds,
  \end{align*}
%For $I_{3,1}$, changing variable and using Lemma~\ref{squ:est}  we get
%  \begin{align*}   
%  I_{3,1} &= \int_{0}^t\int_{h_0}^1\int_{\TT^{2d}}K_h(x-z)
%  |D_{|x-z|}u_{\epsilon}(x)-D_{|x-z|}u_{\epsilon}(z)| w_{\epsilon}^{z}\rho_{\epsilon}^{1+l}(x)\,dxdz\frac{dh}{h}ds
%  \nn&=
%  \int_{0}^t\int_{h_0}^1\int_{\TT^{2d}}K_h(z)
%  |D_{|z|}u_{\epsilon}(x)-D_{|z|}u_{\epsilon}(x-z)| w_{\epsilon}^{x-z}\rho_{\epsilon}^{1+l}(x)\,dxdz\frac{dh}{h}ds
%  \nn&\lesssim
%  \int_{0}^t\int_{h_0}^1\int_{\TT^{d}}K_h(z)
%  \Vert D_{|z|}u_{\epsilon}(x)-D_{|z|}u_{\epsilon}(x-z)\Vert_{L^2_x} \,dz\frac{dh}{h}ds
%  \nn&\lesssim
%  |\log h_0|^{1/2}\Vert u_{\epsilon}\Vert_{L^2H^1}.
%  \end{align*}
from where by H\"older's inequality and Lemma~\ref{squ:est} we obtain a further bound of the above integral
  \begin{align*}   
  \int_{0}^t\int_{h_0}^1\int_{\TT^{d}}K_h(z)K_h(y)&
  \Vert D_{|z|}u_{\epsilon}(x)-D_{|z|}u_{\epsilon}(x-z)\Vert_{L^2_x} \,dz\,dy\frac{dh}{h}ds
%  \nn&
   \\&\lesssim
  |\log h_0|^{1/2}\Vert u_{\epsilon}\Vert_{L^2H^1} \lesssim  |\log h_0|^{1/2}.
  \end{align*}  
Collecting the estimates for the two terms in~\eqref{i:3}, we arrive at
  \begin{equation}
  \label{i:3:fin}
  I_{3} \le  C |\log h_0|^{1/2} + \bep D_1
  \end{equation}
 proving the lemma.
  \end{proof}
\subsection{Pressure term}
\label{Pre}
In this section, we treat the terms involving the pressure. Actually the pressure term appears in both $I_4$ and $I_5$ in slightly different forms. We introduce an abstract function to give the estimate in a more general form and the corresponding bounds in terms $I_4$ and $I_5$ follow easily.
We define the following integral
  \begin{align}
  \label{I:P}
       I_P = -\frac{1}{\log 2}\int_{0}^t\int_{\epsilon}^{2\epsilon}\int_{h_0}^1\int_{\TT^{3d}} &K_{h}(x-y)K_{h}(x-\yy)f(x, y, \yy)w_{\epsilon, h}(x) \nn & \times
  (L_{\epsilon'} *P(x)-L_{\epsilon'} *P(y))
   \,dx\,dy\,d\yy\frac{dh}{h}\frac{d\epsilon'}{\epsilon'}ds
  \end{align}
and establish the estimate of $I_P$ in the Lemma~\ref{T:I:P} below.

In the estimate of the first three terms $I_1$, $I_2$, and $I_3$, the argument is still true even if we replace the mollifying kernel $\cl_{\epsilon}$ by $L_{\epsilon}$, i.e., we may have an upper bound point-wise in $\epsilon$. The kernel $\cl_{\epsilon}$ is only necessary in the treatment of the pressure term. In fact for the pressure term, it is very difficult to obtain an estimate uniform in $\epsilon$ (using the mollifier $L_{\epsilon}$) since when $\epsilon$ is relatively big compared to $h_0$, the error term Diff defined by~\eqref{Diff} is out of control because $L_{\epsilon}*P$ can not approximate $P$ precisely enough. Therefore, instead of consider a $L^\infty_\eps$  topology, we consider $L^1_\epsilon(d\eps/\eps)$. 
In order to treat the term $I_P$, we need to study two cases separately, i.e., $h\le\epsilon'$ and $\epsilon'\le h$. 
The case $h\le\epsilon'$ is easy. We bound the term $\delta(L_{\epsilon'} *P)$ by the H\"older norm of $L_{\epsilon'}$, which is 
under our control since $\epsilon'$ is relatively big. 
For the case $\epsilon'\le h$, it is much more difficult. Roughly speaking, we use the fact that the smoothing effect of $K_h$ is dominant since the scaling of $L_{\epsilon'}$ is smaller. Therefore, we treat 
$L_{\epsilon'} *P$ as an approximation of $P$ which is bounded by $P$ in any $L^p$ for $p\in[1, \infty]$ such that $P\in L^p$.
The main difficulty of executing this idea is that we can not control $L_{\epsilon'} *P$ directly with our penalization. Instead, we need to consider the quantity $L_{\epsilon'} *(w^\theta P)$ for some $\theta>0$ (see \eqref{I:s:1}). Hence, we have to control commutator between the weight function and the convolution with $L_{\epsilon}$ to close the estimate.
\begin{lemma}
\label{T:I:P}
Let $I_P$ be defined by~\eqref{I:P} and  $(\rho_\eps, u_\eps)$ be a sequence of solutions to the system \eqref{app:00}-\eqref{app:01} satisfying the bound \eqref{pri:app:1} with $\gg\ge \max(2s_0^*, s_1^*, 3d/(d+2))$ where $s_0^*$ and $s_1^*$ are the H\"older conjugate exponent of $\ss$ and $s_1$ respectively. 
Assume the pressure $P$ satisfies~\eqref{P:1}, \eqref{P:2}, \eqref{P:5}, and \eqref{P:las}. 
Let $f(x, y, \yy)$ be such that
  \begin{equation}
  \label{f:lp}
    | f(x, x-y, x-\yy)| \le C(\chi'(\delta\rho_{\epsilon}(x, y))\overline{\rho_{\epsilon}}(x, y)+\chi'(\delta\rho_{\epsilon}(x, \yy))\overline{\rho_{\epsilon}}(x, \yy)).
  \end{equation}
%for $p\in [1, \infty]$ and $g(x)$ such that the terms in~\eqref{f:lp} make sense.
 Let $r_h$ be defined as in~\eqref{P:las}. We have
  \begin{align*}
  |I_P| \le  C  
   + C \left(\int_{\epsilon}^{2\epsilon} r_{\max(h_0,\epsilon')}\frac{d\epsilon'}{\epsilon'}\right)^{\bar\theta}|\log(h_0)|^\theta
   + C\int_0^t T_{h_0, \eps}(s)\,ds+ \bep D_2 + \frac{3D_3}{8}
  \end{align*}
with $D_2$ given  by~\eqref{D:2} and $D_3$ by
  \begin{equation}  
  \label{D:3}
  D_3=
  \eta(1+l)\int_{0}^t\int_{h_0}^1\int_{\TT^{2d}} K_h(x-y)W_{\epsilon, h}^{x, y}\chi(\delta\rho_{\epsilon})\overline{\rho_{\epsilon}^{\gg}}(x)
  \,dxdy\frac{dh}{h}ds,
  \end{equation}
for some $0<\bar\theta$, $0< \theta<1$, and $t\le T$, where $T$ can be any positive number and the implicit constant may depend on time $T$.
\end{lemma}
  \begin{proof}
Here we give a uniform estimate in $\epsilon$ of this term, which may be divided into two cases: $\epsilon'<h$ and $\epsilon'\ge h$:
  \begin{align*}
     I_P &= -\frac{1}{\log 2}\int_{0}^t\int_{\epsilon}^{2\epsilon}\int_{h_0}^1\int_{\TT^{3d}} K_{h}(x-y)K_{h}(x-\yy)f w_{\epsilon, h}(x) 
  \delta(L_{\epsilon'} *P)
   \,dx\,dy\,d\yy\frac{dh}{h}\frac{d\epsilon'}{\epsilon'}ds 
    \nn&=
   -\frac{1}{\log 2}\int_{0}^t\int_{\epsilon}^{2\epsilon}\int_{h_0}^1\int_{\TT^{3d}} (\mathbf{1}_{\epsilon'\ge h} + \mathbf{1}_{\epsilon'<h})K_{h}(x-y)K_{h}(x-\yy)f w_{\epsilon, h}(x) 
  \\&\sp\qquad\qquad  \times
  \delta(L_{\epsilon'} *P)
   \,dx\,dy\,d\yy\frac{dh}{h}\frac{d\epsilon'}{\epsilon'}ds \nn&=
    I_{b}+I_{s}
  \end{align*} 
where $I_b$ and $I_s$ are corresponding to the integrals with characteristic functions $\mathbf{1}_{\epsilon'\ge h}$ and $ \mathbf{1}_{\epsilon'<h}$ in them respectively. As we see below, the term $I_b$ is easier to treat since in this case the $K_h$ is the mollifier playing the key role, which is more consistent with the whole compactness argument. While for term $I_s$, we need to take the advantage of regularity of the weight function to 
generate an extra small factor $(\epsilon')^\theta$, which help us control the singularity of $K_h$ around the origin.
First we rewrite $I_b$ as
  \begin{align*}   
  |I_{b}| &= \frac{1}{\log 2}\left| \int_{0}^t\int_{\epsilon}^{2\epsilon}\int_{h_0}^1\int_{\TT^{3d}}\mathbf{1}_{\epsilon'\ge h} K_{h}(x-y)K_{h}(x-\yy)fw_{\epsilon, h}(x) 
  \delta(L_{\epsilon'} *P)
   \,dx\,dy\,d\yy \frac{dh}{h}\frac{d\epsilon'}{\epsilon'}ds \right|
    \nn&=
    \frac{1}{\log 2}\bigg| \int_{0}^t\int_{\epsilon}^{2\epsilon}\int_{h_0}^{\max(h_0, \epsilon')}\int_{\TT^{4d}} K_{h}(x-y)K_{h}(x-\yy)
    fw_{\epsilon, h}(x) P(t, z, \rho_{\epsilon}(z))
      \nn&\sp
    (L_{\epsilon'}(x-z)-L_{\epsilon'}(y-z))
   \,dx\,dy\,d\yy dz\frac{dh}{h}\frac{d\epsilon'}{\epsilon'}ds \bigg|
   \nn&=\frac{1}{\log 2}
    \bigg| \int_{0}^t\int_{\epsilon}^{2\epsilon}\int_{h_0}^{\max(h_0, \epsilon')}\int_{\TT^{4d}} K_{h}(y)K_{h}(\yy)f(x, x-y, x-\yy) w_{\epsilon, h}(x) P(t, z, \rho_{\epsilon}(z))
   \nn&\sp
    (L_{\epsilon'}(x-z)-L_{\epsilon'}(x-y-z))
   \,dx\,dy\,d\yy dz\frac{dh}{h}\frac{d\epsilon'}{\epsilon'}ds \bigg|.
  \end{align*}
Due to the smoothness of $L_{\epsilon'}$, we have the uniform bound in $x-z$
  \begin{align*}
  L_{\epsilon'}(x-z)-L_{\epsilon'}(x-y-z) \le \frac{|y|^\theta}{\epsilon'^\theta}
  \end{align*}
with $1>\theta>0$.
By~\eqref{P:1} and \eqref{P:2}, we get
  \begin{align*}
  \int_{\TT^{d}} P(t, z, \rho_{\epsilon}(z)) \,dz \lesssim \int_{\TT^{d}} R(t, z) + \Theta_1(z) + \rho^{p}(z)\,dz \lesssim 1 
  \end{align*}
since $\gg\ge p$.
Therefore, by~\eqref{f:lp}, using the uniform integrability of $\rho_{\epsilon}$ and the fact that
  \begin{equation*}
   \left\lVert  K_h  \right \rVert_{L^1} =1,
  \end{equation*}
 we arrive at
  \begin{align*}
   |I_{b}|&\lesssim
   \int_{\epsilon}^{2\epsilon}\int_{h_0}^{\max(h_0, \epsilon')}\int_{\TT^{d}} K_{h}(y)\frac{|y|^\theta}{\epsilon'^\theta}
   \,dy\frac{dh}{h}\frac{d\epsilon'}{\epsilon'}
   \nn&\lesssim
   \int_{\epsilon}^{2\epsilon}\int_{h_0}^{\max(h_0, \epsilon')} \frac{h^\theta}{\epsilon'^\theta}
   \,\frac{dh}{h}\frac{d\epsilon'}{\epsilon'}
   \lesssim 1.
  \end{align*}  
Next we treat the difficult term $I_s$.
Denoting $\tilde\eps = \max(h_0,\epsilon')$,  by assumptions~\eqref{P:5}, we obtain
\begin{align}
  \label{eps:sma}
  |I_{s}| &  \le
 C\int_{0}^t\int_{\epsilon}^{2\epsilon}\int_{\tilde\eps}^1\int_{\TT^{4d}} K_{h}(x-y)K_{h}(x-\yy)f(x, y, \yy)w_{\epsilon, h}(x) 
  L_{\epsilon'}(z)|\rho_{\epsilon}(x-z)-\rho_{\epsilon}(y-z)|
  \nn&\qquad
  (\rho_{\epsilon}^{\gamma-1}(x-z)+\rho_{\epsilon}^{\gamma-1}(y-z))
   \,dx\,dy\,d\yy dz\frac{dh}{h}\frac{d\epsilon'}{\epsilon'}ds
   \nn&\quad+
    C\int_{0}^t\int_{\epsilon}^{2\epsilon}\int_{\tilde\eps}^1\int_{\TT^{4d}} K_{h}(x-y)K_{h}(x-\yy)f(x, y, \yy) w_{\epsilon, h}(x) 
  L_{\epsilon'}(z)(Q_{\epsilon}^{x-z, y-z}
   \nn&\qquad 
   + (\tilde P_{\epsilon}^{x-z}+\tilde P_{\epsilon}^{y-z})|\rho_{\epsilon}(t, x-z)-\rho_{\epsilon}(t, y-z)|)
   \,dx\,dy\,d\yy dz\frac{dh}{h}\frac{d\epsilon'}{\epsilon'}ds
   \nn&= I_{s,1} + I_{s,2} + I_{s,3}
  \end{align}
where $I_{s,1}$ is the first integral with $I_{s,2}$ and $I_{s,3}$ corresponding to the integrals containing $Q_{\epsilon}^{x-z, y-z}$ and $(\tilde P_{\epsilon}^{x-z}+\tilde P_{\epsilon}^{y-z})$ respectively. For the sake of simplicity, we suppress the constant $C$ in $I_{s,1}$, $I_{s,2}$, and $I_{s,3}$. By making constants in the following estimates bigger if necessary,  we may recover the bound for $I_s$.
The first integral $I_{s, 1}$ is the most difficult one among the three. In order to estimate this term, we need to use the penalization term $D_3$ as well as the regularity of the weight function $w_{\epsilon, h}$.  To be more specific, we have
  \begin{align*}
     I_{s,1} =&  \int_{0}^t\int_{\epsilon}^{2\epsilon}\int_{\tilde\eps}^1\int_{\TT^{4d}} 
  K_{h}(y)K_{h}(\yy)f(x, x-y, x-\yy)w_{\epsilon, h}(x) 
  L_{\epsilon'}(z)\,\\
  &\quad |\rho_{\epsilon}(x-z)-\rho_{\epsilon}(x-y-z)|\,
  (\rho_{\epsilon}^{\gamma-1}(x-z)+\rho_{\epsilon}^{\gamma-1}(x-y-z))
  \,dx\,dy\,d\yy dz\frac{dh}{h}\frac{d\epsilon'}{\epsilon'}ds\\
  &= \bar I_{s, 1} + \text{Diff}
  \end{align*}
where we denoted using the notation in Subsection~\ref{notation}
  \begin{align}
  \label{I:s:1}
      \bar I_{s, 1}  =& \int_{0}^t\int_{\epsilon}^{2\epsilon}\int_{\tilde\eps}^1\int_{\TT^{4d}} 
  K_{h}(y)K_{h}(\yy)f(x, x-y, x-\yy) w^{1/\gg}_{\epsilon, h}(x) w^{1-1/\gg}_{\epsilon, h}(x-z)
  L_{\epsilon'}(z)\nn&\quad \times
  |\delta\rho_{\epsilon}(x-z, y)|\overline{\rho_{\epsilon}^{\gamma-1}}(x-z, y)
   \,dx\,dy\,dz\frac{dh}{h}\frac{d\epsilon'}{\epsilon'}ds
  \end{align}
and 
  \begin{align}
    \label{Diff}
     \text{Diff} = & \int_{0}^t\int_{\epsilon}^{2\epsilon}\int_{\tilde\eps}^1\int_{\TT^{4d}} 
  K_{h}(y)K_{h}(\yy)f(x, x-y, x-\yy) w^{1/\gg}_{\epsilon, h}(x)
  L_{\epsilon'}(z)   |\delta\rho_{\epsilon}(x-z, y)| \nn&\quad \times
   (w^{1-1/\gg}_{\epsilon, h}(x)-w^{1-1/\gg}_{\epsilon, h}(x-z))
  \overline{\rho_{\epsilon}^{\gamma-1}}(x-z, y)
   \,dx\,dy\,dz\frac{dh}{h}\frac{d\epsilon'}{\epsilon'}ds.
  \end{align}
As we see below, the term $ \bar I_{s, 1} $ is the leading order term and Diff is a perturbation of constant size.
Using H\"older's inequality, the term $ \bar I_{s, 1} $ is bounded by
  \begin{align}  
  \label{bar:I:s:1}
   &
%  \bar I_{s, 1} 
%  \nn&\le
  \int_{0}^t\int_{\epsilon}^{2\epsilon}\int_{\tilde\eps}^1\int_{\TT^{2d}} 
  K_{h}(y)K_h(\yy)\Vert L_{\epsilon'}*(|\delta\rho_{\epsilon}(x, y)|
  \overline{\rho_{\epsilon}^{\gamma-1}}(x, y)w^{1-1/\gg}_{\epsilon, h})\Vert_{L^{\gg'}_x}
   \nn&\qquad\qquad\times
  \Vert f(x, x-y, x-\yy) w^{1/\gg}_{\epsilon, h} \Vert_{L^{\gg}_x}
   \,dy\,d\yy\frac{dh}{h}\frac{d\epsilon'}{\epsilon'}ds
   \nn&\le
  C\int_{0}^t\int_{\epsilon}^{2\epsilon}\int_{\tilde\eps}^1\int_{\TT^{2d}} 
  K_{h}(y)K_h(\yy)\Vert \delta\rho_{\epsilon}(x, y)
  \overline{\rho_{\epsilon}^{\gamma-1}}(x, y)w^{1-1/\gg}_{\epsilon, h}\Vert_{L^{\gg'}_x}
  \nn&\qquad \times
  \Vert (|\chi'|\overline{\rho_{\epsilon}}(x, y)+|\chi'|\overline{\rho_{\epsilon}}(x, \yy)) w^{1/\gg}_{\epsilon, h} \Vert_{L^{\gg}_x}
   \,dy\,d\yy\frac{dh}{h}\frac{d\epsilon'}{\epsilon'}ds
   \nn&\le
  C\int_{0}^t\int_{\epsilon}^{2\epsilon}\int_{\tilde\eps}^1\int_{\TT^{2d}} 
  K_{h}(y)K_h(\yy)
  \Vert (\delta\rho_{\epsilon})^{\sigma}(x, y)
  w^{1-\gamma/\gg}_{\epsilon, h}\Vert_{L^{\alpha_1}_x}
  \nn
  &\quad \Vert (\delta\rho_{\epsilon})^{1-\sigma}(x, y)
  \overline{\rho_{\epsilon}^{\gamma-1}}(x, y)w^{(\gamma-1)/\gg}_{\epsilon, h}\Vert_{L^{\alpha_2}_x}
  \,  \Vert (|\chi'|\overline{\rho_{\epsilon}}(x, y)+|\chi'|\overline{\rho_{\epsilon}}(x, \yy)) w^{1/\gg}_{\epsilon, h} \Vert_{L^{\gg}_x}\nn
  &\qquad\qquad\qquad
   \,dy\,d\yy\frac{dh}{h}\frac{d\epsilon'}{\epsilon'}ds
  \end{align}
where $\alpha_1$, $\alpha_2$, and $\sigma$ are given by
  \begin{align*}
  \alpha_1 = \frac{\gg}{\gg-\gamma} \ \ \ \ \ \ \ \ \ \  \alpha_2 = \frac{\gg}{\gamma-1}, \ \ \ \ \ \ \ \ \ \   \sigma= 1-\frac{(\gamma-1)(1+l)}{\gg}.
  \end{align*}
%the exponents satisfy the following conditions:
%  \begin{align}
%  \label{exp}
%  \beta = \gg, &\ \ \ \ \ \ \ \ \ \  l\beta=l\gg=1+l, \ \ \ \ \ \ \ \ \ \   (1-\theta)\beta=1\nn
%  \alpha_2(\gamma-1) = \gg, &\ \ \ \ \ \ \ \ \ \  (1-\sigma)\alpha_2 = 1+l, \ \ \ \ \ \ \ \ \ \  \theta_2\alpha_2 = 1 \nn
%  \sigma\alpha_1 = 1+l, &\ \ \ \ \ \ \ \ \ \ \theta_1\alpha_1 = 1. 
%  \end{align}
We also require 
  \begin{equation*}
  \label{}
  l\gg=1+l.
  \end{equation*}
%Note that the above requirement are satisfied provided $\gg$ is big enough and moreover the parameters enjoy the equality:
%  \begin{equation*}
%  \theta_1 + \theta_2 = \frac{1}{\alpha_1} + \frac{1}{\alpha_2} = 1- \frac{1}{\beta},
%  \end{equation*}
%which makes the estimates of $I_h^L$ valid. 
Using Young's inequality, one further gets
  \begin{align*}   
  & \bar I_{s, 1}  \le 
  \int_{0}^t\int_{h_0}^1\int_{\TT^{2d}} 
  K_{h}(y)K_h(\yy)\bigg(\frac{C}{\eta} \int  |\delta\rho_{\epsilon}|^{1+l}(x, y)
  w_{\epsilon, h}\, dx
  \nn&\ 
  + \frac{\eta}{16} \int
  |\delta\rho_{\epsilon}|^{1+l}(x, y)
  \overline{\rho_{\epsilon}^{\gg}}(x, z)w_{\epsilon, h}\,dx
  + \frac{\eta}{16} \int
  |\delta\rho_{\epsilon}|^{1+l}(x, \yy)
  \overline{\rho_{\epsilon}^{\gg}}(x, y)w_{\epsilon, h}\,dx
  \bigg)
   \,dy\,d\yy\frac{dh}{h}ds
%  \nn&=
%  \int_{0}^t\int_{\epsilon}^{2\epsilon}\int_{\TT^{d}} 
%  K_{h}(x-y)\bigg(\frac{8}{\mu} \int  |\delta\rho_{\epsilon}|^{1+l}
%  w_{\epsilon, h}\, dx
%  \nn&\sp
%  + \frac{\mu}{8} \int
%  |\delta\rho_{\epsilon}|^{1+l}
%  \overline{\rho_{\epsilon}^{\gg}}w_{\epsilon, h}\,dx
%  \bigg)
%   \,dy\frac{dh}{h}\frac{d\epsilon'}{\epsilon'}ds
  \nn&=
  \frac{C}{\eta}\int_{0}^t\int_{h_0}^1\int_{\TT^{2d}} 
  K_{h}(x-y) |\delta\rho_{\epsilon}|^{1+l}
  w_{\epsilon, h}\, dxdy\frac{dh}{h}ds
  \nn&\sp
  + \frac{\eta}{8} \int_{0}^t\int_{h_0}^1\int_{\TT^{2d}} 
  K_{h}(x-y) |\delta\rho_{\epsilon}|^{1+l}
  \overline{\rho_{\epsilon}^{\gg}}w_{\epsilon, h}
   \,dxdy\frac{dh}{h}ds
   \end{align*}
 where we used $\left\lVert  K_h  \right \rVert_{L^1} =1$ and  the last integral may be bounded by $D_3/8$. 
Next we turn to the term Diff. Noting 
  \begin{equation*}
     w^{1-1/\gg}_{\epsilon, h}(x)-w^{1-1/\gg}_{\epsilon, h}(x-z) \le C\frac{|z|^{1-1/\gg}}{h^{1-1/\gg}},
  \end{equation*}
 we obtain
  \begin{align}
  \label{diff}
     \text{Diff} & \le  \int_{0}^t\int_{\epsilon}^{2\epsilon}\int_{\tilde\eps}^1\int_{\TT^{4d}} 
  K_{h}(y)K_{h}(\yy)f(x, x-y, x-\yy) w^{1/\gg}_{\epsilon, h}(x) \frac{|z|^{1-1/\gg}}{h^{1-1/\gg}}
  L_{\epsilon'}(z)
  \nn&\qquad\qquad\qquad\qquad
  \times |\delta\rho_{\epsilon}(x-z, y)|\overline{\rho_{\epsilon}^{\gamma-1}}(x-z, y)
   \,dx\,dy\,dz\,d\yy\frac{dh}{h}\frac{d\epsilon'}{\epsilon'}ds
   \nn&\le
   C \int_{\epsilon}^{2\epsilon}\int_{\tilde\eps}^1\int_{\TT^{3d}} 
  K_{h}(y)  K_{h}(\yy) \frac{|z|^{1-1/\gg}}{h^{1-1/\gg}}
  L_{\epsilon'}(z)
  \Vert f(x, x-y, x-\yy) w^{1/\gg}_{\epsilon, h} \Vert_{L^{\gg}_x}
     \nn&\sp\times
  \Vert |\delta\rho_{\epsilon}(x, y)|
  \overline{\rho_{\epsilon}^{\gamma-1}}(x, y)\Vert_{L^{\gg'}_x},
   \,dy\,dz\,d\yy\frac{dh}{h}\frac{d\epsilon'}{\epsilon'}
  \end{align}
from where using~\eqref{f:lp} and Young's inequality, by the uniform integrability of $\rho_{\epsilon}$ and $\left\lVert  K_h  \right \rVert_{L^1} =1$, we further get
  \begin{align*}
   \mbox{Diff}  &
   \le 
  C\nu\eta\int_{0}^t \int_{\epsilon}^{2\epsilon}\int_{\tilde\eps}^1\int_{\TT^{d}} 
   \frac{|z|^{1-1/\gg}}{h^{1-1/\gg}}
  L_{\epsilon'}(z)
   \,dz \int_{\TT^{2d}} 
  K_{h}(x-y) |\delta\rho_{\epsilon}|^{1+l}
  \overline{\rho_{\epsilon}^{\gg}}w_{\epsilon, h}
   \,dxdy \frac{dh}{h}\frac{d\epsilon'}{\epsilon'} ds
    \\ &\sp +
  \frac{C}{\nu}\int_{\epsilon}^{2\epsilon}\int_{\tilde\eps}^1\int_{\TT^{d}} 
   \frac{|z|^{1-1/\gg}}{h^{1-1/\gg}}
  L_{\epsilon'}(z)
   \,dz\frac{dh}{h}\frac{d\epsilon'}{\epsilon'}
  \end{align*}
for a small parameter $\nu>0$.
For the second integral in the right side of the above inequality, we have
  \begin{align*}
   \frac{C}{\nu}\int_{\epsilon}^{2\epsilon}\int_{\tilde\eps}^1\int_{\TT^{d}} 
   \frac{|z|^{1-1/\gg}}{h^{1-1/\gg}}
  L_{\epsilon'}(z)
   \,dz\frac{dh}{h}\frac{d\epsilon'}{\epsilon'}
   \le
   \frac{C}{\nu} \int_{\epsilon}^{2\epsilon}\int_{\tilde\eps}^1
   \frac{(\epsilon')^{1-1/\gg}}{h^{1-1/\gg}}
   \,\frac{dh}{h}\frac{d\epsilon'}{\epsilon'}\le \frac{C}{\nu}.
  \end{align*}
Using $\epsilon'\le h$ and choosing $\nu$ sufficiently small, we arrive at
  \begin{align*}
  &C\nu\eta\int_{0}^t \int_{\epsilon}^{2\epsilon}\int_{\tilde\eps}^1\int_{\TT^{d}} 
   \frac{|z|^{1-1/\gg}}{h^{1-1/\gg}}
  L_{\epsilon'}(z)
   \,dz \int_{\TT^{2d}} 
  K_{h}(x-y) |\delta\rho_{\epsilon}|^{1+l}
  \overline{\rho_{\epsilon}^{\gg}}w_{\epsilon, h}
   \,dxdy \frac{dh}{h}\frac{d\epsilon'}{\epsilon'} ds
  \\&\qquad \le
  C\nu\eta\int_{0}^t \int_{\epsilon}^{2\epsilon}\int_{\tilde\eps}^1 
   \int_{\TT^{2d}} \frac{(\epsilon')^{1-1/\gg}}{h^{1-1/\gg}}
  K_{h}(x-y) |\delta\rho_{\epsilon}|^{1+l}
  \overline{\rho_{\epsilon}^{\gg}}w_{\epsilon, h}
   \,dxdy \frac{dh}{h}\frac{d\epsilon'}{\epsilon'} ds
     \\&\qquad \le
  \frac{\eta}{16}\int_{0}^t \int_{\tilde\eps}^1 
   \int_{\TT^{2d}}
  K_{h}(x-y) |\delta\rho_{\epsilon}|^{1+l}
  \overline{\rho_{\epsilon}^{\gg}}w_{\epsilon, h}
   \,dxdy \frac{dh}{h} ds
  \end{align*}
which may be bounded by $D_3/16$. Therefore, we obtain
  \begin{align*}
  \mbox{Diff} \le C + \frac{D_3}{16}.
  \end{align*}

%$0\le\bar\phi\le 1$, $\bar\phi(s)=1$ for $0\le s \le 1$, and $\bar\phi(s)=0$ for $s\ge2$.
Next we turn to the treatment of the term $I_{s, 2}$. By changing the variables, we rewrite it as
  \begin{align}
  \label{I:s:2}
  I_{s, 2} 
  &=
  \int_{0}^t\int_{\epsilon}^{2\epsilon}\int_{\tilde\eps}^1\int_{\TT^{4d}} K_{h}(x-y)K_{h}(x-\yy)f(x, y, \yy)w_{\epsilon, h}(x) 
  L_{\epsilon'}(z)Q_{\epsilon}^{x-z, y-z}\\
  &\sp
  \,dx\,dy\,dz\,d\yy\frac{dh}{h}\frac{d\epsilon'}{\epsilon'}ds
   \nn&=
  \int_{0}^t\int_{\epsilon}^{2\epsilon}\int_{\tilde\eps}^1\int_{\TT^{4d}} K_{h}(y)K_{h}(\yy)f(x, x-y, x-\yy)w_{\epsilon, h}(x) 
  L_{\epsilon'}(z)
 Q_{\epsilon}^{x-z, x-y-z}\\
&\sp  \,dx\,dy\,dz\,d\yy\frac{dh}{h}\frac{d\epsilon'}{\epsilon'}ds.
  \end{align}
In view of $w_{\epsilon, h}(x)\le 1$, we get
  \begin{align*}
  I_{s, 2}\le
   \int_{0}^t\int_{\epsilon}^{2\epsilon}\int_{\tilde\eps}^1\int_{\TT^{3d}} K_{h}(y)K_{h}(\yy)\left|f(x, x-y, x-\yy)\right| w^{1/\gg}_{\epsilon, h}(x)
%   \nn&\quad
%  \indeq
  L_{\epsilon'}(z) 
  Q_{\epsilon}
  \,dx\,dy\,dz\frac{dh}{h}\frac{d\epsilon'}{\epsilon'}ds 
%  \nn&+
%     \Bigg|\int_{0}^t\int_{\epsilon}^{2\epsilon}\int_{\tilde\eps}^1\int_{\TT^{3d}} K_{h}(y)K_{h}(\yy)f(x, x-y, x-\yy)w_{\epsilon, h}(x) 
%  \nn&\quad
%  \indeq 
%   (1-\phi_{\epsilon}^M)
%  L_{\epsilon'}(z)
%  Q_{\epsilon}
%  \,dxdydz\frac{dh}{h}\frac{d\epsilon'}{\epsilon'}ds \Bigg|
  \end{align*}
where  $Q_{\epsilon} =   Q_{\epsilon}^{x-z, x-y-z}$.
Using~\eqref{f:lp}, H\"older's inequality, and that $\left\lVert  L_{\epsilon'}  \right \rVert_{L^1} =1$, we arrive at
  \begin{align*}  
     I_{s, 2} & \le
%  \bar I_{s, 1} 
%  \nn&\le
  \int_{0}^t\int_{\epsilon}^{2\epsilon}\int_{\tilde\eps}^1\int_{\TT^{2d}} 
  K_{h}(y)K_h(\yy)
  \Vert f(x, x-y, x-\yy) w^{1/\gg}_{\epsilon, h} \Vert_{L^{\gg}_x}
   \nn&\sp\qquad\qquad\qquad\times
   \left\Vert \int_{\TT^{d}}  L_{\epsilon'}(z) 
  Q_{\epsilon}
  \,dz\right\Vert_{L^{\gg'}_x}
   \,dy\,d\yy\frac{dh}{h}\frac{d\epsilon'}{\epsilon'}ds
%   \nn&\le
%  C\int_{0}^t\int_{\epsilon}^{2\epsilon}\int_{\tilde\eps}^1\int_{\TT^{2d}} 
%  K_{h}(y)K_h(\yy)\Vert \delta\rho_{\epsilon}(x, y)
%  \overline{\rho_{\epsilon}^{\gamma-1}}(x, y)w^{1-1/\gg}_{\epsilon, h}\Vert_{L^{\gg'}_x}
%  \nn&\sp\qquad\qquad\times
%  \Vert (|\chi'|\overline{\rho_{\epsilon}}(x, y)+|\chi'|\overline{\rho_{\epsilon}}(x, \yy)) w^{1/\gg}_{\epsilon, h} \Vert_{L^{\gg}_x}
%   \,dyd\yy\frac{dh}{h}\frac{d\epsilon'}{\epsilon'}ds
   \nn&\le
  C\int_{0}^t\int_{\epsilon}^{2\epsilon}\int_{\tilde\eps}^1\int_{\TT^{2d}} 
  K_{h}(y)K_h(\yy)
  \Vert (|\chi'|\overline{\rho_{\epsilon}}(x, y)+|\chi'|\overline{\rho_{\epsilon}}(x, \yy)) w^{1/\gg}_{\epsilon, h} \Vert_{L^{\gg}_x}
     \nn&\sp\qquad\qquad\qquad\times
   \left\Vert
  Q_{\epsilon}^{x, x-y}
  \right\Vert_{L^{\gg'}_x}
   \,dy\,d\yy\frac{dh}{h}\frac{d\epsilon'}{\epsilon'}ds
  \end{align*}
where  $\gg'$ is the H\"older conjugate exponent of $\gg$.
%While for the second term,
%, we use the simple relation
%  \begin{equation*}
%  (\{\rho(x)\ge M\}\cap\{\rho(z)\ge M\})^c = \{\rho(x)\ge M\}^c\cup\{\rho(z)\ge M\}^c
%  \end{equation*} to 
By Young's inequality, we further get
  \begin{align*}   
   I_{s, 2} &\le
  \frac{\eta}{8} \int_{0}^t\int_{h_0}^1\int_{\TT^{2d}} 
  K_{h}(x-y) |\delta\rho_{\epsilon}|^{1+l}
  \overline{\rho_{\epsilon}^{\gg}}w_{\epsilon, h}
   \,dxdy\frac{dh}{h}ds
  \nn&\sp
  + 
    \frac{C}{\eta}\int_{0}^t\int_{h_0}^1\int_{\TT^{2d}} 
  K_{h}(x-y) |Q_{\epsilon}^{x, x-y}|^{\gg'}\, dxdy\frac{dh}{h}ds
    \end{align*}
 where the first integral on the right side is bounded by $D_3/8$.
Using H\"older's inequalities, the second integral may be estimated as
  \begin{align*}
  &\frac{C}{\eta}\int_{0}^t\int_{h_0}^1\int_{\TT^{2d}} 
  K_{h}(x-y) |Q_{\epsilon}^{x, x-y}|^{\gg'}\, dxdy\frac{dh}{h}ds
  \\
  &\quad\le
  \frac{C}{\eta}
  \left(\int_{0}^t\int_{h_0}^1\int_{\TT^{2d}} 
  K_{h}(x-y)\, dxdy\frac{dh}{h}ds\right)^{(s_1-\gg')/s_1}
  \\ &\sp\times
   \left(\int_{0}^t\int_{h_0}^1\int_{\TT^{2d}} 
  K_{h}(x-y) |Q_{\epsilon}^{x, x-y}|^{s_1}\, dxdy\frac{dh}{h}ds\right)^{\gg'/s_1}
  \end{align*}
with $s_1-\gg'\ge0$ since $\gg\ge s_1'$.
From~\eqref{P:las}, the above expression may be further bounded by
  \begin{align*}
  C
   \left( \int_{\epsilon}^{2\epsilon}r_{\tilde\eps}
   \,\frac{d\epsilon'}{\epsilon'}\right)^{\gg'/s_1} |\log h_0|^{(s_1-\gg')/s_1}.
  \end{align*}
Therefore, we obtain
  \begin{align*}
  I_{s, 2} \le \frac{D_3}{8} + C
   \left( \int_{\epsilon}^{2\epsilon}r_{\tilde\eps}
   \,\frac{d\epsilon'}{\epsilon'}\right)^{\gg'/s_1} |\log h_0|^{(s_1-\gg')/s_1}.
  \end{align*}
We estimate the term $I_{s, 3}$ next and rewrite it as
%  \begin{align*}
%  I_{s, 3} 
%  &=
%  \int_{0}^t\int_{\epsilon}^{2\epsilon}\int_{\tilde\eps}^1\int_{\TT^{3d}} K_{h}(x-y)f(x, y)w_{\epsilon, h}(x) 
%  L_{\epsilon'}(z) (\tilde P_{\epsilon}^{x-z}+\tilde P_{\epsilon}^{y-z})
%  \nn&\quad\indeq
%  |\rho_{\epsilon}(t, x-z)-\rho_{\epsilon}(t, y-z)| \,dxdydz\frac{dh}{h}\frac{d\epsilon'}{\epsilon'}ds
%%   \nn&=
%%  \int_{0}^t\int_{\epsilon}^{2\epsilon}\int_{\tilde\eps}^1\int_{\TT^{3d}} K_{h}(y)|\chi'(\rho^{x}-\rho^{x-y})|\overline{\rho_{\epsilon}}w_{\epsilon, h}(x) 
%%  L_{\epsilon'}(z)
%%  \nn&\quad
%%  \indeq Q_{\epsilon}^{x-z, x-y-z}
%%  \,dxdydz\frac{dh}{h}\frac{d\epsilon'}{\epsilon'}ds 
%%   \nn&\quad
%%  \int_{0}^t\int_{\epsilon}^{2\epsilon}\int_{\tilde\eps}^1\int_{\TT^{3d}} K_{h}(y)f(x, y)w_{\epsilon, h}(x) 
%%  L_{\epsilon'}(z)
%%   (\tilde P_{\epsilon}^{x-z}+\tilde P_{\epsilon}^{x-y-z})
%%     \nn&\quad\indeq
%%   |\rho_{\epsilon}(t, x-z)-\rho_{\epsilon}(t, x-y-z)| \,dxdydz\frac{dh}{h}\frac{d\epsilon'}{\epsilon'}ds .
%  \end{align*}
%which may be further expressed as
  \begin{align}   
  \label{I:s:3}
  I_{s, 3}&=
  \int_{0}^t\int_{\epsilon}^{2\epsilon}\int_{\tilde\eps}^1\int_{\TT^{4d}} K_{h}(x-y)K_{h}(x-\yy)f(x, y, \yy)w_{\epsilon, h}(x) 
  L_{\epsilon'}(z) 
   (\tilde P_{\epsilon}^{y-z}-\tilde P_{\epsilon}^{x-z})
      \nn&\qquad \ \ \ \ \ \ \ \ \ \ \ \ \ \ \ \ 
    |\rho_{\epsilon}(t, x-z)-\rho_{\epsilon}(t, y-z)|
   \,dx\,dy\,dz\,d\yy\frac{dh}{h}\frac{d\epsilon'}{\epsilon'}ds
   \nn&
   \quad 
   +
   2\int_{0}^t\int_{\epsilon}^{2\epsilon}\int_{\tilde\eps}^1\int_{\TT^{4d}} K_{h}(x-y)K_{h}(x-\yy)f(x, y, \yy)w_{\epsilon, h}(x) 
  L_{\epsilon'}(z)
   \tilde P_{\epsilon}^{x-z}
   \nn&\qquad \ \ \ \ \ \ \ \ \ \ \ \ \ \ \ \ 
   |\rho_{\epsilon}(t, x-z)-\rho_{\epsilon}(t, y-z)|
   \,dx\,dy\,dz\,d\yy\frac{dh}{h}\frac{d\epsilon'}{\epsilon'}ds.
  \end{align}
For the first term, we perform the change of variables and use H\"older's inequality to arrive at
  \begin{align*}   
   \int_{0}^t&\int_{\epsilon}^{2\epsilon}\int_{\tilde\eps}^1\int_{\TT^{4d}} K_{h}(x-y)K_{h}(x-\yy)f(x, y, \yy)w_{\epsilon, h}(x) 
  L_{\epsilon'}(z)
   (\tilde P_{\epsilon}^{y-z}-\tilde P_{\epsilon}^{x-z})
      \nn&\sp\dq\times
   |\rho_{\epsilon}(t, x-z)-\rho_{\epsilon}(t, y-z)|
   \,dx\,dy\,dz\,d\yy\frac{dh}{h}\frac{d\epsilon'}{\epsilon'}ds
   \nn&
   \le C
   \int_{0}^t\int_{\epsilon}^{2\epsilon}\int_{\tilde\eps}^1\int_{\TT^{2d}} K_{h}(y)K_{h}(\yy)
   \left\lVert  |\chi'|\overline{\rho_{\epsilon}} w_{\epsilon, h}^{1/\gg}(x, y) + |\chi'|\overline{\rho_{\epsilon}} w_{\epsilon, h}^{1/\gg}(x, \yy) \right \rVert_{L^{\gg}_x} 
   \nn&\sp\qquad \times
   \left\lVert  (\delta\tilde P_{\epsilon}^{x,y})|\delta\rho_{\epsilon}(x, y)|
     \right \rVert_{L^{\gg'}_x} 
   \,dy\,d\yy\frac{dh}{h}\frac{d\epsilon'}{\epsilon'}ds
  \end{align*}
where we also used the bound $w_{\epsilon, h}(x)\le 1$ and $\left\lVert  L_{\epsilon'}  \right \rVert_{L^1} =1 $ for any $\epsilon'>0$.
Using Young's inequality and Minkowsky's inequality, we get a further bound for the above term
  \begin{align*}
%   \nn&
%   \le
   \frac{\eta}{16}\int_{0}^t&\int_{\epsilon}^{2\epsilon}\int_{\tilde\eps}^1\int_{\TT^{d}} K_{h}(y)
   \left\lVert  |\chi'|\overline{\rho_{\epsilon}} w_{\epsilon, h}(x)  \right \rVert_{L^{\gg}_x} ^{\gg}
   \,dy\frac{dh}{h}\frac{d\epsilon'}{\epsilon'}ds
   \nn&   +\frac{C}{\eta}\int_{0}^t\int_{\epsilon}^{2\epsilon}\int_{\tilde\eps}^1\int_{\TT^{d}} K_{h}(y)
   \left\lVert  (\delta\tilde P_{\epsilon}^{x,y})|\delta\rho_{\epsilon}(x, y)|
     \right \rVert_{L^{\gg'}_x}^{\gg'} 
%     \nn&\indeq\qquad
   \,dy\frac{dh}{h}\frac{d\epsilon'}{\epsilon'}ds.
  \end{align*}
The first integral in the above bound is bounded by $D_3/16$. In order to estimate the second integral, 
we introduce the truncation 
function \[\tilde\phi_{\epsilon}^M(x,y)=\bar \phi(\rho_{\epsilon}^{x}/M)\bar \phi(\rho_{\epsilon}^{y}/M)\]
where $\bar\phi$ is a smooth function such that
  \begin{equation}
  \label{bar:phi}
  \bar\phi(s)=
  \begin{cases}
  1, & 0\le s \le 1, \\
  0, & s\ge2 \\
  \in[0, 1], & \text{otherwise}
  \end{cases}.
  \end{equation}
Then we have
  \begin{align*}  
  \frac{C}{\eta}&\int_{0}^t\int_{\epsilon}^{2\epsilon}\int_{\tilde\eps}^1\int_{\TT^{d}} K_{h}(y)
   \left\lVert  (\delta\tilde P_{\epsilon}^{x,y})|\delta\rho_{\epsilon}(x, y)|
     \right \rVert_{L^{\gg'}_x}^{\gg'} 
%     \nn&\indeq\qquad
   \,dy\frac{dh}{h}\frac{d\epsilon'}{\epsilon'}ds
   \nn&
   \le
  C\int_{0}^t\int_{\epsilon}^{2\epsilon}\int_{\tilde\eps}^1\int_{\TT^{2d}} K_{h}(y)
    |\delta\tilde P_{\epsilon}^{x,y}|^{\gg'}\tilde\phi_{\epsilon}^M(x,x-y)
%    \nn&\indeq\qquad
    |\delta\rho_{\epsilon}(x, y)|
     ^{\gg'}
   \,dxdy\frac{dh}{h}\frac{d\epsilon'}{\epsilon'}ds
   \nn& \quad+   
   C\int_{0}^t\int_{\epsilon}^{2\epsilon}\int_{\tilde\eps}^1\int_{\TT^{2d}} K_{h}(y)
    |\delta\tilde P_{\epsilon}^{x,y}|^{\gg'}(1-\tilde\phi_{\epsilon}^M(x,x-y))
%   \nn&\indeq\qquad
    |\delta\rho_{\epsilon}(x, y)|
     ^{\gg'} 
   \,dxdy\frac{dh}{h}\frac{d\epsilon'}{\epsilon'}ds
  \end{align*}
 Applying H\"older's inequality and using~\eqref{P:las}, we bound the truncated term as
  \begin{align*}  \frac{C}{\eta}&\int_{0}^t\int_{\epsilon}^{2\epsilon}\int_{\tilde\eps}^1\int_{\TT^{2d}} K_{h}(y)
    |\delta\tilde P_{\epsilon}^{x,y}|^{\gg'}\phi_{\epsilon}^M(x,x-y)
    |\delta\rho_{\epsilon}(x, y)|
     ^{\gg'}
   \,dxdy\frac{dh}{h}\frac{d\epsilon'}{\epsilon'}ds
   \nn&\le
CM^{\gg'}\Bigg(\int_{0}^t\int_{\epsilon}^{2\epsilon}\int_{\tilde\eps}^1\int_{\TT^{2d}} K_{h}(y)
   \,dxdy\frac{dh}{h}\frac{d\epsilon'}{\epsilon'}ds\Bigg)^{1-\gg'/\ss}
   \nn&\ \ \ \ \ \ \ \ \ \ \ \ \ \ \ \qquad
   \times\Bigg(\int_{0}^t\int_{\epsilon}^{2\epsilon}\int_{\tilde\eps}^1\int_{\TT^{2d}} K_{h}(y)
    |\tilde P_{\epsilon}^{x-y}-\tilde P_{\epsilon}^{x}|^{\ss}
  \,dxdy\frac{dh}{h}\frac{d\epsilon'}{\epsilon'}ds\Bigg)^{\gg'/\ss}
  \nn&\le C
  M^{\gg'}\left| \log h_0\right|^{1-\gg'/\ss} \Bigg(\int_{\epsilon}^{2\epsilon}r_{\tilde\eps}
  \frac{d\epsilon'}{\epsilon'}\Bigg)^{\gg'/\ss}.
  \end{align*} 
For the remainder term, (i.e., the term involving $1-\phi_{\epsilon}^M$), we use the simple relation
  \begin{equation*}
  (\{\rho(x)\ge M\}\cap\{\rho(z)\ge M\})^c = \{\rho(x)\ge M\}^c\cup\{\rho(z)\ge M\}^c
  \end{equation*} 
to obtain
  \begin{align*}  
&\int_{0}^t\int_{\epsilon}^{2\epsilon}\int_{\tilde\eps}^1\int_{\TT^{2d}} K_{h}(y)
    |\delta\tilde P_{\epsilon}^{x,y}|^{\gg'}(1-\tilde\phi_{\epsilon}^M(x,x-y))
    |\delta\rho_{\epsilon}(x, y)|
     ^{\gg'} 
   \,dxdy\frac{dh}{h}\frac{d\epsilon'}{\epsilon'}ds
     \nn&\qquad \le
  \int_{0}^t\int_{\epsilon}^{2\epsilon}\int_{\tilde\eps}^1\int_{\TT^{2d}} K_{h}(y)
    |\delta\tilde P_{\epsilon}^{x,y}|^{\gg'}
     (\mathbf{1}_{\{\rho^{x}\ge M\}} + \mathbf{1}_{\{\rho^{x-y}\ge M\}} )
    |\delta\rho_{\epsilon}(x, y)|
     ^{\gg'} 
   \,dxdy\frac{dh}{h}\frac{d\epsilon'}{\epsilon'}ds
  \end{align*}
By 
H\"older's and Young's inequalities, we get
  \begin{align*}  \frac{C}{\eta}&\int_{0}^t\int_{\epsilon}^{2\epsilon}\int_{\tilde\eps}^1\int_{\TT^{2d}} K_{h}(y)
    |\delta\tilde P_{\epsilon}^{x,y}|^{\gg'}(1-\phi_{\epsilon}^M(x,x-y))
%   \nn&\indeq\qquad
    |\delta\rho_{\epsilon}(x, y)|
     ^{\gg'} 
   \,dxdy\frac{dh}{h}\frac{d\epsilon'}{\epsilon'}ds
%   \nn&\le
%\frac{8}{\mu}\int_{0}^t\int_{\epsilon}^{2\epsilon}\int_{\tilde\eps}^1\int_{\TT^{d}} K_{h}(y)
%   \left\lVert   |\delta\tilde P_{\epsilon}^{x,y}|^{\gg'} \right \rVert_{L^{\ss/\gg'}_{x}} 
%   \nn&\indeq\qquad
%    \left\lVert (\mathbf{1}_{\{\rho^{x}\ge M\}} + \mathbf{1}_{\{\rho^{x-y}\ge M\}}) |\delta\rho_{\epsilon}(x, y)|
%     ^{\gg'}   \right \rVert_{L^{\left(\ss/\gg'\right)'}_{x}} 
%   \,dy\frac{dh}{h}\frac{d\epsilon'}{\epsilon'}ds
   \nn&\lesssim
\int_{0}^t\int_{\epsilon}^{2\epsilon}\int_{\tilde\eps}^1\int_{\TT^{2d}} K_{h}(y)
   |\tilde P_{\epsilon}^{x-y}-\tilde P_{\epsilon}^{x}|^{\ss}
   \,dxdy\frac{dh}{h}\frac{d\epsilon'}{\epsilon'}ds
   \nn&\qquad+
  \int_{0}^t\int_{\epsilon}^{2\epsilon}\int_{\tilde\eps}^1\int_{\TT^{2d}} K_{h}(y)
   \mathbf{1}_{\{\rho^{x}\ge M\}} \rho_{\epsilon}(t, x)
     ^{\ss\gg'/(\ss-\gg')}
   \,dxdy\frac{dh}{h}\frac{d\epsilon'}{\epsilon'}ds
   \nn&\lesssim
   r_{h_0} + M^{-(\gg-\ss\gg'/(\ss-\gg'))}|\log h_0|.
  \end{align*} 
Note for $\gg\ge 2s_0'$, one can easily check that $\gg-\ss\gg'/(\ss-\gg')>0$. 
For the second term in~\eqref{I:s:3}, we need to use the penalty function defined in~\eqref{D:eps}. More specifically, we need to extract 
an integral involving $K_h*\tilde P$ and estimate the remainder term with a quantity converging to $0$. To proceed, we rewrite this integral as
  \begin{align*}
  &2\int_{0}^t\int_{\epsilon}^{2\epsilon}\int_{\tilde\eps}^1\int_{\TT^{3d}} K_{h}(y)K_{h}(\yy)f(x, x-y, x-\yy)w^{1/\gg}_{\epsilon, h}(x) 
  L_{\epsilon'}(z)
  \tilde P_{\epsilon}^{x-z}w^{1/(1+l)}_{\epsilon, h}(x-z)
      \nn& \quad
     \times |\rho_{\epsilon}(t, x-z)-\rho_{\epsilon}(t, x-y-z)|
   \,dx\,dy\,dz\frac{dh}{h}\frac{d\epsilon'}{\epsilon'}ds
   \nn&\quad+
   2\int_{0}^t\int_{\epsilon}^{2\epsilon}\int_{\tilde\eps}^1\int_{\TT^{3d}} K_{h}(y)K_{h}(\yy)f(x, x-y, x-\yy)w^{1/\gg}_{\epsilon, h}(x) 
   L_{\epsilon'}(z)\nn
   &(w^{1/(1+l)}_{\epsilon, h}(x)-w^{1/(1+l)}_{\epsilon, h}(x-z))\, 
   \tilde P_{\epsilon}^{x-z}|\rho_{\epsilon}(t, x-z)-\rho_{\epsilon}(t, x-y-z)|
   \,dx\,dy\,dz\frac{dh}{h}\frac{d\epsilon'}{\epsilon'}ds
   \nn&=
   I_{G} + \text{Diff}_1.
  \end{align*}
The treatment of $I_{G}$ is slightly difficult. Similar to previous calculations in~\eqref{bar:I:s:1}, we change variable and use H\"older's inequality to obtain
  \begin{align*}   
  |I_G| &\le
   \frac{\eta}{16} \int_{0}^t\int_{\epsilon}^{2\epsilon}\int_{\tilde\eps}^1\int_{\TT^{2d}} K_{h}(y)\chi
 \overline{\rho_{\epsilon}^{\gg}}w_{\epsilon, h}(x) 
   \,dxdy\frac{dh}{h}\frac{d\epsilon'}{\epsilon'}ds
   \nn&\quad
   +
   \frac{C}{\eta} \int_{0}^t\int_{\epsilon}^{2\epsilon}\int_{\tilde\eps}^1\int_{\TT^{2d}} K_{h}(y)
   \tilde P_{\epsilon}^{1+l}(x)w_{\epsilon, h}(x)|\delta\rho_{\epsilon}(x, y)|^{1+l}
%   \nn&\qquad
   \,dxdy\frac{dh}{h}\frac{d\epsilon'}{\epsilon'}ds.
  \end{align*}
The first term in the above inequality is bounded by $D_3/16$. To estimate the second term, we need to introduce $K_h*G$ to use the penalty function:
  \begin{align*}   
  \frac{C}{\eta}& \int_{0}^t\int_{\epsilon}^{2\epsilon}\int_{\tilde\eps}^1\int_{\TT^{2d}} K_{h}(y)
   \tilde P_{\epsilon}^{1+l}(x)w_{\epsilon, h}(x)\chi
   \,dxdy\frac{dh}{h}\frac{d\epsilon'}{\epsilon'}ds
   \nn&
   \le
   \frac{C}{\eta} \int_{0}^t\int_{\epsilon}^{2\epsilon}\int_{\tilde\eps}^1\int_{\TT^{3d}} K_{h}(y)K_h(z)
   \left|\tilde P_{\epsilon}(x)-\tilde P_{\epsilon}(x-z)\right|^{1+l}
   w_{\epsilon, h}(x)\chi
   \,dx\,dy\,dz\frac{dh}{h}\frac{d\epsilon'}{\epsilon'}ds
   \nn&\quad 
   +
   \frac{C}{\eta} \int_{0}^t\int_{\epsilon}^{2\epsilon}\int_{\tilde\eps}^1\int_{\TT^{3d}} K_{h}(y)K_h(z)
   |\tilde P_{\epsilon}|^{1+l}(x-z)
   w_{\epsilon, h}(x)\chi
   \,dx\,dy\,dz\frac{dh}{h}\frac{d\epsilon'}{\epsilon'}ds
  \end{align*}
where the last term may be bounded by $\bep D_2$ with $\bep$ being arbitrarily small provided $\lambda$ is sufficiently large. 
By H\"older we bound the first term as
  \begin{align*}
     \frac{C}{\eta}& \int_{0}^t\int_{\epsilon}^{2\epsilon}\int_{\tilde\eps}^1\int_{\TT^{3d}} K_{h}(y)K_h(z)
   \left|\tilde P_{\epsilon}(x)-\tilde P_{\epsilon}(x-z)\right|^{1+l}
   w_{\epsilon, h}(x)\chi
   \,dx\,dy\,dz\frac{dh}{h}\frac{d\epsilon'}{\epsilon'}ds
   \nn&
   \le C
   \int_{0}^t\int_{\epsilon}^{2\epsilon}\int_{\tilde\eps}^1\int_{\TT^{2d}} K_{h}(y)K_h(z)
   \left\lVert  \left|\tilde P_{\epsilon}(x)-\tilde P_{\epsilon}(x-z)\right|^{1+l}  \right \rVert_{L^{\ss/(1+l)}_{x}} 
   \nn&\indeq\sp\times
   \left\lVert  \chi  \right \rVert_{L^{\ss/(\ss-(1+l))}_{x}}
   \,dz\,dy\frac{dh}{h}\frac{d\epsilon'}{\epsilon'}ds
    \end{align*}
Note that for $\gg\ge 2s_0'$ we always have $\ss(1+l)/(\ss-(1+l)) \le\gg$. Hence, we get
  \begin{equation*}
     \left\lVert  \chi  \right \rVert_{L^{\ss/(\ss-(1+l))}_{x}} \le C.
  \end{equation*}
Therefore, we have a further bound
   \begin{align*}
   \frac{C}{\eta}& \int_{0}^t\int_{\epsilon}^{2\epsilon}\int_{\tilde\eps}^1\int_{\TT^{3d}} K_{h}(y)K_h(z)
   \left|\tilde P_{\epsilon}(x)-\tilde P_{\epsilon}(x-z)\right|^{1+l}
   w_{\epsilon, h}(x)\chi
   \,dx\,dy\,dz\frac{dh}{h}\frac{d\epsilon'}{\epsilon'}ds
   \nn&
   \le C
   \int_{0}^t\int_{\epsilon}^{2\epsilon}\int_{\tilde\eps}^1\int_{\TT^{d}} K_h(z)
   \left\lVert \tilde P_{\epsilon}(x)-\tilde P_{\epsilon}(x-z)  \right \rVert_{L^{\ss}_{x}} ^{1+l}
   \,dz\frac{dh}{h}\frac{d\epsilon'}{\epsilon'}ds
   \nn&
   \le C
   \left( \int_{\epsilon}^{2\epsilon}r_{\tilde\eps}^1
   \,\frac{d\epsilon'}{\epsilon'}\right)^{(1+l)/\ss} |\log h_0|^{(\ss-1-l)/\ss}.
  \end{align*}
By H\"older's inequality, the Diff$_1$ term is estimated similarly to~\eqref{diff} as 
  \begin{align*}
  \text{Diff}_1 &\le 
   \frac{\eta}{16} \int_{0}^t\int_{\epsilon}^{2\epsilon}\int_{\tilde\eps}^1\int_{\TT^{2d}} K_{h}(y)\chi
 \overline{\rho_{\epsilon}^{\gg}}w_{\epsilon, h}(x) 
   \,dxdy\frac{dh}{h}\frac{d\epsilon'}{\epsilon'}ds
   \nn&\quad
   +C
  \int_{0}^t\int_{\epsilon}^{2\epsilon}\int_{\tilde\eps}^1\int_{\TT^{2d}} K_{h}(y)
   \left\lVert  \rho^{1+l}  \right \rVert_{L_x^{\ss/(\ss-(1+l))}}
  L_{\epsilon'}(z) \left(\frac{z}{h}\right)^{1/(1+l)}
   \nn&\qquad \times
   \left\lVert  |\tilde P_{\epsilon}^{x-z}|^{1+l}  \right \rVert_{L_x^{\ss/(1+l)}} 
   \,dy\,dz\frac{dh}{h}\frac{d\epsilon'}{\epsilon'}ds
   \nn&\le 
   \frac{1}{16}D_3 +C
  \int_{0}^t\int_{\epsilon}^{2\epsilon}\int_{\tilde\eps}^1\int_{\TT^{2d}} K_{h}(y)
  L_{\epsilon'}(z) \left(\frac{z}{h}\right)^{1/(1+l)}
   \,dy\,dz\frac{dh}{h}\frac{d\epsilon'}{\epsilon'}ds
   \le    \frac{1}{16}D_3 + C
  \end{align*}
provided $\gg>2s_0'$. Collecting all the estimates of $I_{s, 1}$, $I_{s, 2}$, with $I_{s, 3}$ and optimizing in $M$ concludes the proof. 
  \end{proof}

\subsection{Term $I_4$}
Before giving the bound for the integral terms $I_4$ and $I_5$, we introduce the following lemma needed for the treatment of the effective viscous flux $F = \Delta^{-1}\div(\partial_t (\rho_{\epsilon} u_{\epsilon}) + \div(\rho_{\epsilon} u_{\epsilon}\otimes u_{\epsilon}))$. We refer the readers to~\cite{BreJab18} for a proof of this result.
\begin{lemma}
\label{vis:flu}
Let $F$ be the effective viscous flux introduced above. Assume that $(\rho_{\epsilon}, u_{\epsilon})$ is a solution of the 
system~\eqref{app1:00}--\eqref{app1:01} satisfying the bound \eqref{pri:app:1} with $\gg>d/2$. Suppose that $\Phi\in L^\infty([0, T]\times\TT^{2d})$ and that
  \begin{align*}   
  C_{\Phi}:&=\left\Vert \int_{\TT^{d}} K_h(x-y)\Phi(t, x, y)\,dy \right\Vert_{W^{1,1}(0, T; W^{-1, 1}_x(\TT^{d}))}
  \nn&\quad+
  \left\Vert \int_{\TT^{d}} K_h(x-y)\Phi(t, x, y)\,dx \right\Vert_{W^{1,1}(0, T; W^{-1, 1}_y(\TT^{d}))} <\infty,
  \end{align*}
then there exists $\theta>0$ such that
  \begin{align*}
     \int_{0}^t\int_{\TT^{2d}} &K_h(x-y)\Phi(t, x, y)( F (t, x)- F(t, y))\,dx\,dy\,dt
  \nn&
  \lesssim
  h^\theta(C_{\Phi}+\Vert\Phi \Vert_{L^{\infty}((0, T)\times \TT{2d})})
  \end{align*}
holds, where the implicit constant in $\lesssim$ is independent of $\eps$.
\end{lemma}

Next we estimate $I_4$ in the lemma below. We use $\theta$ to denote a parameter between $0$ and $1$ which may be different from line to line.
\begin{lemma}
\label{T:I:4}
Let $I_4$ be defined by~\eqref{I:4}. Under the assumptions of Lemma~\ref{T:I:P}, it follows
  \begin{align*}
  I_4 \le C  
   + C \left(\int_{\epsilon}^{2\epsilon} r_{\max(h_0,\epsilon')}\frac{d\epsilon'}{\epsilon'}\right)^{\bar\theta}|\log(h_0)|^\theta
   + C\int_0^t T_{h_0, \eps}(s)\,ds -D_1-D_2 - \frac{7D_3}{8}.
  \end{align*}
with $D_1$, $D_2$, and $D_3$ given by~\eqref{D:1}, \eqref{D:2}, and \eqref{D:3} respectively.
%  \begin{equation*}  
%  D_3=
%  -\mu(1+l)\int_{0}^t\int_{h_0}^1\int_{\TT^{2d}} Z_{\epsilon, h}^{x, y}\chi(\delta\rho_{\epsilon})\overline{\rho_{\epsilon}^{\gg}}(x)
%  \,dxdy\frac{dh}{h}ds,
%  \end{equation*}
Here $0<\bar\theta$, $0< \theta<1$, and $t\le T$, where $T$ can be any positive number and the implicit constant may depend on time $T$.
\end{lemma}
  \begin{proof}
  We first recall  
  \begin{equation*}
  I_4 = 
  -\frac{1}{2}\int_{0}^t\int_{h_0}^1\int_{\TT^{2d}} K_h(x-y)W_{\epsilon, h}^{x, y}\chi'\overline\rho\delta{(\div u_{\epsilon}})(x)   %4
   \,dxdy\frac{dh}{h}ds.
  \end{equation*}  
  We proceed by getting a representation formula for $\div u_{\epsilon}$ from~\eqref{app1:01}
  \begin{align}
  \label{div:u}
  \div u_{\epsilon} =  \eta \rho_{\epsilon}^{\gg} + \cl_{\epsilon}*P +  F 
  \end{align}
where $F$ is the effective viscous flux:
  \begin{equation*}
     F = \Delta^{-1}\div F (\partial_t (\rho_{\epsilon} u_{\epsilon}) + \div(\rho_{\epsilon} u_{\epsilon}\otimes u_{\epsilon})).
  \end{equation*}
Then the term $I_4$ may be rewritten as
  \begin{align*}
  I_4
   &=
   -\frac{1}{2}\int_{0}^t\int_{h_0}^1\int_{\TT^{2d}} K_h(x-y)W_{\epsilon, h}^{x, y}\chi'\overline{\rho_{\epsilon}}\delta(\eta \rho_{\epsilon}^{\gg} 
   + \cl_{\epsilon}*P +  F )(x)
   \,dxdy\frac{dh}{h}ds
   \nn&= I_{4, 1} + I_{4, 2} + I_{4, 3}
  \end{align*}
with $I_{4, 1}, \ I_{4, 2},$ and $I_{4, 3}$ being the integrals corresponding to the three terms in the parentheses of the above formula.  
Noting that 
  \begin{align*}   
  \eta\chi'\overline{\rho_{\epsilon}}\delta(\rho_{\epsilon}^{\gg}) &\ge \eta\chi'\overline{\rho_{\epsilon}}(\rho_{\epsilon}(x)-\rho_{\epsilon}(y))(\rho_{\epsilon}^{\gg-1}(x) + \rho_{\epsilon}^{\gg-1}(y))
  \nn&=
\eta(1+l)\chi(\delta\rho_{\epsilon})\overline{\rho_{\epsilon}}(\rho_{\epsilon}^{\gg-1}(x) + \rho_{\epsilon}^{\gg-1}(y))
  \nn&\ge \eta(1+l) \chi(\delta\rho_{\epsilon})\overline{\rho_{\epsilon}^{\gg}}
  \end{align*}
we arrive at
  \begin{align}
  \label{dif:gam}
  I_{4,1} \le -\eta(1+l)\int_{0}^t\int_{h_0}^1\int_{\TT^{2d}} K_h(x-y)W_{\epsilon, h}^{x, y}\chi(\delta\rho_{\epsilon})\overline{\rho_{\epsilon}^{\gg}}(x)
  \,dxdy\frac{dh}{h}ds
  \end{align}
which serves as a penalization.
To bound the term $I_{4, 2}$, we rewrite it as 
  \begin{align*}
  I_{4, 2} &= 
  -\frac{1}{2}\int_{0}^t\int_{h_0}^1\int_{\TT^{2d}} K_h(x-y)W_{\epsilon, h}^{x, y}\chi'\overline{\rho_{\epsilon}}\delta
   (\cl_{\epsilon}*P) (x)
   \,dxdy\frac{dh}{h}ds
   \\&=-
   \int_{0}^t\int_{h_0}^1\int_{\TT^{3d}} K_h(x-y)K_h(x-\yy)w_{\epsilon, h}^{x}\chi'\overline{\rho_{\epsilon}}\delta
   (\cl_{\epsilon}*P) (x)
   \,dx\,dy\,d\yy\frac{dh}{h}ds.
  \end{align*}
Let $f(x, y, \yy) = \chi'(\delta\rho(x, x-y))\overline{\rho_{\epsilon}}(x, x-y)$, then it is straightforward to check that $f$ satisfies the condition~\eqref{f:lp}. Appealing to the Lemma~\ref{T:I:P}, we arrive at
  \begin{align*}
  |I_{4, 2}| \le  C  
   + C \left(\int_{\epsilon}^{2\epsilon} r_{\max(h_0,\epsilon')}\frac{d\epsilon'}{\epsilon'}\right)^{\bar\theta}|\log(h_0)|^\theta
   + C\int_0^t T_{h_0, \eps}(s)\,ds+ \bep D_2 + \frac{3D_3}{8}.
  \end{align*}
%Here we give a uniform estimate in $\epsilon$ of this term, which may be divided into two cases: $\epsilon<h$ and $\epsilon\ge h$:
%  \begin{align*}
%     I_{4, 2} &= -\frac{1}{\log 2}\int_{0}^t\int_{\epsilon}^{2\epsilon}\int_{h_0}^1\int_{\TT^{2d}} K_{h}(x-y)\chi'\overline{\rho_{\epsilon}}w_{\epsilon, h}(x) 
%  \delta(L_{\epsilon} *P)
%   \,dxdy\frac{dh}{h}\frac{d\epsilon'}{\epsilon'}ds 
%    \nn&=
%   -\frac{1}{\log 2}\int_{0}^t\int_{\epsilon}^{2\epsilon}\int_{h_0}^1\int_{\TT^{2d}} (\mathbf{1}_{\epsilon'\ge h} + \mathbf{1}_{\epsilon'<h})K_{h}(x-y)\chi'\overline{\rho_{\epsilon}}w_{\epsilon, h}(x) 
%  \delta(L_{\epsilon} *P)
%   \,dxdy\frac{dh}{h}\frac{d\epsilon'}{\epsilon'}ds \nn&=
%    I_{b}+I_{s}
%  \end{align*} 
%where $I_b$ and $I_s$ are corresponding to the integrals with characteristic functions $\mathbf{1}_{\epsilon'\ge h}$ and $ \mathbf{1}_{\epsilon'<h}$ in them respectively. As we see below, the term $I_b$ is easier to treat since in this case the $K_h$ is the mollifier playing the key role, which is more consistent with the whole compactness argument. While for term $I_s$, we need to take the advantage of regularity of the weight function to create smallness with $\epsilon'$.
Finally, we deal with the effective viscous flux term $I_{4, 3}$, which is rewritten as 
  \begin{align}
  \label{I:4:3}
     I_{4, 3} &= -\frac{1}{2}\int_{0}^t\int_{h_0}^1\int_{\TT^{2d}} \phi_{\epsilon}^MK_h(x-y)W_{\epsilon, h}^{x, y}\chi'\overline\rho\delta F (x)
   \,dxdy\frac{dh}{h}ds\nn&
   \quad-\frac{1}{2}\int_{0}^t\int_{h_0}^1\int_{\TT^{2d}} (1-\phi_{\epsilon}^M)K_h(x-y)W_{\epsilon, h}^{x, y}\chi'\overline\rho\delta F (x)
   \,dxdy\frac{dh}{h}ds.
  \end{align}
For the second integral, we use the uniform integrability of $\rho_{\epsilon}$ and $\div u_{\epsilon}$ to obtain
  \begin{align*}   
  &\left|\int_{0}^t\int_{h_0}^1\int_{\TT^{2d}} (1-\phi_{\epsilon}^M)K_h(x-y)W_{\epsilon, h}^{x, y}\chi'\overline\rho\delta F (x)
   \,dxdy\frac{dh}{h}ds\right|
   \nn&\quad=
   \left|\int_{0}^t\int_{h_0}^1\int_{\TT^{2d}} (1-\phi_{\epsilon}^M)K_h(y)W_{\epsilon, h}^{x, x-y}\chi'\overline\rho\delta F (x)
   \,dxdy\frac{dh}{h}ds\right|
   \nn&\quad\lesssim
   \int_{0}^t\int_{h_0}^1\int_{\TT^{d}} K_h(y)\Vert (1-\phi_{\epsilon}^M)\chi'\overline\rho\Vert_{L^{p/(p-\gg)}_x}\Vert \delta F (x)\Vert_{L^{p/\gg}_x}
   \,dy\frac{dh}{h}ds
   \nn&\quad\lesssim |\log h_0|M^{-\theta}
  \end{align*}
with some $1>\theta>0$ and $p=\gg+2\gg/d-1 - 1/\lambda_0$ for a sufficiently large constant $\lambda_0$. Note here $(1+l)p/(p-\gg) < \gg$ since we require $\gg>2+d.$
While for the first integral in~\eqref{I:4:3}, we need to use Lemma~\ref{vis:flu} with 
  \begin{equation*}
     \Phi= W_{\epsilon, h}^{x, y}\chi'\overline\rho \phi_{\epsilon}^M.
  \end{equation*}
Obviously we have  that $\Vert \Phi\Vert_{L^{\infty}} \lesssim M^{1+l}$.
In view of the system~\eqref{app1:00}--\eqref{app1:01}, we get an equation for $\Phi$ as
  \begin{align*}   
  \partial_{t}\Phi + \div_x \left(\Phi u_{\epsilon}^{x}\right) + \div_y \left(\Phi u_{\epsilon}^{y}\right)
  =
  f^{x, y}_{\epsilon, 1}\div_xu_{\epsilon}^{x} + f^{x, y}_{\epsilon, 2}\div_y u_{\epsilon}^{y}
  +
  f^{x, y}_{\epsilon, 3}\frac{1}{\lambda}D_{\epsilon}^{x} + f^{x, y}_{\epsilon, 4}\frac{1}{\lambda}D_{\epsilon}^{y}
  \end{align*}
where $D_{\epsilon}$ is the penalization introduced in~\eqref{D:eps} and $f^{x, y}_{\epsilon, i}$ are polynomials of $\rho_{\epsilon},$ $w_{\epsilon}$, $\phi_{\epsilon}^M$, and derivatives of $\phi_{\epsilon}^M$ for $i=1, 2, 3, 4$. 
Noting that
  \begin{equation*}   
  \Vert f^{x, y}_{\epsilon, i}\Vert_{L^{\infty}} \lesssim M^{1+l} \ \ \ \ \ \text{for} \ i=1, 2, 3,4,
  \end{equation*}
it is not difficult to get that 
  \begin{equation*}
     C_\Phi \lesssim M^{1+l}
  \end{equation*}
where $C_\Phi$ is defined in Lemma~\ref{vis:flu}.
  Hence Lemma~\ref{vis:flu} implies
  \begin{equation*}
     \left|\int_{0}^t\int_{h_0}^1\int_{\TT^{2d}} \phi_{\epsilon}^MK_h(x-y)W_{\epsilon, h}^{x, y}\chi'\overline\rho\delta F (x)
   \,dxdy\frac{dh}{h}ds\right|
   \lesssim
   M^{1+l}.
  \end{equation*}
Optimizing the bound in $M$ gives
  \begin{equation*}  
  I_{4,3} \lesssim |\log h_0|^{\theta}
  \end{equation*}
for some $0<\theta<1$. The proof is concluded by collecting the estimates for $I_{4,1}$, $I_{4,2}$, and $I_{4,3}$.
\end{proof}
\subsection{Term $I_5$} We give the estimate for $I_5$ in this subsection.
\begin{lemma}
\label{T:I:5}
Let $I_5$ be defined by~\eqref{I:5}. Under the assumptions in Lemma~\ref{T:I:P}, we have 
  \begin{align*}
  I_5 \le C  
   + C \left(\int_{\epsilon}^{2\epsilon} r_{\max(h_0,\epsilon')}\frac{d\epsilon'}{\epsilon'}\right)^{\bar\theta}|\log(h_0)|^\theta
   + C\int_0^t T_{h_0, \eps}(s)\,ds+ \bep D_2 + \frac{D_3}{2}
  \end{align*}
with $D_2$ and $D_3$ given by~\eqref{D:2} and \eqref{D:3} respectively, 
%  \begin{equation*}  
%  D_3=
%  -\mu(1+l)\int_{0}^t\int_{h_0}^1\int_{\TT^{2d}} Z_{\epsilon, h}^{x, y}\chi(\delta\rho_{\epsilon})\overline{\rho_{\epsilon}^{\gg}}(x)
%  \,dxdy\frac{dh}{h}ds,
%  \end{equation*}
for some $0<\bar\theta$, $0< \theta<1$, and $t\le T$, where $T$ can be any positive number and the implicit constant may depend on time $T$.
\end{lemma}
\begin{proof} We recall
  \begin{align*}
  \label{}
  I_5 = 
  \int_{0}^t\int_{h_0}^1\int_{\TT^{2d}}K_h(x-y)W_{\epsilon, h}^{x, y}\left(\chi-\frac{1}{2}\chi'\delta\rho\right)\overline{\div_xu_{\epsilon}}(x) 
   \,dxdy\frac{dh}{h}ds.
  \end{align*}
By the definition of $\chi$ in~\eqref{chi}, the term $I_5$ may be rewritten as 
  \begin{align*}
     I_5&=\frac{1-l}{2}
  \int_{0}^t\int_{h_0}^1\int_{\TT^{2d}}K_h(x-y)W_{\epsilon, h}^{x, y}\chi\overline{\div_xu_{\epsilon}}(x) \,dxdy\frac{dh}{h}ds
  \nn&=
  \frac{1-l}{2}
  \int_{0}^t\int_{h_0}^1\int_{\TT^{2d}}K_h(x-y)W_{\epsilon, h}^{x, y}\chi\overline{(P_{art,\eta}(\rho_\eps(x)) 
    + \cl_{\epsilon}*P(x) +  F (x))} \,dxdy\frac{dh}{h}ds\nn&
%  \quad + \frac{1-l}{2}\int_{0}^t\int_{h_0}^1\int_{\TT^{2d}}Z_{\epsilon, h}^{x, y}\chi\overline{\cl_{\epsilon}*P}(x) \,dxdy\frac{dh}{h}ds\nn&
%  \quad + \frac{1-l}{2}\int_{0}^t\int_{h_0}^1\int_{\TT^{2d}}Z_{\epsilon, h}^{x, y}\chi\overline{\Delta^{-1}\div F}(x) \,dxdy\frac{dh}{h}ds
  = I_{5, 1} + I_{5, 2} + I_{5, 3}.
  \end{align*}
Note that since $P_{art,\eta}(\rho)\leq C\,\rho^{\gg}$, the term $I_{5, 1}$ may be absorbed by the term $D_3/2$ in~\eqref{D:3}. Next we treat $I_{5, 2}$ as
  \begin{align*}
     I_{5, 2} &=
  \frac{1-l}{2}\int_{0}^t\int_{h_0}^1\int_{\TT^{2d}}K_{h}(x-y)w_{\epsilon, h}(x)\chi(\delta\rho)(\cl_{\epsilon}*P(x)+\cl_{\epsilon}*P(y)) \,dxdy\frac{dh}{h}ds
  \nn&=
  (1-l)\int_{0}^t\int_{h_0}^1\int_{\TT^{2d}}K_{h}(x-y)w_{\epsilon, h}(x)\chi(\delta\rho)\cl_{\epsilon}*P(x) \,dxdy\frac{dh}{h}ds
  \nn&\quad
  -
  \frac{1-l}{2}\int_{0}^t\int_{h_0}^1\int_{\TT^{2d}}K_{h}(x-y)w_{\epsilon, h}(x)\chi(\delta\rho)\delta(\cl_{\epsilon}*P)(x) \,dxdy\frac{dh}{h}ds.
  \end{align*}
Since the second integral in the right side of the last equality is already estimated in $I_4$, we only need to consider the first integral.
We need to use the penalization $D_2$ defined in~\eqref{D:2} to control the main contribution of this term. Note
  \begin{align*}
     &\int_{0}^t\int_{h_0}^1\int_{\TT^{2d}}K_{h}(x-y)w_{\epsilon, h}(x)\chi(\delta\rho)\cl_{\epsilon}*P(x) \,dxdy\frac{dh}{h}ds
  \nn&\quad=
  \int_{0}^t\int_{h_0}^1\int_{\TT^{3d}}K_{h}(x-y)K_h(x-z)w_{\epsilon, h}(x)\chi(\delta\rho)\cl_{\epsilon}*P(x) \,dx\,dy\,dz\frac{dh}{h}ds
  \nn&\quad=
  \int_{0}^t\int_{h_0}^1\int_{\TT^{3d}}K_{h}(x-y)K_h(x-z)w_{\epsilon, h}(x)\chi(\delta\rho)(\cl_{\epsilon}*P(x)-\cl_{\epsilon}*P(z)) \,dx\,dy\,dz\frac{dh}{h}ds
  \nn&\quad\quad+
  \int_{0}^t\int_{h_0}^1\int_{\TT^{3d}}K_{h}(x-y)K_h(x-z)w_{\epsilon, h}(x)\chi(\delta\rho)\cl_{\epsilon}*P(z) \,dx\,dy\,dz\frac{dh}{h}ds
  \end{align*}
where the last integral is bounded by $\bep D_2$.
We switch variables to rewrite the first integral as
  \begin{align*}
  &\int_{0}^t\int_{h_0}^1\int_{\TT^{3d}}K_{h}(x-y)K_h(x-z)w_{\epsilon, h}(x)\chi(\delta\rho)(\cl_{\epsilon}*P(x)-\cl_{\epsilon}*P(z)) \,dx\,dy\,dz\frac{dh}{h}ds
  \\&=
  \int_{0}^t\int_{h_0}^1\int_{\TT^{3d}}K_{h}(x-y)K_h(x-\yy)w_{\epsilon, h}(x)\chi(\delta\rho(x, x-\yy))(\cl_{\epsilon}*P(x)-\cl_{\epsilon}*P(y)) \\
  &\sp\sp\,dx\,dy\,d\yy\frac{dh}{h}ds.
  \end{align*}
Let $f(x, y, \yy) = \chi(\delta\rho(x, x-\yy))$, then it is easy to check that~\eqref{f:lp} holds. Using Lemma~\ref{T:I:P}, we arrive at
  \begin{align*}
     &
  \int_{0}^t\int_{h_0}^1\int_{\TT^{3d}}K_{h}(x-y)K_h(x-\yy)w_{\epsilon, h}(x)\chi(\delta\rho(x, x-\yy))(\cl_{\epsilon}*P(x)-\cl_{\epsilon}*P(y))\\ &\sp\sp dx\,dy\,d\yy\frac{dh}{h}ds
  \nn&\quad\quad\le
  C  
   + C \left(\int_{\epsilon}^{2\epsilon} r_{\max(h_0,\epsilon')}\frac{d\epsilon'}{\epsilon'}\right)^{\bar\theta}|\log(h_0)|^\theta
   + C\int_0^t T_{h_0, \eps}(s)\,ds+ \bep D_2 + \frac{3D_3}{8}.
  \end{align*}
At last, we treat the effective viscous flux term as
  \begin{align*}
  I_{5, 3}
  &=
  (1-l)\int_{0}^t\int_{h_0}^1\int_{\TT^{2d}}K_h(x-y)w_{\epsilon, h}^{x}\chi\overline{ F }(x) \,dxdy\frac{dh}{h}ds
  \nn&=
  (1-l)\int_{0}^t\int_{h_0}^1\int_{\TT^{2d}}K_h(x-y)w_{\epsilon, h}^{x}\chi( F (y)- F (x)) \,dxdy\frac{dh}{h}ds
  \nn&\qquad+
  2(1-l)\int_{0}^t\int_{h_0}^1\int_{\TT^{2d}}K_h(x-y)w_{\epsilon, h}^{x}\chi F (x) \,dxdy\frac{dh}{h}ds.
  \end{align*}
Note that the first integral is already treated in $I_{4, 2}$, and we now deal with the second integral as
  \begin{align*}  
  &\int_{0}^t\int_{h_0}^1\int_{\TT^{2d}}K_h(x-y)w_{\epsilon, h}^{x}\chi F (x) \,dxdy\frac{dh}{h}ds
%  \nn&\quad =
%  \int_{0}^t\int_{h_0}^1\int_{\TT^{3d}}K_h(x-y)K_h(x-z)w_{\epsilon, h}^{x}\chi F (x) \,dxdydz\frac{dh}{h}ds
  \nn&\quad =
  \int_{0}^t\int_{h_0}^1\int_{\TT^{3d}}K_h(x-y)K_h(x-z)w_{\epsilon, h}^{x}\chi( F (x)- F (z)) \,dx\,dy\,dz\frac{dh}{h}ds
  \nn&\quad\quad+
  \int_{0}^t\int_{h_0}^1\int_{\TT^{3d}}K_h(x-y)K_h(x-z)w_{\epsilon, h}^{x}\chi F (z) \,dx\,dy\,dz\frac{dh}{h}ds.
  \end{align*}
For the first integral, by similar argument as in the treatment of $I_{4, 3}$, we arrive at
  \begin{align*}   
  \int_{0}^t\int_{h_0}^1\int_{\TT^{3d}}K_h(x-y)K_h(x-z)w_{\epsilon, h}^{x}\chi( F (x)-\Delta^{-1}\div& F(z)) \,dx\,dy\,dz\frac{dh}{h}ds\nn&
  \lesssim
  |\log h_0|^{\theta}
  \end{align*}
for some $0<\theta<1$. While for the second integral, we use the formula \eqref{div:u} to obtain
  \begin{align*}
     &\int_{0}^t\int_{h_0}^1\int_{\TT^{3d}}K_h(x-y)K_h(x-z)w_{\epsilon, h}^{x}\chi F (z) \,dx\,dy\,dz\frac{dh}{h}ds
  \nn\quad&\le
  \int_{0}^t\int_{h_0}^1\int_{\TT^{3d}}K_h(x-y)K_h(x-z)w_{\epsilon, h}^{x}\chi(\delta\rho_{\epsilon})|\div u_{\epsilon}|(z) \,dx\,dy\,dz\frac{dh}{h}ds
  \end{align*}
which is bounded by $\bep D_2$.
Collecting all the estimate and optimizing in $M$ concludes the proof.
%  \begin{align*} 
%  T_{h_0, \eps}(t) &\lesssim  T_{h_0, \eps}(0) +|\log h_0|^{\theta} -D_1-D_2 - D_3 + \int_{0}^t T_{h_0, \eps}(s)\,ds
%  \nn&\quad+M^{1+l} \int_{\epsilon}^{2\epsilon} r_{\tilde\eps}\frac{d\epsilon'}{\epsilon'}|\log(h_0)|
%   \nn&\quad
%  +  M^{-(\gg-r_0(1+l))}\mu^{-1}|\log(h_0)| +  \int_{\epsilon}^{2\epsilon}
%  r_{\tilde\eps}
% \frac{d\epsilon'}{\epsilon'}  
% \nn&\quad
% + M^{\gg'}\frac{8}{\mu}\left| \log h_0\right|^{1-\gg'/\ss} \Bigg(\int_{\epsilon}^{2\epsilon}r_{\tilde\eps}
%  \frac{d\epsilon'}{\epsilon'}\Bigg)^{\gg'/\ss}
%  \nn&\quad
%  +
%  \frac{8}{\mu}r_{h_0} + \frac{8}{\mu^2} M^{\gg-\ss\gg'/(\ss-\gg')}|\log h_0|
%  \nn&\quad
%  +
%  \mu^{-(1+(1+l)/\gg)} \left( \int_{\epsilon}^{2\epsilon}r_{\tilde\eps}^1
%   \,\frac{d\epsilon'}{\epsilon'}\right)^{(1+l)/\ss} |\log h_0|^{(\ss-1-l)/\ss}
%  \nn&\quad
%  +  |\log h_0|M^{-\theta_0}
%  \end{align*}
%from where optimizing in the parameter $M$ and a simple Gronwall argument conclude the proof.
  \end{proof}

 \subsection{Compactness argument} 
  \begin{proof}[Proof of Theorem~\ref{thepstomu}]
%  We prove the existence first. Since $\cl_{\epsilon}*P(t, x, \rho(x))$ is smooth for any $\rho(x)\in L^\gamma$ for every  $\epsilon>0$. By the result of~\cite{BreJab18}, applying the Schauder--Tychonoff fixed point theorem, we get that there is a global weak solution for the system\eqref{app:00}--\eqref{app:01} for any $\epsilon, \mu>0$. We fix a $\mu>0$ and prove that we may send $\epsilon$ to $0$ to get a global solution to the system~\ref{app1:00}--\ref{app1:01}. For every $\epsilon>0$,  we take the same initial data $(\rho_0, u_0)$, as a result of which
%  \begin{equation*}
%  T_{h_0, \eps}(0) = 0.
%  \end{equation*} 
Collecting the estimates from Lemmas~\ref{T:1}, \ref{T:I1}--\ref{T:I3}, and \ref{T:I:4}--\ref{T:I:5}, choosing $\lambda$ sufficiently large, and dropping the extra penalization $D_1$, $D_2$, and $D_3$, we have
  \begin{equation*}
     T_{h_0, \eps}(t) \lesssim T_{h_0, \eps}(0) +C\int_0^t T_{h_0, \eps}(s)\,ds + |\log h_0|^{\theta} 
  \end{equation*}
for some $0<\theta<1$. A Gronwall inequality implies
  \begin{align*}
  \label{}
  T_{h_0, \eps}(t) \lesssim e^{CT}|\log h_0|^{\theta} 
  \end{align*}
for $t\le T$.
 Recalling the definition of $T_{h_0, \eps}$, in order to get the compactness of the solution $\rho_{\epsilon}$, we need to get rid of the weight function. Note that
  \begin{align*}
     \int_{\TT^{2d}}\ck_{h_0}(x-y)\chi(\delta\rho_{\epsilon})\,dxdy 
  &= 
  \int_{\TT^{2d}}\ck_{h_0}(x-y)\chi(\delta\rho_{\epsilon})\mathbf{1}_{w_{\epsilon, h}^{x}\le\eta}\mathbf{1}_{w_{\epsilon, h}^{y}\le\eta}\,dxdy
  \nn&\quad+
  \int_{\TT^{2d}}\ck_{h_0}(x-y)\chi(\delta\rho_{\epsilon})(1-\mathbf{1}_{w_{\epsilon, h}^{x}\le\eta}\mathbf{1}_{w_{\epsilon, h}^{y}\le\eta})\,dxdy
  \end{align*}
where $\eta>0$ is a big parameter depending on $h_0$ to be chosen later. For the first integral, in view of~\eqref{wei:2}, we have
  \begin{align*}   
  \int_{\TT^{2d}}&\ck_{h_0}(x-y)\chi(\delta\rho_{\epsilon})\mathbf{1}_{w_{\epsilon, h}^{x}\le\eta}\mathbf{1}_{w_{\epsilon, h}^{y}\le\eta}\,dxdy\nn&
  \lesssim
  \int_{\TT^{2d}}\ck_{h_0}(x-y)\rho_{\epsilon}^{1+l}(x)\mathbf{1}_{w_{\epsilon, h}^{x}\le\eta}\,dxdy
  +
  \int_{\TT^{2d}}\ck_{h_0}(x-y)\rho_{\epsilon}^{1+l}(y)\mathbf{1}_{w_{\epsilon, h}^{y}\le\eta}\,dxdy\nn&
  \lesssim
   \int_{\TT^{d}}\rho_{\epsilon}^{1+l}(x)\mathbf{1}_{w_{\epsilon, h}^{x}\le\eta}\,dx |\log h_0|
  \lesssim \frac{|\log h_0|}{|\log \eta|^{\alpha}}
  \end{align*}
for some $0<\alpha<1$. For the second integral, we use $T_{h_0, \eps}$ to get
  \begin{align*}   
  \int_{\TT^{2d}}\ck_{h_0}(x-y)\chi(\delta\rho_{\epsilon})(1-\mathbf{1}_{w_{\epsilon, h}^{x}\le\eta}\mathbf{1}_{w_{\epsilon, h}^{y}\le\eta})\,dxdy
  \le
  \frac{1}{\eta} T_{h_0, \eps}(t) \lesssim  \frac{1}{\eta} |\log h_0|^{\theta}.
  \end{align*}
By choosing $\eta=|\log h_0| $, we arrive at
  \begin{align*}   
  \int_{\TT^{2d}}\ck_{h_0}(x-y)\chi(\delta\rho_{\epsilon})\,dxdy \lesssim \frac{|\log h_0|}{\log |\log h_0|},
  \end{align*}
which implies the compactness of the solution $\rho_{\epsilon}$ by Lemma~\ref{cpt}.
  \end{proof}

%\begin{theorem}
%\label{exi:eq}
%Let $(\rho_{\mu}, u_{\mu})$ be a sequence of global weak solution of the compressible NS system~\eqref{app1:00}--\eqref{app1:01} with the initial condition~\eqref{eq:ini} for every $\mu$. We assume that this sequence of solutions $(\rho_{\mu}, u_{\mu})$ satisfy~\eqref{app1:ene:1}.
%Suppose that the pressure $P$ satisfies $(P1)$ and $(P2)$. Then the sequence of solutions is compact in $\mu$.
%\end{theorem}
%  \begin{proof}
%  The proof of this lemma is essentially covered in~\cite{BreJab18} except there is an extra pressure term $\mu\rho^{\gg}$. But from the proof of Lemma~\ref{T:2}, we see that the terms involving $\mu\rho^{\gg}$ have the right sign. Therefore, we can pass to the limit $\mu\rightarrow0$ to conclude the proof.
%  \end{proof}
%%%%%%%%%%%%%%%%%%%%%%%%%%%%%%%%%%%%%%%%%%%%%%%%%%%%%%%%%%%%  
%%%%%%%%%%%%%%%%%%%%%%%%%%%%%%%%%%%%%%%%%%%%%%%%%%%%%%%%%%
\appendix
%%%%%%%%%%%%%%%%%%%%%%%%%%%%%%%%%%%%%%%%%%%%%%%%%%%%%%%%%
\section{Proof of Theorems \ref{thmuto0} and \ref{thapprox1}}
%%%%%%%%%%%%%%%%%%%%%%%%%%%%%%%%%%%%%%%%%%%%%%%%%%%%%%%%%
%%%%%%%%%%%%%%%%%%%%%%%%%%%%%%%%%%%%%%%%%%%%%%%%%%%%%%%%%%
\subsection{Proof of Th. \ref{thmuto0}}
%%%%%%%%%%%%%%%%%%%%%%%%%%%%%%%%%%%%%%%%%%%%%%%%%%%%%%%
The proof is performed by taking several consecutive limits, first $\eta_1\to 0$, then $\eta_2\to 0$ till the last limit $\eta_m\to 0$. The generic step is hence, once we already have $\eta_1=\dots=\eta_{i}=0$, to pass to the limit $\eta_{i+1}\to 0$. For this reason, we introduce the notation $\rho_{\eta,i},\; u_{\eta,i}$ which is obtained by taking the first $i-1$ weak limits $\eta_1\to 0$, $\eta_{i-1}\to 0$. More precisely, after extracting subsequences, we have that $\rho_{\eta,1}=\rho_{\eta}$, $u_{\eta,1}=u_\eta$ and
\[
\rho_{\eta,i+1}=w-\lim_{\eta_i\to 0} \rho_{\eta,i},\quad u_{\eta,i+1}=w-\lim_{\eta_i\to 0} u_{\eta,i}.
\]
The final solution that we will obtain is simply $\rho=\rho_{\eta,m+1},\;u=u_{\eta,m+1}$ which is independent of all $\eta_i$.
Assuming that $\rho_{\eta,i}$ is a weak solution to the system
\begin{equation}
\begin{split}
&\partial_t \rho_{\eta_i}+\div(\rho_{\eta_i}\,u_{\eta_i})=0,\\
  &\partial_t (\rho_{\eta_i}\,u_{\eta_i})+\div(\rho_{\eta_i}\,u_{\eta_i}\otimes u_{\eta_i})-\Delta u_{\eta_i}+\nabla(\eta_i\,\rho_{\eta_i}^{\ggg{i}}+\ldots+\eta_{m}\,\rho_{\eta_i}^{\ggg{m}} +P(t,x,\rho_{\eta_i}))=0,
  \end{split}\label{intermeta}
\end{equation}
then we have to show that $\rho_{\eta,i+1}$ solves the same system with $\eta_i=0$.

\smallskip

{\em Step~1: Basic energy inequality for $\rho_\eta,\;u_\eta$.} We observe that $\rho_{\eta},\; u_\eta$ solves \eqref{intermeta} directly from Theorem~\ref{thapprox1}. However the a priori estimates provided by Theorem~\ref{thapprox1} are not uniform in $\eta$ so that our first step consists in deriving such estimate starting from the energy inequality \eqref{energyineq0}.

The first point is to pass to the limit as $\eps\to 0$ in \eqref{energyineq0}. Of course the left-hand side is convex in $\rho_{\eps,\eta},\;u_{\eps,\eta}$ so it handled in the usual manner. We have that $\div u_{\eps,\eta}$ is uniformly bounded in $L^2_{t,x}$ so $\div u_{\eps,\eta}\to \div u_{\eta}$ in $w-L^2_{t,x}$.

On the other hand by \eqref{P:1}-\eqref{P:2}, we have that $|P(t,x,\rho_{\eps,\eta})|\leq R+\Theta_1+C\,\rho_{\eps,\eta}^p$ with $p\leq \gamma+\frac{2\,\gamma}{d}-1$ and $R+\Theta_1\in L^q_{t,x}$ with $q>2$. By Theorem~\ref{thapprox1}, we have that $\rho_{\eps,\eta}\in L^{p_{art}}_{t,x}$ uniformly in $\eps$ for any $p_{art}\leq \gg+2\,\gg/d-1$. Observe that $2\,(\gamma+\frac{2\,\gamma}{d}-1)< 2\,\gamma+\frac{4\,\gamma}{d}-1\leq \gg+2\,\gg/d-1$ since $2\,\gamma\leq \gg$. This is the first place where the assumption $2\,\gamma\leq \gg$ is critical.

Hence $P(t,x,\rho_{\eps,\eta})$ is uniformly bounded in $\eps$ in $L^q_{t,x}$ for some $q>2$. By the compactness of $\rho_{\eps,\eta}$ provided by Theorem~\ref{thepstomu}, we obtain that $P(t,x,\rho_{\eps,\eta})\to P(t,x,\rho_\eta)$ strongly in $L^2_{t,x}$.

Therefore this provides a solution $\rho_\eta,\;u_\eta$ to the system \eqref{app1:00}-\eqref{app1:01} with, {\em for a fixed $\eta$}, the bounds $\rho_\eta\in L^\infty_t L^{\gg}_x$, $\rho_\eta\in L^p_{t,x}$ for any $p\leq \gg+2\,\gg/d-1$, $u_\eta\in L^2_t H^1_x$, and the basic energy inequality
\begin{equation}
\begin{split}
  & \int_{\TT^d} (\eta_1\,\frac{\rho_{\eta}^{\ggg{1}}(t,x)}{\ggg{1}-1}+\ldots +\eta_m\,\frac{\rho_{\eta}^{\ggg{m}}(t,x)}{\ggg{m}-1} +\rho_{\eta} (t,x)\,|u_{\eta} (t,x)|^2)\,dx\\
  &\qquad+\int_0^t\int_{\TT^d} |\nabla u_{\eta}(s,x)|^2\,dx\,ds \leq \int_0^t\int_{\TT^d} \div u_{\eta}\, P\, dx\,ds\\
  &\qquad +\int_{\TT^d} \left(\eta_1\,\frac{(\rho_{\eps,\eta}^0)^{\ggg{1}}(t,x)}{\ggg{1}-1}+\ldots +\eta_m\,\frac{(\rho_{\eps,\eta}^0)^{\ggg{m}}(t,x)}{\ggg{m}-1} +\rho_{\eps,\eta}^0 (t,x)\,|u_{\eps,\eta}^0 (t,x)|^2\right)\,dx. 
  \end{split}\label{energyineq1}
\end{equation}

\smallskip

{\em Step~2: Modified energy inequality.} Our next step is to work with \eqref{energyineq1} to obtain a form that is more suitable to the derivation of a priori estimates.

We recall that $\mathcal{E}_0=\rho_\eta\,(|u_\eta|^2/2+e_0(\rho_\eta))$ with $e_0(t,x,\rho)=\int_{\rho_{ref}}^\rho P_0(t,x,s)/s^2\,ds$.

We have that
\[\begin{split}
&\frac{d}{dt}\int_{\TT^d} \rho_\eta\,e_0(t,x,\rho_\eta)\,dx={\int_{\TT^d} (\rho_\eta\partial_t e_0(\rho_\eta)}+\rho_\eta\,u_\eta\cdot\nabla_x e_0(\rho_\eta))\,dx\\
&\qquad+\int_{\TT^d} \div u_\eta(\rho_{\eta}\,e_0(\rho_\eta)-\rho_\eta\,\partial_{\rho}(\rho_\eta e_0(\rho_\eta)))\,dx. 
\end{split}
\]
From the definition of $e_0$, we get that
\[\begin{split}
&\frac{d}{dt}\int_{\TT^d} \rho_\eta\,e_0(t,x,\rho_\eta)\,dx=\int_{\TT^d} (\rho_\eta\partial_t e_0(\rho_\eta)+\rho_\eta\,u_\eta\cdot\nabla_x e_0(\rho_\eta))\,dx\\
&\qquad-\int_{\TT^d} \div u_\eta\,P_0(\rho_\eta)\,dx. 
\end{split}
\]
Note that from \eqref{P:2}, \eqref{P:3}, \eqref{P:4}, we have that for a fixed $\eta$, $P_0\in L^2$ while $\rho_\eta \partial_t e_0(\rho_\eta)\in L^1_{t,x}$ and $\rho_\eta\,\nabla_x e_0(\rho_\eta)\in L^2_t\,L^{2d/(d+2)}_x$ so that  $\rho_\eta\,u_\eta\cdot\nabla_x e_0(\rho_\eta)\in L^1_{t,x}$ as well. Therefore all terms make sense and this is again due to the assumption $\gg\geq 2\,\gamma$.

Adding this to \eqref{energyineq1} yields the more precise energy inequality
    \begin{equation}
    \begin{split}
      &\int_{\TT^d} (\mathcal{E}_0(\rho_\eta,u_\eta)+\eta_1\,\frac{\rho_\eta^{\ggg{1}}}{\ggg{1}-1} +\ldots+\eta_m\,\frac{\rho_\eta^{\ggg{m}}}{\ggg{m}-1})\,dx+\int_0^t\int_{\TT^d} |\nabla u_\eta(s,x)|^2\,dx\,ds\\
      &\qquad\leq \,\int_0^t\int_{\TT^d} \div_x u_\eta(s,x)\,(P(s,x,\rho_\eta(s,x))-P_0(s,x,\rho_\eta(s,x)))\,ds\,dx\\
      &\qquad +\int_0^t\int_{\TT^d} (\rho_\eta\,\partial_t e_0(\rho_\eta)+\rho_\eta\,u_\eta\cdot\nabla_x e_0(\rho_\eta))\,dx\,ds\\
      &\qquad +\int_{\TT^d} \left(\eta_1\,\frac{(\rho_{\eps,\eta}^0)^{\ggg{1}}(t,x)}{\ggg{1}-1}+\ldots +\eta_m\,\frac{(\rho_{\eps,\eta}^0)^{\ggg{m}}(t,x)}{\ggg{m}-1} +\mathcal{E}_0(\rho_\eta^0,u_\eta^0) \right)\,dx,
      \end{split}\label{energyineq2}
    \end{equation}
    which we will use to obtain our a priori estimates.

    \smallskip
    
    {\em Step~3: A priori estimates on $\rho_\eta,\;u_\eta$.} From \eqref{energyineq2}, we first observe that from \eqref{P:1} since $\bar\gamma\leq \gamma/2$,
    \[
    \begin{split}
      &\int_0^t\int_{\TT^d} \div_x u_\eta(s,x)\,(P(s,x,\rho_\eta(s,x))-P_0(s,x,\rho_\eta(s,x)))\,ds\,dx\leq C\\
      &\quad+\frac{1}{4}\,\int_0^t\int_{\TT^d} |\nabla u_\eta(s,x)|^2\,dx\,ds
      +C\,\int_0^t\int_{\TT^d} |\rho_\eta(s,x)|^\gamma\,dx\,ds.
\end{split}
    \]
    Similarly by \eqref{P:3}-\eqref{P:4}, we can bound
    \[
    \begin{split}
      &\int_0^t\int_{\TT^d} (\partial_t e_0(\rho_\eta)+\rho_\eta\,u_\eta\cdot\nabla_x e_0(\rho_\eta))\,dx\,ds\\
      &\quad\leq C+\frac{1}{4}\,\int_0^t\int_{\TT^d} |\nabla u_\eta(s,x)|^2\,dx\,ds+C\,\int_0^t\int_{\TT^d} |\rho_\eta(s,x)|^\gamma\,dx\,ds.
      \end{split}
    \]
    By \eqref{P:2}, we hence obtain that
\[
  \begin{split}
      &\int_{\TT^d} (\frac{\rho_\eta^{\gamma}}{C}+\eta_1\,\frac{\rho_\eta^{\ggg{1}}}{\ggg{1}-1} +\ldots+\eta_m\,\frac{\rho_\eta^{\ggg{m}}}{\ggg{m}-1} +\rho_\eta\,\frac{|u_\eta|^2}{2})\,dx+\frac{1}{2}\, \int_0^t\int_{\TT^d} |\nabla u_\eta(s,x)|^2\,dx\,ds\\
      &\qquad\leq C+C\,\int_0^t\int_{\TT^d} |\rho_\eta(s,x)|^\gamma\,dx\,ds.
      \end{split}
  \]
  By Gronwall's lemma, we deduce the first main estimate on $\rho_\eta$ and $u_\eta$, for some constant $C$ independent of $\eta$
  \begin{equation}
    \begin{split}
      &\int_{\TT^d} (\frac{\rho_\eta^{\gamma}}{C}+\eta_1\,\frac{\rho_\eta^{\ggg{1}}}{\ggg{1}-1} +\ldots+\eta_m\,\frac{\rho_\eta^{\ggg{m}}}{\ggg{m}-1} +\rho_\eta\,\frac{|u_\eta|^2}{2})\,dx\leq C\,e^{C\,t},\\
      & \int_0^t\int_{\TT^d} |\nabla u_\eta(s,x)|^2\,dx\,ds \leq C\,e^{C\,t}.
\end{split}\label{apriorieta}
  \end{equation}
  Those estimates are convex in $\rho_\eta$ and $u_\eta$. Hence by the definition of the $\rho_{\eta,i},\;u_{\eta,i}$, we trivially have as well that
   \begin{equation}
    \begin{split}
      &\int_{\TT^d} (\frac{\rho_{\eta,i}^{\gamma}}{C}+\eta_i\,\frac{\rho_{\eta,i}^{\ggg{i}}}{\ggg{i}-1} +\ldots+\eta_m\,\frac{\rho_{\eta,i}^{\ggg{m}}}{\ggg{m}-1} +\rho_{\eta,i}\,\frac{|u_{\eta,i}|^2}{2})\,dx\leq C\,e^{C\,t},\\
      & \int_0^t\int_{\TT^d} |\nabla u_{\eta,i}(s,x)|^2\,dx\,ds \leq C\,e^{C\,t}.
\end{split}\label{apriorietai}
  \end{equation}
When considering the limit $\eta_i\to 0$ on $\rho_{\eta,i},\;u_{\eta,i}$, we have that $\eta_{i+1},\ldots,\eta_m>0$. 
We hence have all the bounds needed to apply Lemma \ref{bogovski} with $S=P-P_0$, $\gamma_0=\ggg{i+1}$ and $1/p=1+1/\ggg{i+1}-2/d$ or $\ggg{i+1}/p^*=2\ggg{i+1}/d-1$. This lets us obtain our last a priori estimate
\begin{equation}
\sup_{\eta_i} \int_0^T\int_{\TT^d} \rho_{\eta,i}^{q}(t,x)\,dx\,dt<\infty,\quad \forall q<\ggg{i+1}+2\,\ggg{i+1}/d-1. \label{bogovskietai}
  \end{equation}

\smallskip

{\em Step~4: Passing to the limit.}
  Equipped with those bounds, we have the {\em weakly converging subsequences} as $\eta_i\to 0$: $\rho_{\eta,i}\to\rho_{\eta,i+1}$ in $w-L^\infty_t L^{\ggg{i+1}}_x$ and $w-L^q_{t,x}$ for any $q<\ggg{i+1}+2\,\ggg{i+1}/d-1$, and $u_{\eta,i}\to u_{\eta,i+1}$ in $w-L^2_t H^1_x$.

  As usual, this is also enough to show the weak limits $\rho_{\eta,i}\,u_{\eta,i}\to \rho_{\eta,i+1}\,u_{\eta,i+1}$ and $\rho_{\eta,i}\,u_{\eta,i}\otimes u_{\eta,i}\to \rho_{\eta,i+1}\,u_{\eta,i+1}\otimes u_{\eta,i+1}$. Those bounds also provides equi-integrability on $P(t,x,\rho_{\eta,i})$ by the upper bounds following from \eqref{P:1}-\eqref{P:2}. Equi-integrability also holds on
  \[
\eta_i\,\frac{\rho_{\eta,i}^{\ggg{i}}}{\ggg{i}-1} +\ldots+\eta_m\,\frac{\rho_{\eta,i}^{\ggg{m}}}{\ggg{m}-1},
  \]
  since $\ggg{i}< \ggg{i+1}+2\,\ggg{i+1}/d-1$ which is the key relation between the coefficients~$\ggg{i}$.
  
  The main remaining question is to prove the compactness of $\rho_{\eta,i}$ in $L^1_{t,x}$. This is in general the difficult question for compressible Navier-Stokes but, fortunately in this case, we may directly apply the result of \cite{BreJab18}.

  Specifically we invoke Th. 5.1, case (ii) in that article (page 613). Our sequence $\rho_{\eta,i},\;u_{\eta,i}$ solves the continuity equation (denoted (5.1) in the article). The momentum equation implies that $u_{\eta,i}$ solves equation (5.2) in the article with constant viscosity and $R_k=0$. Our {\em a priori} estimates directly ensures the bounds (5.3)-(5.7) that are required by Th. 5.1 in \cite{BreJab18}. Finally the assumption on the pressure law for this theorem is identical to our assumptions \eqref{P:5}-\eqref{P:las}.

  We hence deduce the compactness of $\rho_{\eta,i}$ and hence the convergence of $P(t,x,\rho_{\eta,i})+\eta_i\,\rho_{\eta,i}^{\ggg{i}} +\ldots+\eta_m\,\rho_{\eta,i}^{\ggg{m}}$ to $P(t,x,\rho_{\eta,i+1})+\eta_{i+1}\,\rho_{\eta,i+1}^{\ggg{i+1}} +\ldots+\eta_m\,\rho_{\eta,i+1}^{\ggg{m}}$. This implies that $\rho_{\eta,i+1},\;u_{\eta,i+1}$ solves \eqref{intermeta} with $\eta_i=0$ and finally that $\rho,\;u$ is indeed a global solution to the system \eqref{eq:00}-\eqref{eq:01} as claimed with the corresponding estimates for $i=m+1$ following from \eqref{apriorietai} and \eqref{bogovskietai}. Finally the energy inequality is directly obtained from \eqref{energyineq2} by taking the successive limits.
%
  %%%%%%%%%%%%%%%%%%%%%%%%%%%%%%%%%%%%%%%%%%%%
  \subsection{Proof of Th. \ref{thapprox1}}
  %%%%%%%%%%%%%%%%%%%%%%%%%%%%%%%%%%%%%%%%%%%%%

  We can obtain solutions to \eqref{app:00}-\eqref{app:01} through a fixed point theorem. Given any $S\in L^2([0,\ T]\times\TT^d))$, we define $N_S,\;U_S$ as a global weak solution to
  \begin{equation}
\begin{split}
    &\partial_t N_S+\div (N_S\,U_S)=0,\quad N_S(t=0)=\rho_\eps^0,\\
& \partial_t (N_S\,U_S)+\div (N_S\,U_S\otimes U_S)-\Delta U_S+\nabla(P_\eta(N_S)+S)=0,\quad U_\rho(t=0)=u_\eps^0.\\
\end{split}\label{NSUS}
\end{equation}
  System \eqref{NSUS} is in fact the classical compressible Navier-Stokes system with barotropic pressure law $P_\eta(\rho)=\eta_1\,\rho^{\ggg{1}}+\ldots+\eta_m\,\rho^{\ggg{m}}$ and a source term. Provided that $\gg+2\gg/d-1>2$ with $\gg=\ggg{1}=\max_i \ggg{i}$, which we assumed, existence of global solution to this system is guaranteed by \cite{Li} and moreover such solutions satisfy the following energy estimate for some constant $C$
  \begin{equation}
    \begin{split}
      &\sup_{t\in[0,\ T]}\int_{\TT^d} ((N_S(t,x))^{\gg}+N_S\,|U_S|^2)\,dx+\int_0^T\int_{\TT^d} |\nabla U_S|^2\,dx\\
     & \quad \leq C\,\int_{\TT^d} ((\rho_\eps^0(t,x))^{\gg}+\rho_\eps^0\,|u_\eps^0|^2)\,dx+C\,\|S\|_{L^2_{t,x}}^2.\label{energyepsdep}
    \end{split}
  \end{equation}
and the following energy inequality
\begin{equation}
\begin{split}
  & \int_{\TT^d} \left(\eta_1\,\frac{N_S^{\ggg{1}}(t,x)}{\ggg{1}-1}+\ldots +\eta_m\,\frac{N_S^{\ggg{m}}(t,x)}{\ggg{m}-1} +N_S(t,x)\,|U_{S} (t,x)|^2\right)\,dx\\
  &\qquad+\int_0^t\int_{\TT^d} |\nabla U_S(s,x)|^2\,dx\,ds \leq \int_0^t\int_{\TT^d} \div U_S\cdot S\, dx\,ds\\
  &\qquad +\int_{\TT^d} \left(\eta_1\,\frac{(\rho_{\eps,\eta}^0)^{\ggg{1}}(t,x)}{\ggg{1}-1}+\ldots +\eta_m\,\frac{(\rho_{\eps,\eta}^0)^{\ggg{m}}(t,x)}{\ggg{m}-1} +\rho_{\eps,\eta}^0 (t,x)\,|u_{\eps,\eta}^0 (t,x)|^2\right)\,dx. 
  \end{split}\label{energyineqNU}
\end{equation}
 We are now using Lemma \ref{bogovski} with $P_0=P_\eta(N_S)=\eta_1\,N_S^{\ggg{1}}+\ldots+\eta_m\,N_s^{\ggg{m}}$.

One has that $N_S\in L^\infty_t L^{\gg}_x$ solves \eqref{conservative} with $u_\eps\in L^2_t H^1_x$. Since $\gg>2$ then $S\in L^2_{t,x}\subset L^1_t L^{\gg^*}_x$ trivially. On the other hand
\[
\sup_\eps\|\Delta U_S\|_{L^2_t H^{-1}}<\infty,
\]
and using Sobolev embeddings $U_S\in L^2_t L^q_x$ with $1/q=1/2-1/d$ so that
\[
\sup_\eps\|N_S\,U_S\otimes U_S\|_{ L^1_t L^p_x}<\infty,\quad \frac{1}{p}=\frac{1}{\gg}+\frac{2}{q}=\frac{1}{\gg}+1-\frac{2}{d}.
\]
Similarly
\[
\sup_\eps\|N_S\,U_S\|_{L^2_t L^r_x}<\infty,\quad \frac{1}{r}=\frac{1}{\gg}+\frac{1}{q}=\frac{1}{\gg}+\frac{1}{2}-\frac{1}{d},
\]
and one notes that $2pd/(2d+2p-pd)=r$ or $1/r=1/p+1/d-1/2$.

Using the bound on the kinetic energy $\int N_S\,|U_S|^2\,dx$, we also have that
\[
\int_{\TT^d} N_S^s\,|U_S|^s\,dx\leq \left(\int_{\TT^d} N_S\,|U_S|^2\right)^{s/2}\,\left(\int_{\TT^d} N_S^{s/(2-s)}\right)^{1-s/2}.
\]
Note that $s/(2-s)=\gg$ iff $s=2\gg/(1+\gg)$, implying that
\[
\sup_\eps \|N_S\,U_S\|_{L^\infty_t L^s_x}<\infty,\quad s=2\gg/(1+\gg),
\]
with in particular $s=2pd/(d+2p)\geq pd/(p+d)$.

We hence deduce that for $\theta<\gg/p^*$ or $\theta<2\gg/d-1$
\[
\sup_\eps \int_0^T\int_{\TT^d} N_S^\theta\,P_0(N_S)\,dt\,dx<\infty,
\]
or, in other words, Lemma \ref{bogovski} implies that
\begin{equation}
  \int_0^T\int_{\TT^d} N_S^q(t,x)\,dx\,dt\leq C\,\int_{\TT^d} ((\rho_\eps^0(t,x))^{\gg}+\rho_\eps^0\,|u_\eps^0|^2)\,dx+C\,\|S\|_{L^2_{t,x}}^2.
  \label{extraintegNS}
  \end{equation}
  This leads to defining the following operator
  \[
F: S\longrightarrow F(S)(t,x)=\mathcal{L}_\eps\star P(N_S).
\]
From the definition, we have that
\[
\|F(S)\|_{L^2_{t,x}}^2\leq C\,\eps^{-d}\,\|P(N_S)\|_{L^2_t\,L^1_x}^2\leq C\,\eps^{-d}\,\|R\|_{L^1}^2+C\,\eps^{-d}\, \|N_S^p\|_{L^{2}_t\,L^{1}_x}^2,
\]
for some $p<\gamma+2\,\gamma/d-1$, by using assumptions \eqref{P:1}-\eqref{P:2} on $P$. Since $R\in L^2_{t,x}$,, we deduce that
\[
\|F(S)\|_{L^2_{t,x}}^2\leq C\,\eps^{-d}+C\,\eps^{-d}\,\|N_S\|_{L^{2p}_{t,x}}^{2p}.%\leq C\,\eps^{-d}+C\,\eps^{-d}\,T^4\,\|N_S\|_{L^{\infty}_t\,L^{\gg}_x}^{2\,\gamma^2/\gg}.
\]
Finally, $\gg+2\,\gg/d-1\ge 2\gamma+4\,\gamma/d-1>2p$ since $\gg\geq 2\,\gamma$, we have by \eqref{extraintegNS}
\[
\|F(S)\|_{L^2_{t,x}}^2\leq C\,\eps^{-d}+C\,\eps^{-d}\,T^4\,\|S\|_{L^2_{t,x}}^{2\,\theta}, 
\]
for some exponent $\theta<1$ through the uniform in $\eps$ bound on 
\[
\sup_\eps \int_{\TT^d} ((\rho_\eps^0(t,x))^{\gg}+\rho_\eps^0\,|u_\eps^0|^2)<\infty.
\]
As $\theta<1$, there exists a ball $B\subset L^2_{t,x}$ with large enough radius such that $F(B)\subset B$.

Moreover $F(S)\in L^2_t H^1_x$ for any $S\in B$ thanks to the convolution in $x$ giving compactness in the space variable. To prove the time compactness, one could observe that the argument in \cite{Li} or the quantitative estimates from \cite{BreJab18} provide full compactness on the density provided that the source term is compact in space ({\em i.e.} without time compactness being required).

However, since it is possible to obtain the time compactness in a straightforward manner and for the sake of completeness, we present the argument here. We need to introduce various regularization and truncations. First of all \eqref{P:1} implies that $P/(1+s^{p})$ is in $L^2_{t,x}$ uniformly in $s$. Hence we can choose $P_\eta(t,x,s)$ a regularization of $P$ in $t$ and $x$ with for example
\[\begin{split}
&|\partial_t P_\eta(t,x,s)|+|\nabla_x P_\eta(t,x,s)|+|\partial_s^2 P_\eta(t,x,s)|\leq \frac{C}{\eta}\,(1+s^{p}),\\ &\|P(.,.,s)-P_\eta(.,.,s)\|_{L^2_{t,x}}\leq f(\eta)\,(1+s^{p}),
\end{split}
\]
for some continuous function $f$ with $f(0)=0$.

By \eqref{P:1} again and since \eqref{energyepsdep} shows that $N_S\in L^\infty_t L^{\gg}_x$ with $\gg>\gamma$, we may immediately deduce from the last point that there exists $\tilde f$ continuous with $\tilde f(0)=0$ such that for any $S\in B$
\begin{equation}
\|P(.,.,N_S)-P_\eta(.,.,N_S)\|_{L^2_{t,x}}\leq \tilde f(\eta).\label{Peta1}
\end{equation}
Now choosing any standard convolution kernel $L$, we may write
\[\begin{split}
&\mathcal{L}_\eps\star P_\eta(N_S)(t,x)=\int_{\TT^d} \mathcal{L}_\eps(x-y)\,P_\eta(t,y,N_S(t,y))\,dy\\
&\ =\int_{\TT^d} \mathcal{L}_\eps(x-y)\,L_{\sqrt{\eta}}(y-z)\,P_\eta(t,z,N_S(t,y))\,dy\,dz\\
&\qquad+\int_{\TT^d} \mathcal{L}_\eps(x-y)\,L_{\sqrt{\eta}}(y-z)\,(P_\eta(t,y,N_S(t,y))- P_\eta(t,z,N_S(t,y)))\,dy\,dz.
\end{split}
\]
Therefore
\begin{equation}
\left\|\mathcal{L}_\eps\star P_\eta(N_S)-\int_{\TT^d} \mathcal{L}_\eps(x-y)\,L_{\sqrt{\eta}}(y-z)\,P_\eta(t,z,N_S(t,y))\,dy\,dz \right\|_{L^1_{t,x}}\leq C\,\sqrt{\eta}.\label{Peta2}
\end{equation}
Since $N_S$ solves the continuity equation \eqref{conservative} and $U_S\in L^2_t H^1_x$, we have by Th. \ref{ren:sol} that for any fixed $z$
\[\begin{split}
&\partial_t (P_\eta(t,z,N_S(t,x)))=\partial_t P_\eta(t,z,N_S(t,x))+\div_x(P_\eta(t,z,N_S(t,x))\,U_S(t,x))\\
&\qquad= (P_\eta(t,z,N_S(t,x))-N_S\,\partial_s P_\eta(t,z,N_S(t,x)))\,\div U_S.
\end{split}
\]
From this, we obtain that
\[
\begin{split}
  & \frac{d}{dt}\int_{\TT^d} \mathcal{L}_\eps(x-y)\,L_{\sqrt{\eta}}(y-z)\,P_\eta(t,z,N_S(t,y))\,dy\,dz\\
  &\ =\int_{\TT^d} \mathcal{L}_\eps(x-y)\,L_{\sqrt{\eta}}(y-z)\,\partial_t P_\eta(t,z,N_S(t,y))\,dy\,dz\\
  &\quad+\int_{\TT^d} \mathcal{L}_\eps(x-y)\,\nabla_y L_{\sqrt{\eta}}(y-z)\,P_\eta(t,z,N_S(t,y))\,U_S(t,y)\,dy\,dz\\
  &\quad+\int_{\TT^d} \mathcal{L}_\eps(x-y)\,L_{\sqrt{\eta}}(y-z)\,(P_\eta(t,z,N_S(t,x))-N_S\,\partial_s P_\eta(t,z,N_S(t,x)))\,\div U_S\,dy\,dz.
\end{split}  
\]
Bounding directly each term, this implies that
\begin{equation}
\left|\frac{d}{dt}\int_{\TT^d} \mathcal{L}_\eps(x-y)\,L_{\sqrt{\eta}}(y-z)\,P_\eta(t,z,N_S(t,y))\,dy\,dz\right|\leq C_\eps\,\eta^{-k}\label{Peta3},
\end{equation}
for some exponent $k>0$.

We may now combine \eqref{Peta1}, \eqref{Peta2} and \eqref{Peta3} to obtain the compactness in time of $\mathcal{L}_\eps\star P(N_S)$ and hence the compactness in $L^2_{t,x}$ of $F(B)$.
  By the Schauder fixed point, $F$ has a fixed point $S$ in $B\subset L^2_{t,x}$. We now simply choose $\rho_\eps=N_S$ and $u_\eps=U_S$ and since $N_S,\;U_S$ solve \eqref{NSUS} with $S=F(S)=\mathcal{L}_\eps\star P(N_S)=\mathcal{L}_\eps\star P(\rho_\eps)$, we obtain a solution to \eqref{app:00}-\eqref{app:01}. The energy bound \eqref{energyepsdep} provides all uniform in $\eps$ bounds on $\rho_\eps$ while the energy inequality \eqref{energyineqNU} of course leads to the corresponding inequality in the theorem. Estimate \eqref{extraintegNS} provides the extra-integrability on $\rho_{\eps,\eta}$.
%%%%%%%%%%%%%%%%%%%%%%%%%%%%%%%%%
\section{}
%%%%%%%%%%%%%%%%%%%%%%%%%%%%%%%%%
\subsection{Notations}
\label{notation}
Because we use functions at various points and differences of functions, we introduce specific notations. First, the symbol $f^x$ stands for a function of $x$, i.e., $f^x=f(x)$. Next, we also denote 
  \begin{equation*}   
  \delta f(x, \xi)= f(x)- f(x-\xi)
  \end{equation*}
 and 
 \begin{equation*}
     \bar f(x, \xi)= f(x) + f(x-\xi).
  \end{equation*} 
If the argument is not mentioned explicitly or only the $x$ variable is mentioned in the above notation, then we set $\xi=x-y$, i.e.,
  \begin{equation*}
     \delta f = \delta f(x) = \delta f(x, x-y) = f(x)- f(y)
  \end{equation*}
and 
  \begin{equation*}
     \overline f = \overline f(x) = \overline f(x, x-y) = f(x) + f(y).
  \end{equation*}
 We denote the maximum operator by
  \begin{equation*}
     M\,f(x) = \sup_{r>0} \frac{1}{|B_r|}\int_{B_r} |f(x)|\,dx.
  \end{equation*}
Recall that 
  \begin{equation*}
     \Vert M\,f \Vert_{L^{p}} \lesssim  \Vert f \Vert_{L^{p}}
  \end{equation*}
for $p>1$ and where the relation $f\lesssim g$ stands for that
$f \le Cg$ for some constant $C>0$.

We use bracket to stand for the commutator
  \begin{align*}
     [f, T]g = f\,T\,g-T(f\,g)
  \end{align*}
where $f$ and $g$ are smooth functions and $T$ is an operator. 
%%%%%%%%%%%%%%%%%%%%%%%%%%%%%%%%%%%%%%%
\subsection{Renormalized solutions}
%%%%%%%%%%%%%%%%%%%%%%%%%%%%%%%%%%%%%%%%
We rely on the concept of renormalized solution to justify several a priori formal calculations in the article. For this reason, we recall here the main definitions.  Given our system, we naturally focus on the conservative transport equation
  \begin{equation}
  \label{conservative}
  \partial_{t}\rho+\div(\rho u) =0.
  \end{equation}
  Given a weak solution $\rho$ to the above, it is not a priori possible to calculate non-linear functions of $\rho$ which is precisely what we need here. Hence one introduces the notion of renormalized solutions
\begin{definition}
A weak solution $\rho\in L^p_{t,x}$ to \eqref{conservative} with $u\in L^q_{t,x}$ for $1/p+1/q=1$ is a renormalized solution iff for any $\chi\in C^{1}(\mathbb{R})$ with $|\chi'(s)|\le C(1+|s|^{p-1})$, one has that
  \begin{equation}
  \label{renormalized}
  \partial_{t}\chi(\rho) + \div(\chi(\rho) u) = (\chi(\rho)-\rho\chi'(\rho))\div u
  \end{equation}
in the sense of distributions.
     \end{definition}
Renormalized solutions were first introduced in the famous \cite{DL} which in particular proved that if $u$ belongs to the right Sobolev space then all weak solutions are renormalized.
\begin{theorem}
\label{ren:sol}
Assume that $\rho\in L^p_{t, x}$ is a solution to \eqref{conservative} in the sense of distributions. Suppose that
$u\in L^q_t W^{1,q}_x$ with $1/p+1/q=1$, then $\rho$ is a renormalized solution to \eqref{conservative}.
\end{theorem}
For linear equations, {\em i.e.} when $u$ is given in \eqref{conservative}, then the theory of renormalized solutions immediately provides many key properties such as the compactness for a sequence or the uniqueness of a solution. For example, assume there are two solutions $\rho_1$ and $\rho_2$ to \eqref{conservative} for the same $u$. Applying Theorem~\ref{ren:sol} to the function $\rho = \rho_1-\rho_2$ with $\chi(x) = |x|$ and integrating in time and space gives
  \[
\frac{d}{dt}\int_{\TT^{d}} \chi(\rho) \,dx = 0
  \]
which immediately implies that $\rho_1=\rho_2$.

Observe however that in general and unless $\div u \in L^\infty$, it is not possible to have a general existence result for \eqref{conservative} for a {\em given} $u\in L^q_t W^{1,q}_x$.  A solution with only $\div u \in L^2$ may for example concentrate, by forming Dirac masses.

Following \cite{DL} and the BV extensions in \cite{Bo} for the kinetic case and the seminal \cite{Am} in the general case, the theory of renormalized solutions is now an extensive field for which we refer for example to the reviews \cite{AmCr, DeL}.

In the context of compressible Fluid Mechanics, renormalized solutions have been critical to obtaining the compactness of the density since the first breakthrough in \cite{Li} and they also form the basis of the extension introduced in \cite{Febook,Fe0}. We in particular cite the straightforward compactness result from \cite{DL}
\begin{theorem}
\label{ren:compact}
Consider a sequence $u_n$ converging strongly to $u$ in $L^1([0,\ T],\ L^q(\TT^d))$ s.t $\div u_n$ converges to $\div u$ in $L^1([0,\ T],\ L^q(\TT^d))$ as well. Consider any sequence $\rho_n$ such that $\rho_n,\;u_n$ satisfies Eq. \eqref{conservative} and $\rho_n$ uniformly bounded in $L^\infty([0,\ T],\ L^p(\TT^d))$ with $1/p+1/q<1$. Assume finally that $u\in L^1([0,\ T],\ W^{1,p^*}(\TT^d)$ with $1/p+1/p^*=1$. Then the sequence $\rho_n$ is compact in $L^1([0,\ T]\times\TT^d)$.
\end{theorem}

Th. \ref{ren:compact} can be deduced from Th. \ref{ren:sol}. The proof of Th. \ref{ren:sol} itself relies on a so-called qualitative commutator estimate and in several respects, the method introduced in \cite{BreJab18} consists in quantifying this commutator estimate.

\bigskip

\noindent {\bf Acknowledgments.}  The first author is partially supported by the SingFlows project, grant ANR-18-CE40-0027. The second author is  partially supported by NSF DMS Grant 161453, 1908739, and NSF Grant RNMS (Ki-Net) 1107444.

%%%%%%%%%%%%%%%%%%%%%%%%%%%%%%%%%%%%%%
%\bibliographystyle{alpha}
%\bibliography{VladBib}

\end{document}